\documentclass{amsart}

\usepackage{amsmath,amssymb,amsthm,mathtools}
\usepackage{aliascnt}
\usepackage{amsrefs}
\usepackage{enumitem}
\usepackage{microtype}
\usepackage[dvipsnames]{xcolor}
\usepackage{tikz-cd}
\usepackage[colorlinks=true,linkcolor=blue!55!black,citecolor=green!45!black,urlcolor=red!60!black]{hyperref}
\numberwithin{equation}{section}

\usepackage{fourier}
\usepackage[margin=1.3in]{geometry}

\newtheorem{theorem}{Theorem}[section]
\newaliascnt{proposition}{theorem}
\newtheorem{proposition}[proposition]{Proposition}
\aliascntresetthe{proposition}
\newaliascnt{lemma}{theorem}
\newtheorem{lemma}[lemma]{Lemma}
\aliascntresetthe{lemma}
\newaliascnt{corollary}{theorem}
\newtheorem{corollary}[corollary]{Corollary}
\aliascntresetthe{corollary}
\newaliascnt{assumption}{theorem}
\newtheorem{assumption}[assumption]{Assumption}
\aliascntresetthe{assumption}
\theoremstyle{definition}
\newaliascnt{definition}{theorem}
\newtheorem{definition}[definition]{Definition}
\aliascntresetthe{definition}
\newaliascnt{remark}{theorem}
\newtheorem{remark}[remark]{Remark}
\aliascntresetthe{remark}
\usepackage[capitalise,noabbrev]{cleveref}
\crefname{theorem}{Theorem}{Theorems}
\Crefname{theorem}{Theorem}{Theorems}
\crefname{proposition}{Proposition}{Propositions}
\Crefname{proposition}{Proposition}{Propositions}
\crefname{lemma}{Lemma}{Lemmas}
\Crefname{lemma}{Lemma}{Lemmas}
\crefname{corollary}{Corollary}{Corollaries}
\Crefname{corollary}{Corollary}{Corollaries}
\crefname{assumption}{Assumption}{Assumptions}
\Crefname{assumption}{Assumption}{Assumptions}
\crefname{definition}{Definition}{Definitions}
\Crefname{definition}{Definition}{Definitions}
\crefname{remark}{Remark}{Remarks}
\Crefname{remark}{Remark}{Remarks}
\crefname{appendix}{Appendix}{Appendices}
\Crefname{appendix}{Appendix}{Appendices}

\DeclareMathOperator{\Ent}{Ent}

\newcommand{\ud}{\,\mathrm{d}}
\newcommand{\RR}{\mathbb{R}}
\newcommand{\e}{\mathrm{e}}
\newcommand{\veps}{\varepsilon}

\newcommand{\La}{\mathcal L_a}
\newcommand{\Ls}{\mathcal L_s}
\newcommand{\Pt}{\mathcal P}
\newcommand{\Piv}{\Pi_v}
\newcommand{\Entv}{\Ent_v}
\newcommand{\Iv}{I_v}
\newcommand{\COT}{\mathcal C_{\rm OT}}
\newcommand{\muT}{\bar\mu_T}
\newcommand{\mux}{\mu_x}
\newcommand{\kv}{\kappa}
\newcommand{\abs}[1]{\lvert#1\rvert}
\newcommand{\Abs}[1]{\left\lvert#1\right\rvert}
\newcommand{\norm}[1]{\lVert#1\rVert}
\newcommand{\Ren}{\operatorname R}

\DeclareMathOperator{\divg}{div}
\newcommand{\dz}{\ud x\ud v}

\def \mc{\mathcal}

\title[Space-Time LSI and Hypercontractivity for Underdamped Langevin Dynamics]{Space-Time Log-Sobolev Inequality and Hypocoercive Hypercontractivity for Underdamped Langevin Dynamics}
\thanks{B.L. is supported by the National Key R$\&$D Program of China, Grant No. 2024YFA1016000, and the start-up fund 7200843 from City University of Hong Kong. The work of J.L. is supported in part by the National Science Foundation under grant DMS-2309378. The authors thank Jason Altschuler, Yuansi Chen, Sinho Chewi, Lihan Wang, and Yuliang Wang for helpful discussions.}
\author{Bowen Li}
\address{Department of Mathematics, City University of Hong Kong}
\email{boweli4@cityu.edu.hk}

\author{Jianfeng Lu}
\address{Department of Mathematics, Department of Physics, and Department of Chemistry, Duke University}
\email{jianfeng@math.duke.edu}
\date{}

\begin{document}

\begin{abstract}
We study hypercontractivity for underdamped Langevin dynamics with a convex confining potential whose spatial Gibbs marginal satisfies a logarithmic Sobolev inequality (LSI). Unlike in the overdamped case, the noise acts only on the velocity variable, so the usual LSI-based argument does not apply. Nevertheless, under an additional tame Hessian-growth assumption on the potential, we prove that, when the spatial LSI constant is $\rho$ and the friction parameter is of order $\sqrt{\rho}$, the semigroup satisfies a Gross-type $L^p$-to-$L^q$ estimate in which the integrability exponent grows exponentially on the kinetic time scale $t \sim \rho^{-1/2}$.

The central ingredient, of independent interest, is a space-time LSI, valid for controlled kinetic paths of finite action, which quantifies how dissipation in the velocity variable is transferred to the position variable. We derive it from a controlled version of the hypocoercive entropy-decay estimate, and we then restore the Gross hypercontractivity mechanism through a duality argument built on a non-reversible forward/backward interpolation of the underdamped Langevin semigroup. As a corollary, R\'{e}nyi divergences of every order decay exponentially at the sharp hypocoercive rate $\mathcal{O}(\sqrt{\rho})$.

\end{abstract}

\maketitle

\section{Introduction}

The notion of hypercontractivity, \textup{i.e.}, $L^p$-to-$L^q$ $(q > p)$ contractive estimates for a semigroup, originated in the study of quantum field theory~\cites{glimm1968boson,simon1972hypercontractive} and was developed in the seminal works of Nelson~\cites{nelson1966quartic,Nelson1973} for the Ornstein--Uhlenbeck semigroup. It was Gross~\cite{Gross1975}, however, who identified the logarithmic Sobolev inequality (LSI) as the underlying mechanism, establishing the equivalence between LSI and hypercontractivity. Specifically, if $(\Pt_t)_{t\ge 0}$ is an overdamped Langevin semigroup with invariant measure $\mu$, then $\mu$ satisfies an LSI with constant $\rho$:
\begin{equation}  \label{eq:intro-classical-lsi}
    \Ent_{\mu}(\varphi) \le \frac{1}{2\rho}\int \frac{\abs{\nabla \varphi}^2}{\varphi}\ud\mu
    \qquad \text{for all } \varphi \ge 0 \text{ with } \int \varphi\ud\mu = 1,
\end{equation}
if and only if for $L^p$ integrable $f$,
\begin{equation}\label{eq:Gross}
    \norm{\Pt_t f}_{L^{q(t)}(\mu)}
    \le
    \norm{f}_{L^p(\mu)},
    \qquad
    q(t)-1 = \e^{2\rho t}(p-1),
    \qquad \forall\, p>1.
\end{equation}
That is, the semigroup improves integrability over time, with the admissible exponent growing exponentially at a rate given by the LSI constant. Bakry and \'Emery~\cite{BakryEmery1985} subsequently developed the $\Gamma_2$ calculus approach to hypercontractivity for diffusion semigroups, based on a symmetric dual formulation \cite{Neveu1976}; we refer to~\cite{BakryGentilLedoux2014} for a systematic exposition. Beyond their intrinsic interest, LSI and the associated hypercontractive estimates are now standard tools with deep connections to many areas, including concentration of measure, transportation cost inequalities, isoperimetric inequalities, and mixing-time bounds for Markov chains~\cites{guionnet2004lectures,ledoux2006concentration,BCR2006}.

In this work, we investigate a quantitative analogue of this connection between LSI and hypercontractivity for a prototypical hypoelliptic and non-reversible dynamics. We consider the underdamped Langevin equation:
\begin{equation}\label{eq:underdamped-Langevin}
        \begin{aligned}
        \ud X_t&=V_t\ud t,\\
        \ud V_t&=-\nabla U(X_t)\ud t-\gamma V_t\ud t+\sqrt{2\gamma}\ud B_t,
        \end{aligned}
\end{equation}
where $(X_t,V_t)\in\RR^d\times\RR^d$ are the position and velocity, $U:\RR^d\to\RR$ is a convex confining potential, $\gamma>0$ is the friction parameter, and $(B_t)_{t\ge0}$ is a standard Brownian motion in $\RR^d$. Its invariant law is the Gibbs measure $
 \mu(\ud x\ud v) = \mux(\ud x)\,\kv(\ud v) \propto \e^{-U(x)-\abs v^2/2}\ud x\ud v$
 and the associated density semigroup is
\begin{equation*}
\Pt_t = \e^{t(-\La+\gamma\Ls)},\qquad \La = v\cdot\nabla_x-\nabla U\cdot\nabla_v,  \qquad \Ls = \Delta_v-v\cdot\nabla_v.
\end{equation*}
The Hamiltonian transport $\La$ is skew-adjoint in $L^2(\mu)$, whereas the velocity Ornstein--Uhlenbeck operator $\Ls=-\nabla_v^*\nabla_v$ is symmetric and non-positive.
The underdamped Langevin dynamics is more delicate to analyze than the overdamped case: the diffusion is degenerate and acts only on the velocity variable, so the regularizing effect of the noise must be transferred to the position variable through the Hamiltonian transport. The understanding of this mechanism dates back to Kolmogorov~\cite{Kolmogorov1934} and H\"ormander~\cite{Hormander1967} within the framework of hypoellipticity; see, \textup{e.g.}, \cites{DesvillettesVillani2001, HerauNier2004, HelfferNier2005, Herau2007, GolseImbertMouhotVasseur2019} for subsequent developments. To quantify this indirect dissipation and to obtain explicit decay rates, Villani~\cite{Villani2009} introduced the hypocoercivity framework, and the resulting quantitative theory has been substantially developed in recent years; see, \textup{e.g.}, \cites{DolbeaultMouhotSchmeiser2015, Baudoin2017, BernardFathiLevittStoltz2022, BrigatiStoltz2025, CaoLuWang2023, AlbrittonArmstrongMourratNovack2024, FanLiLu2026, Lu2026}. In particular, recent works have established sharp hypocoercive decay estimates for the underdamped Langevin dynamics, both in the $L^2$ norm~\cites{CaoLuWang2023, FanLiLu2026} and in relative entropy~\cites{Lu2026}.

The degeneracy of the noise creates a structural obstruction to the classical Gross argument. Specifically, along the underdamped Langevin semigroup $\Pt_t$, we only have
\begin{equation}
\label{eq:intro-dissipation}
    -\frac{\ud}{\ud t}\Ent_\mu(\Pt_t f)
    =\gamma\Iv(\Pt_t f),
    \qquad
    \Iv(g):=\int\frac{\abs{\nabla_v g}^2}{g}\ud\mu.
\end{equation}
Consequently, no instantaneous inequality of the form $\Ent_\mu(g)\le C\Iv(g)$ can hold uniformly over probability densities on phase space. It is worth noting that this is not due to the absence of an LSI for the Gibbs measure $\mu$. In fact, by tensorization, the product measure $\mu=\mux\otimes\kv$ satisfies a phase-space LSI with constant $\min\{\rho,1\}$ if $\mu_x \propto e^{- U(x)}$ has an LSI constant $\rho$. That inequality, however, controls the entropy by the full Fisher information $\int (\abs{\nabla_x g}^2+\abs{\nabla_v g}^2)/g\ud\mu$,
whereas the dynamics dissipates only its velocity component, as shown in \eqref{eq:intro-dissipation}. Thus, the missing spatial coercivity has to be produced dynamically, and quantifying that production is one of the tasks of this work.

Hypercontractivity, a stronger notion than the $L^2$ and entropy convergence, for kinetic Fokker-Planck equations was previously studied by Wang~\cite{Wang2017}, building on a coupling and dimension-free Harnack inequality argument~\cite{Wang1997}. Specifically, Wang's argument combines a fixed-time Harnack inequality, obtained by coupling by change of measure, with a separate long-time contractive coupling assumption, and yields eventual $L^2$-to-$L^4$ hypercontractivity in the sense of Nelson~\cite{Nelson1973}. This approach was recently refined and applied to sampling algorithms by Altschuler, Chewi, and Zhang~\cites{AltschulerChewiZhang2025, ZhangAltschulerChewi2026}. Its long-time component, however, relies on a pathwise contraction estimate. When specialized to the underdamped Langevin dynamics~\cite{Wang2017}*{Section 5}, the synchronous coupling with a twisted metric used to verify it applies only in the high-friction regime: after normalizing the friction, one needs the effective curvature $\|\nabla^2 U\|/\gamma^2$ to be small, which for a $\beta$-smooth potential ($\nabla^2 U\preceq\beta\,\mathrm{Id}$) means $\gamma\gtrsim\sqrt{\beta}$. Hence, these Harnack-coupling estimates can neither reach the low-friction regime $\gamma=\Theta(\sqrt\rho)$ nor yield the Gross-type $L^p$-to-$L^q$ hypercontractive growth of the form $q(t)-1\simeq\e^{c\sqrt\rho\,t}(p-1)$ on the kinetic time scale $t\sim\rho^{-1/2}$.

\subsection*{Main results}
Throughout the work, we assume that $U$ is smooth and convex, that the spatial Gibbs marginal $\mux(\ud x)\propto\e^{-U(x)}\ud x$ satisfies the LSI \eqref{eq:intro-classical-lsi} with constant $\rho>0$, and that $U$ obeys the tame Hessian-growth bound $\abs{\nabla^2U} \lesssim 1+\abs{\nabla U}$; see \cref{gm:ass:lsi,gm:ass:standing} for the precise statements. We work in the kinetic scaling of the friction, where the dimensionless parameter $\Gamma>0$ measures the friction relative to the spatial LSI scale:
\begin{equation}
\label{eq:intro-kinetic-scaling}
    \gamma = \Gamma\sqrt\rho, \qquad  \Gamma>0.
\end{equation}
The main contribution of this work is a space-time logarithmic Sobolev inequality, valid for every admissible controlled kinetic path of finite action, which replaces the missing instantaneous LSI. We then show that this inequality restores the Gross hypercontractivity mechanism through a non-reversible forward/backward interpolation. As a further application, we establish the sharp hypocoercive R\'enyi divergence decay, pointing toward further applications to sampling algorithms.

\medskip

\noindent (1) \textbf{Space-time logarithmic Sobolev inequality.} Our first main result is the space-time LSI. The idea draws on a recent thread of quantitative hypocoercivity, in which one replaces the missing phase-space inequality by a functional inequality on space-time. In the $L^2$ theory, space-time Poincar\'e inequalities were introduced by Albritton, Armstrong, Mourrat, and Novack~\cite{AlbrittonArmstrongMourratNovack2024} and used quantitatively by Cao, Lu, and Wang~\cite{CaoLuWang2023} to obtain a sharp hypocoercive $L^2$ estimate for the underdamped Langevin dynamics.

Let $T>0$, and define the time-augmented probability measure $\ud\muT=\frac1T\ud t\otimes\ud\mu$ on $(0,T)\times\RR^{2d}$. For a path of probability densities $g=\{g_t\}_{t\in(0,T)}$ with respect to $\mu$, introduce
\begin{equation*}
    \Ent_{\muT}(g) = \frac1T \int_0^T \Ent_\mu(g_t)\ud t,
    \qquad
    I_{v,T}(g) = \frac1T \int_0^T \Iv(g_t)\ud t.
\end{equation*}
We say that a measurable velocity current $W:(0,T)\times\RR^{2d}\to\RR^d$ drives the path $g$ if
\begin{equation}
\label{eq:intro-controlled-path}
    (\partial_t+\La)g_t = \nabla_v^*W_t
    \qquad\text{in }\mathcal D'((0,T)\times\RR^{2d}),
\end{equation}
so that the departure of $g$ from the Hamiltonian transport is an exact velocity divergence. The quadratic action of the pair $(g,W)$ is $J_T(g,W) := \int \abs{W}^2/g \ud\muT$, and minimizing over all currents driving $g$ gives the intrinsic action of the path,
\begin{equation} \label{eq:repretA}
    \mathcal A_T(g) := \inf\bigl\{J_T(g,W)\ :\ W \text{ drives } g\bigr\} =  \frac{1}{T} \int_0^T \norm{(\partial_t+\La)g_t}_{-1,g_t}^2\ud t,
\end{equation}
which compensates for the missing spatial Fisher information in \eqref{eq:intro-dissipation}. Then, subject only to mild boundedness conditions on the path $g$, we prove the following space-time LSI: there exist constants $\tau_{\rm ST}(\Gamma)$ and $C_{\rm ST}(\Gamma)$, depending only on $\Gamma$, such that
\begin{equation}
\label{eq:intro-stlsi}
    \Ent_{\muT}(g) \le  C_{\rm ST}(\Gamma)
    \left[ I_{v,T}(g) + \frac1\rho\mathcal A_T(g)
    \right],
\end{equation}
whenever $T\ge\tau_{\rm ST}(\Gamma)\,\rho^{-1/2}$. The precise regular-path formulation is given in \cref{thm:stlsi}, and its finite-action extensions in \cref{gm:cor:stlsi,gm:cor:stlsi-action}.

We emphasize three features of \eqref{eq:intro-stlsi}. First, the right-hand side contains no spatial Fisher information: the spatial coercivity is produced by the space-time operator $\partial_t+\La$. Second, the space-time LSI \eqref{eq:intro-stlsi} is a pathwise functional inequality, not merely an estimate along solutions of the Langevin equation. It applies to every controlled path satisfying \eqref{eq:intro-controlled-path}; the forcing may be rough, provided its action is finite and the mild moment conditions in \cref{gm:ass:path} hold.
Third, this inequality contains the spatial LSI for $\mu_x$ as its stationary case. Indeed, let $q$ be a spatial probability density with finite entropy and finite spatial Fisher information $I_x(q):=\int\abs{\nabla_xq}^2/q\ud\mux$, and consider the constant-in-time, velocity-independent path $g_t\equiv q$, satisfying
\begin{equation*}
    (\partial_t+\La)g_t =  v\cdot\nabla_xq  = \nabla_v^*W,  \qquad  W:=\nabla_xq.
\end{equation*}
It follows that for such paths, $I_{v,T}(g)=0$, $J_T(g,W)=I_x(q)$, and $\Ent_{\muT}(g)=\Ent_{\mux}(q)$, and the inequality \eqref{eq:intro-stlsi} collapses to the spatial LSI \eqref{eq:intro-classical-lsi}, up to a constant,
\begin{equation*}
    \Ent_{\mux}(q) \le \frac{C_{\rm ST}(\Gamma)}{\rho}\,  I_x(q).
\end{equation*}
This also suggests that the weight $\rho^{-1}$ in front of the action $\mc{A}_T(g)$ in \eqref{eq:intro-stlsi} cannot be improved.

To discuss the main idea of the proof, let us briefly recall the hypocoercive entropy decay established by one of us in~\cite{Lu2026}: the underdamped Langevin semigroup $\Pt_t$ with friction $\gamma=\Gamma\sqrt\rho$, $\Gamma>0$, satisfies
\begin{equation*}
        \Ent_\mu(\Pt_tf)
        \le
        C_\Gamma\,\e^{-\lambda_\Gamma\sqrt\rho\,t}\,\Ent_\mu(f),
        \qquad t\ge 0,
\end{equation*}
for every probability density $f$, with constants $C_\Gamma,\lambda_\Gamma>0$ depending only on $\Gamma$. The decay rate $\sqrt\rho$ is sharp and occurs on the ballistic time scale $t\sim\rho^{-1/2}$. In the overdamped case, exponential entropy decay of this kind follows directly from the LSI through the gradient-flow viewpoint of Otto calculus~\cites{JordanKinderlehrerOtto1998, Otto2001, OttoVillani2000}. For the hypocoercive underdamped dynamics, however, no LSI of this type is available, owing to the degeneracy of the diffusion. The analysis of~\cite{Lu2026} instead relies on a modified entropic Lyapunov functional \`{a} la \cites{DolbeaultMouhotSchmeiser2015,FanLiLu2026} with a Wasserstein entropy-current corrector:
\begin{equation}
\label{eq:intro-modified-entropy}
\mathcal H_\theta(g) = \Ent_\mu(g) +
    \theta\sqrt\rho\,\COT(g).
\end{equation}
Here, writing $q=\Piv g$ and $j=\Piv(vg)$ for the spatial marginal and the velocity current, and $T_q$ for the Brenier map transporting $q\mux$ to $\mux$, the corrector is given by
\begin{equation*}
\COT(g) = \int j(x)\cdot\bigl(x-T_q(x)\bigr)\ud\mux(x),
\end{equation*}
which couples the spatial displacement from equilibrium to the velocity current and measures the transfer of dissipation from $v$ to $x$. This modified entropy \eqref{eq:intro-modified-entropy} is equivalent to $\Ent_\mu$ and obeys a Gr\"onwall-type differential inequality, $\tfrac{\ud}{\ud t}\mathcal H_\theta(g_t)\lesssim - \sqrt\rho\,\mathcal H_\theta(g_t)$, which thereby yields the exponential decay of the relative entropy.

The space-time LSI \eqref{eq:intro-stlsi} is an entropic counterpart of the space-time Poincar\'e inequality. The starting point of its proof is the modified entropy functional introduced in \eqref{eq:intro-modified-entropy}. We consider the general controlled kinetic Fokker--Planck equation
\begin{equation}
\label{eq:intro-controlled-langevin}
    (\partial_t+\La)g_t  =  \gamma\Ls g_t
    +  \nabla_v^*(g_tu_t),
\end{equation}
in which an arbitrary control field $u_t$ perturbs the Langevin dynamics in the velocity directions. The central estimate of the proof is the controlled differential inequality along \eqref{eq:intro-controlled-langevin}:
\begin{equation}
\label{eq:intro-controlled-modified-entropy}
    \frac{\ud}{\ud t}
    \mathcal H_\theta(g_t) \le
    -c_\Gamma\sqrt\rho\,
    \Ent_\mu(g_t) +
    \frac{C_\Gamma}{\sqrt\rho} \int g_t\abs{u_t}^2\ud\mu,
\end{equation}
with $c_\Gamma,C_\Gamma>0$ depending only on $\Gamma$; see \cref{prop:controlled}. For $u\equiv0$, this is the hypocoercive entropy-decay estimate of~\cite{Lu2026}. The inequality \eqref{eq:intro-controlled-modified-entropy} quantifies the entropy cost of \emph{every} velocity perturbation of the Hamiltonian flow, and hence contains strictly more information than decay along the semigroup. Two structural estimates enter its proof: a weak Wasserstein acceleration inequality, controlling the second time derivative of the transport distance of the spatial marginal in the distributional sense (\cref{lem:wasserstein-acceleration}, proved in \cref{app:wasserstein-acceleration}), and a localized stress estimate for the centered conditional velocity covariance (\cref{lem:localized-stress}). From the controlled estimate \eqref{eq:intro-controlled-modified-entropy} to the space-time LSI, it suffices to choose a near-minimizer $W = g A$ of $\mc{A}_T(g)$ in \eqref{eq:repretA}, i.e.,
\begin{equation*}
    (\partial_t+\La)g_t  =
    \nabla_v^*(g_tA_t),
    \qquad
    \frac1T  \int_0^T\int  g_t\abs{A_t}^2\ud\mu\ud t  \approx
    \mathcal A_T(g),
\end{equation*}
and, by setting $u_t=A_t+\gamma\nabla_v\log g_t$, integrate \eqref{eq:intro-controlled-modified-entropy} over a time window $T\gtrsim\rho^{-1/2}$ (the time needed for the dissipated entropy to dominate the corrector) and absorb the boundary terms into the time average. This yields \eqref{eq:intro-stlsi} for regular paths (\cref{thm:stlsi}); the extension to finite-action paths (\cref{gm:cor:stlsi,gm:cor:stlsi-action}) is based on a technical approximation result in \cref{gm:thm:recovery}.

\medskip

\noindent (2) \textbf{Hypocoercive hypercontractivity.}
Equipped with the space-time LSI, we obtain our second main result: a kinetic analogue of the Gross estimate \eqref{eq:Gross} for the underdamped Langevin semigroup $\Pt_t$ with $\gamma=\Gamma\sqrt\rho$. We consider the time window $T=\tau\rho^{-1/2}$ with $\tau\ge\tau_{\rm ST}(\Gamma)$ as in \eqref{eq:intro-stlsi}. Then there exists an explicit constant $\alpha_{\Gamma,\tau}>1$, depending only on $(\Gamma,\tau)$, such that for every $p>1$ and $f\in L^p(\mu)$,
\begin{equation}
\label{eq:intro-one-step-hypercontractivity}
    \norm{\Pt_Tf}_{
        L^{1+\alpha_{\Gamma,\tau}(p-1)}(\mu)}
    \le \norm{f}_{L^p(\mu)}.
\end{equation}
This is the one-step hypercontractive estimate of \cref{thm:one-slab}. Iterating \eqref{eq:intro-one-step-hypercontractivity} over consecutive windows yields \cref{cor:renyi}: for all $t\ge0$,
\begin{equation}
\label{eq:intro-iterated-hypercontractivity}
    \norm{\Pt_t f}_{L^{q_{\Gamma,\tau}(t)}(\mu)} \le
    \norm{f}_{L^p(\mu)},  \qquad
    q_{\Gamma,\tau}(t)-1 =  \alpha_{\Gamma,\tau}^{
        \lfloor\sqrt\rho\,t/\tau\rfloor}(p-1).
\end{equation}
Letting $\lambda_{\Gamma,\tau}:=\tau^{-1}\log\alpha_{\Gamma,\tau}$, the exponent satisfies
\begin{equation*}
    q_{\Gamma,\tau}(t)-1 \ge \alpha_{\Gamma,\tau}^{-1} (p-1) \e^{\lambda_{\Gamma,\tau}\sqrt\rho\,t},
\end{equation*}
in analogy to the classical Gross theorem. Thus, the integrability exponent improves on the ballistic (kinetic) time scale $t\sim\rho^{-1/2}$, matching the sharp $\mathcal{O}(\sqrt\rho)$ hypocoercive decay rate of~\cites{CaoLuWang2023,FanLiLu2026,Lu2026}.

The proof proceeds through a non-reversible Gross interpolation. Specifically, to establish the $L^p$-to-$L^q$ contraction of $\Pt_T$, it suffices to show
\begin{equation*}
    \sup\Bigl\{\textstyle\int \Pt_T\varphi\cdot\psi\,\ud\mu\ :\ \norm{\varphi}_{L^{p}(\mu)} = \norm{\psi}_{L^{q'}(\mu)} = 1\Bigr\} \le 1.
\end{equation*}
For this, because the underdamped Langevin semigroup is non-reversible, the interpolation pairs the forward semigroup with its $L^2(\mu)$-adjoint, namely the time-reversed (backward) semigroup. For positive test functions $\varphi,\psi$, we consider the forward-backward interpolation path:
\begin{equation}
\label{eq:intro-interpolation}
    G_s  =\frac{\varphi_s\psi_s}{Z},
    \qquad
    \varphi_s =  \Pt_s\varphi,
    \qquad
    \psi_s  = \Pt_{T-s}^*\psi,
    \qquad
    Z = \int \Pt_T\varphi\cdot\psi\ud\mu.
\end{equation}
In the reversible case $\Pt^*=\Pt$, and \eqref{eq:intro-interpolation} reduces to the classical two-sided semigroup interpolation of the Bakry--\'Emery method~\cites{Neveu1976,BakryEmery1985}. Although $G_s$ is not a trajectory of the semigroup $\Pt_t$, its defect from the Hamiltonian transport is purely controlled by a velocity current:
\begin{equation}
\label{eq:intro-interpolation-current}
    (\partial_s+\La)G_s = \nabla_v^*(G_sA_s),
    \qquad
    A_s  = -\gamma \left( \nabla_v\log\varphi_s - \nabla_v\log\psi_s
    \right).
\end{equation}
Thus the space-time LSI \eqref{eq:intro-stlsi} applies and bounds the time-averaged entropy of $G$, which propagates to its endpoints:
\begin{equation}
\label{eq:intro-endpoint-entropy}
    \Ent_\mu(G_0)+\Ent_\mu(G_T) \lesssim \mathcal D_T(\varphi,\psi),
    \qquad
    \mathcal D_T(\varphi,\psi) = \int
    \left(
        \abs{\nabla_v\log\varphi_s}^2  + \abs{\nabla_v\log\psi_s}^2
    \right)  G\ud\muT.
\end{equation}
On the other hand, differentiating the pairings $\int G_s\log\varphi_s\ud\mu$ and $\int G_s\log\psi_s\ud\mu$, and using the variational formula for the entropy yields
\begin{equation}
\label{eq:intro-closure}
    2\log Z \le
    \left(\frac2p-1\right)\Ent_\mu(G_0)  +
    \left(1-\frac2q\right)\Ent_\mu(G_T)
    - \gamma T\,\mathcal D_T(\varphi,\psi)
\end{equation}
for $\varphi,\psi$ normalized in $L^p(\mu)$ and $L^{q'}(\mu)$. This allows us to choose a critical $p_c < q_c$ such that $\log Z \le 0$, i.e., $Z \le 1$. Duality and Riesz--Thorin interpolation with the Markov $L^1$- and $L^\infty$-contractions then yield the full one-step estimate \eqref{eq:intro-one-step-hypercontractivity}. These arguments are developed in \cref{sec:interpolation,sec:hypercontractivity}.

\medskip

\noindent (3) \textbf{R\'enyi divergence decay.} Finally, before closing this section, let us mention that a closely related R\'enyi form of kinetic regularization, termed \emph{hyperequilibration} by Altschuler and Chewi~\cite{AltschulerChewi2024}, was used in the analysis of sampling algorithms.  If $\nu_t$ denotes the law at time $t$, one asks for exponential relaxation of the R\'enyi divergence
\begin{equation*}
        \Ren_q(\nu_t\mid\mu)\lesssim \e^{-\lambda t}\Ren_q(\nu_0\mid\mu),
        \qquad q>1.
\end{equation*}
For overdamped Langevin dynamics, this type of R\'enyi decay was established earlier by Cao, Lu, and Lu~\cite{CaoLuLu2019}  and by Vempala and Wibisono~\cite{VempalaWibisono2019}.
In the analysis of sampling algorithms, in particular the Hamiltonian Monte Carlo (HMC) algorithm, such R\'enyi estimates are important for establishing warm starts and obtaining optimal complexity bounds~\cite{ChenGatmiry2023}. In particular, Zhang, Altschuler, and Chewi~\cite{ZhangAltschulerChewi2026} prove that a non-Metropolized HMC scheme, which is essentially a time-splitting discretization of the underdamped Langevin dynamics, can produce the required R\'enyi warm start using a Harnack inequality and shifted-composition estimates. As noted earlier, the coupling-based technique of~\cites{Wang2017, ZhangAltschulerChewi2026} only treats the high-friction regime and thus does not yield the kinetic hypercontractive rate.

As a consequence of the hypocoercive hypercontractivity, \cref{cor:renyi-decay} gives the desired hyperequilibration estimate
\begin{equation*}
        \Ren_q(\nu_t\mid\mu)
        \lesssim
        \e^{-\lambda_{\Gamma,\tau}\sqrt\rho\,t}
        \Ren_q(\nu_0\mid\mu), \qquad q>1,
\end{equation*}
with the sharp hypocoercive rate $\mathcal{O}(\sqrt\rho)$, matching the optimal $L^2$ and entropy decay estimates of~\cites{CaoLuWang2023,FanLiLu2026,Lu2026}. This improves the R\'enyi estimate used in~\cite{ZhangAltschulerChewi2026} and may lead to sharper complexity bounds for the HMC algorithm; we leave this to future work.

\subsection*{Implications and discussion} As elaborated above, in the reversible diffusion case, the LSI is equivalent to hypercontractivity and drives the exponential entropy decay. In the kinetic setting, no instantaneous LSI is available, and we show that the space-time LSI holds and can be a useful substitute. More precisely, the results above give the following diagram of implications:
\begin{equation*}
\begin{tikzcd}[column sep=large, row sep=large]
    \text{Space-time LSI \eqref{eq:intro-stlsi}}
    \arrow[r, Rightarrow]
    \arrow[dr, Rightarrow]
    & \text{Flow LSI \eqref{eq:intro-trajectory-stlsi}}
    \arrow[r, Leftrightarrow]
    & \text{Entropy decay of \cite{Lu2026}}
    \\
    & \text{Hypercontractivity \eqref{eq:intro-one-step-hypercontractivity}}
    \arrow[u, Rightarrow]
    &
\end{tikzcd}
\end{equation*}

The implication from \eqref{eq:intro-stlsi} to \eqref{eq:intro-trajectory-stlsi} is simply a trajectory reduction of the space-time LSI (we thus refer to \eqref{eq:intro-trajectory-stlsi} as the \emph{flow LSI}). Along an underdamped Langevin semigroup trajectory $g_t=\Pt_tf$, the driving current is explicit
\begin{equation*}
    (\partial_t+\La)g_t = \gamma\Ls g_t =
    \nabla_v^*\bigl(g_t\cdot(-\gamma\nabla_v\log g_t)\bigr),
\end{equation*}
and $\mathcal A_T(g) = \gamma^2I_{v,T}(g)$; see \cref{rem:Ls-negative-norm}. Since $\gamma^2/\rho=\Gamma^2$, the inequality \eqref{eq:intro-stlsi} collapses to
\begin{equation}
\label{eq:intro-trajectory-stlsi}
    \Ent_{\muT}(g) \le C_{\rm ST}(\Gamma)\,(1+\Gamma^2)\,  I_{v,T}(g).
\end{equation}
By the dissipation identity \eqref{eq:intro-dissipation}, which gives $\gamma T\,I_{v,T}(g)=\Ent_\mu(f)-\Ent_\mu(\Pt_Tf)$, together with the monotonicity of $t\mapsto\Ent_\mu(g_t)$, the bound \eqref{eq:intro-trajectory-stlsi} is equivalent, up to constants depending only on $(\Gamma,\tau)$, to the one-window entropy contraction
\begin{equation}
\label{eq:intro-one-window}
    \Ent_\mu(\Pt_Tf) \le c\,\Ent_\mu(f), \qquad c = \frac{C_{\rm ST}(\Gamma)(1+\Gamma^2)}{\Gamma\tau+C_{\rm ST}(\Gamma)(1+\Gamma^2)}\in(0,1).
\end{equation}
Moreover, on the sliding window $[s,s+T]$, \eqref{eq:intro-trajectory-stlsi} is equivalent to the Gr\"onwall inequality
\begin{equation*}
    \frac{\ud}{\ud s}F(s) \le - \frac{\gamma}{C_{\rm ST}(\Gamma)(1+\Gamma^2)}\,F(s), \qquad F(s):=\int_s^{s+T}\Ent_\mu(g_t)\ud t.
\end{equation*}
Therefore, the windowed entropy $F$ decays exponentially at the sharp rate $\mathcal O(\sqrt\rho)$, and the monotonicity sandwich $T\,\Ent_\mu(g_{s+T})\le F(s)\le T\,\Ent_\mu(g_s)$ transfers this decay to the entropy itself.

On the other hand, hypercontractivity  \eqref{eq:intro-one-step-hypercontractivity} contains more information than the flow LSI \eqref{eq:intro-trajectory-stlsi}. Indeed, letting $p\downarrow1$ in the R\'enyi form of \eqref{eq:intro-one-step-hypercontractivity} yields the one-window entropy contraction
\eqref{eq:intro-one-window} with $c=\alpha_{\Gamma,\tau}^{-1}$. In contrast, the flow LSI \eqref{eq:intro-trajectory-stlsi} alone does not furnish the $L^p$-to-$L^q$ estimate in our argument. Thus, unlike in classical Gross's theorem, the implication established here is one-way:
 from space-time LSI to hypocoercive hypercontractivity.
The space-time LSI is an a priori stronger pathwise principle: it controls every admissible controlled path, whereas hypercontractivity is a property of the semigroup.  Our proof uses the space-time LSI only on the particular forward/backward interpolation \eqref{eq:intro-interpolation}. We do not establish a converse implication from a potentially stronger notion of hypercontractivity; this is an interesting open direction.

\section{Setting and preliminaries} \label{sec:setting}

Throughout, we work on the phase space $\RR^{2d}$ equipped with the reference Gibbs measure:
\begin{equation*}
    \mu(\ud x\ud v) = \mux(\ud x)\,\kv(\ud v),
\end{equation*}
where
\begin{equation*}
    \mux(\ud x) = Z_x^{-1}\e^{-U(x)}\ud x,
    \qquad
    \kv(\ud v) = (2\pi)^{-d/2}\e^{-\abs{v}^2/2}\ud v,
\end{equation*}
and $Z_x = \int_{\RR^d} \e^{-U(x)}\ud x$ is the spatial normalizing constant.

We consider the underdamped Langevin dynamics given by the SDE \eqref{eq:underdamped-Langevin}, and its associated semigroup $\Pt_t = \e^{t(-\La + \gamma \Ls)}$, $t \ge 0$,  acting on probability densities relative to $\mu$, where the Hamiltonian transport and velocity Ornstein--Uhlenbeck generators are
\begin{equation*}
    \La = v\cdot\nabla_x - \nabla U\cdot\nabla_v,
    \qquad
    \Ls = \Delta_v - v\cdot\nabla_v,
\end{equation*}
respectively. In $L^2(\mu)$, the transport generator $\La$ is skew-adjoint, while $\Ls = -\nabla_v^{*}\nabla_v$ is symmetric and non-positive. Here $$\nabla_v^{*} F = -\operatorname{div}_v F + v\cdot F$$ denotes the adjoint of $\nabla_v$ in $L^2(\kv)$.  We also write
\begin{equation*}
\nabla_x^*F=-\operatorname{div}_x F+\nabla U\cdot F
\end{equation*}
for the adjoint of $\nabla_x$ in $L^2(\mu_x)$. For a matrix field $M=(M_{ij})$, $\nabla_x^*M$ is understood row-wise: $(\nabla_x^*M)_i=-\sum_j\partial_{x_j}M_{ij}
+\sum_j(\partial_{x_j}U)M_{ij}$.
In addition, for an open set $\Omega\subset\RR^n$, we denote by $\mathcal D'(\Omega)$ the space of distributions on $\Omega$, namely the continuous dual of $C_c^\infty(\Omega)$.

We adopt the following standing assumptions on the potential $U$: convexity together with a log-Sobolev inequality for $\mux$, and a Hessian bound.

\begin{assumption}\label{gm:ass:lsi}
The potential $U \in C^\infty(\mathbb{R}^d)$ is convex, and the spatial marginal $\mu_x$ satisfies the \emph{logarithmic Sobolev inequality}
\begin{equation}\label{eq:lsi}
    \operatorname{Ent}_{\mu_x}(f) \leq \frac{1}{2\rho} \int \frac{|\nabla f|^2}{f} \, \mathrm{d}\mu_x
    \qquad \text{for all } f \geq 0 \text{ with } \int f \, \mathrm{d}\mu_x = 1,
\end{equation}
with LSI constant $\rho > 0$.
\end{assumption}

\begin{assumption}\label{gm:ass:standing}
The potential $U$ satisfies the \emph{Hessian bound}: there exists $C_U \geq 1$ such that
\begin{equation}\label{gm:eq:tame}
    \abs{\nabla^2 U(x)} \leq C_U\bigl(1 + \abs{\nabla U(x)}\bigr)
     \qquad x \in \mathbb{R}^d.
\end{equation}
Without loss of generality, by adding a constant to $U$ we assume $\inf U = 0$, so that $U \geq 0$.
\end{assumption}

\begin{remark}
\cref{gm:ass:lsi} is standard and coincides with \cite{Lu2026}*{Assumption 2.1}. Similar Hessian bounds as in \cref{gm:ass:standing} are commonly used in \cites{BernardFathiLevittStoltz2022,CaoLuWang2023,Villani2009}. In particular, the Hessian bound \eqref{gm:eq:tame} implies that $U$ grows at most exponentially fast as $\abs{x} \to \infty$; see \cref{gm:lem:tame-growth}. It is, in turn, implied by the H\'erau--Nier class of polynomially confining potentials in \cite{HerauNier2004}*{Hypothesis 1} (or \cite{Lu2026}*{Assumption 2.2}), where there exists $n \geq 1$ such that for $\abs{x}$ sufficiently large,
\begin{equation}
    \abs{\partial^\alpha U(x)} \lesssim \abs{x}^{2n-\min\{\abs\alpha,2\}} \qquad \text{and} \qquad
        \abs{\nabla U(x)} \gtrsim \abs{x}^{2n-1}.
\end{equation}
This directly gives $\abs{\nabla^2 U} \lesssim \abs{x}^{2n-2} \leq 1+\abs{x}^{2n-1} \lesssim 1+\abs{\nabla U}$ for large $|x|$, as desired.

\end{remark}

Let $\Piv$ denote integration along the velocity fiber against $\kv$, that is,
\begin{equation*}
    (\Piv \varphi)(x) = \int_{\RR^d} \varphi(x, v)\,\kv(\ud v).
\end{equation*}
For a probability density $g$ on $\mathbb{R}^{2d}$ with respect to $\mu$, define its zeroth- and first-order $v$-moments
\begin{equation*}
    q(x) := \Piv g, \qquad j(x) := \Piv(vg),
\end{equation*}
which are the spatial marginal density and the velocity current, respectively. On $\{q > 0\}$ we define the conditional density of $g$ given $x$,
\begin{equation*}
    h_x(v) = \frac{g(x, v)}{q(x)},
\end{equation*}
extended by the convention $h_x \equiv 1$ on $\{q = 0\}$. With this notation, on $\{q > 0\}$ the conditional velocity field
\begin{equation*}
    m(x) := \frac{j(x)}{q(x)} = \int_{\RR^d} v\, h_x(v)\,\kv({\rm d} v) = \mathbb{E}_{g\mu}[V \mid X = x]
\end{equation*}
is the conditional mean of $V$ given $X$ under the joint law $(X, V) \sim g\mu$, where we adopt the convention $j = m = 0$ on $\{q = 0\}$.

The velocity Fisher information of $g$ is
\begin{equation*}
    \Iv(g) = \int \frac{\abs{\nabla_v g}^2}{g}\ud\mu,
\end{equation*}
with the standard convention $\Iv(g) = +\infty$ when the right-hand side is not well-defined. The relative entropy admits the chain-rule decomposition (see, \textup{e.g.}, \cite{Lu2026}*{(2.14)--(2.15)})
\begin{equation}\label{eq:ent-decomp}
    \Ent_\mu(g) := \int g\log g\ud\mu
    = \Ent_{\mux}(q) + \Entv(g),
\end{equation}
where
\begin{equation*}
    \Ent_{\mux}(q) := \int q \log q\ud\mux
\end{equation*}
is the relative entropy of the spatial marginal $q\mux$ with respect to $\mux$, and
\begin{equation*}
    \Entv(g) := \int_{\{q > 0\}} q(x)\, \Ent_\kv(h_x)\ud\mux(x), \qquad \Ent_\kv(h_x) := \int h_x \log h_x\ud\kv,
\end{equation*}
is the conditional velocity entropy under $g\mu$, taken relative to $\kv$ and averaged over the spatial marginal $q\mux$. Both terms $\Ent_{\mux}(q)$ and $\Entv(g)$ are non-negative, and~\eqref{eq:ent-decomp} is understood in the extended sense, with values in $[0, +\infty]$.

A central ingredient in the modified entropy method of~\cite{Lu2026} for sharp hypocoercive entropy decay is the Wasserstein current corrector $\COT$, built from the Brenier transport map between $q\mux$ and $\mux$. By the Herbst argument, the LSI \eqref{eq:lsi} implies Gaussian concentration for Lipschitz functions; in particular $\mux$ has a finite second moment. Moreover, by the Otto-Villani theorem \cites{OttoVillani2000}, $\mu_x$ satisfies Talagrand's inequality
\begin{equation} \label{eq:talagrant}
      W_2^2(\nu,\mux)
    \le \frac{2}{\rho}\Ent_{\mux}\!\left(\frac{\ud \nu}{\ud\mux}\right),
    \qquad \nu\ll \mux\,,
\end{equation}
Consequently, if $\Ent_{\mu_x}(q)<\infty$, then $q\mux$ also has a finite second moment, and Brenier's theorem yields a unique $q\mux$-\textup{a.e.} optimal transport map $T_q\colon\RR^d\to\RR^d$ pushing $q\mux$ onto $\mux$ \cites{brenier1991polar,villani2021topics}. With the associated Brenier displacement $\xi_q(x) := x - T_q(x)$, we set
\begin{equation}\label{eq:COT}
    \COT(g) := \int j(x)\cdot\xi_q(x)\ud\mux,
\end{equation}
whenever the integral is well-defined; a sufficient condition is given in \cref{lem:imported} below.

Applying the Gaussian logarithmic Sobolev inequality of Gross~\cite{Gross1975} to the conditional density $h_x(v)$ and integrating against $q\ud\mux$ yields
\begin{equation} \label{eq:gausslsi}
     \Entv(g) \le \frac{1}{2}\Iv(g),
\end{equation}
which will be used repeatedly. The following lemma collects the estimates of~\cite{Lu2026}*{Section 3} that bound the kinetic energy of the velocity current $\int |j|^2/q \ud \mux $ and the Wasserstein corrector $\COT(g)$.

\begin{lemma}[\cite{Lu2026}*{Lemmas 3.1--3.2}]\label{lem:imported}
For every probability density $g$ on $\RR^{2d}$,
\begin{equation}\label{eq:imported-size}
    J(g) := \int \frac{\abs{j}^2}{q}\ud\mux \le 2\Entv(g),
\end{equation}
with the convention $\abs{j}^2/q := 0$ on $\{q = 0\}$.  If, in addition, $\Ent_\mu(g)<\infty$, then $\COT(g)$ is well-defined and we have
\begin{equation}\label{est:corrector}
    \abs{\COT(g)} \le \left(\frac2\rho\,\Ent_{\mux}(q)\,J(g)\right)^{1/2}  \le \rho^{-1/2}\,\Ent_\mu(g),
\end{equation}
and if moreover $\Iv(g)<\infty$,
\begin{equation}\label{eq:imported-cot-product}
    \abs{\COT(g)}
    \le \left(\frac2\rho\,\Ent_{\mux}(q)\,\Iv(g)\right)^{1/2},
\end{equation}
where $\rho > 0$ is the LSI constant in~\eqref{eq:lsi}.
\end{lemma}

Two further structural inequalities from \cite{Lu2026} will be used to control the time evolution of the corrector $\COT(g_t)$ along a path of probability densities $\{g_t\}_{t \in [0,T]}$ on $\mathbb{R}^{2d}$ (\cref{lem:forced-corrector}).

The first (\cref{lem:wasserstein-acceleration}) is a Wasserstein acceleration inequality, depending only on the continuity equation for the spatial marginal $q_t = \Piv g_t$. It is a variant of \cite{Lu2026}*{Lemma 4.1} in which the pointwise second-order flow expansion is replaced by a stability argument for Brenier maps, weakening the strong characteristics condition \cite{Lu2026}*{Definition 2.5~(R3)} to the boundedness condition \cref{def:regular-pair}~(A2) below, to accommodate the Gaussian-mollified approximation of \cref{app:recovery}, at the price of obtaining \eqref{eq:wasserstein-acceleration} only in the sense of distributions on $(0,T)$, rather than pointwise.

The second (\cref{lem:localized-stress}) is the localized stress estimate of \cite{Lu2026}*{Lemma 5.3}, a static per-cutoff bound on the centered stress tensor, independent of the underlying evolution of $g_t$. Its cutoff-limit consequence, \cref{lem:brenier-stress}, is a convenient global form used in \cref{lem:forced-corrector} and in the regular-path proof of \cref{thm:stlsi}. The recovery sequence constructed in \cref{app:recovery} for the finite-action space-time LSI instead applies only \cref{lem:localized-stress}, and the required cutoff limits are verified directly there.

We first introduce a regularity class of time-dependent density-current pairs, adapted from \cite{Lu2026}*{Definition 2.5} but retaining only the properties needed for the present argument.

\begin{definition}[Regular density-current pair]\label{def:regular-pair}
A pair $\{(q_t,j_t)\}_{t\in(0,T)}$ on $\RR^d$ is called a
\emph{regular density-current pair} if
\begin{equation*}
        q\in C^1((0,T);C^\infty(\RR^d)),
        \qquad
        j\in C^1((0,T);C^\infty(\RR^d;\RR^d)),
\end{equation*}
and the following conditions hold.
\begin{enumerate}
    \item[(A1)] For every $t\in(0,T)$, $q_t$ is a probability density with respect to $\mux$, and
    \begin{equation*}
        \inf_{(t,x)\in K\times\RR^d} q_t(x)>0
        \qquad \text{for every compact } K\subset(0,T).
     \end{equation*}
    Moreover, it holds that $\Ent_{\mux}(q_t)<\infty$ and
    $\int\frac{\abs{j_t}^2}{q_t}\ud\mux < \infty$.
    \item[(A2)] The continuity equation
    \begin{equation} \label{eq:continue}
        \partial_t q_t-\nabla_x^*j_t=0
    \end{equation}
    holds pointwise on $(0,T)\times\RR^d$. Letting
    \begin{equation} \label{eq:velovityanda}
        m_t=\frac{j_t}{q_t},
        \qquad
        a_t=\partial_t m_t+(m_t\cdot\nabla_x)m_t,
    \end{equation}
we assume that, for every compact $K\subset(0,T)$,
    \begin{equation} \label{eq:A2-bounds}
        \sup_{t\in K}\norm{m_t}_{L^\infty(\RR^d)}<\infty,
        \qquad
        \sup_{t\in K}\int\abs{a_t}^2q_t\ud\mux<\infty.
    \end{equation}
\end{enumerate}
\end{definition}

\begin{remark} \label{rem:assumption}
Condition~(A1) ensures that the Brenier map $T_{q_t}$ is well defined and that the displacement $\xi_{q_t} = x - T_{q_t}(x)$ belongs to $L^2(q_t\mux;\RR^d)$, by Talagrand's inequality \eqref{eq:talagrant}:
\begin{equation}\label{eq:boundxiq}
    \int q_t\abs{\xi_{q_t}}^2\ud\mux
    = W_2^2(q_t\mux,\mux)
    \le \frac{2}{\rho}\,\Ent_{\mux}(q_t).
\end{equation}
Condition~(A2) supplies the bounded velocity field $m_t$ and the $L^2(q_t\mux)$ acceleration bound on $a_t$ entering the Wasserstein acceleration inequality \cref{lem:wasserstein-acceleration}.

\end{remark}

\begin{lemma}[Weak form of \cite{Lu2026}*{Lemma 4.1}]\label{lem:wasserstein-acceleration}
Let $\{(q_t,j_t)\}_{t\in(0,T)}$ be a regular density-current pair in the sense of \cref{def:regular-pair}, and let $\xi_{q_t}(x)=x-T_{q_t}(x)$ be the Brenier displacement from $q_t\mux$ to $\mux$. Then, in the sense of distributions on $(0,T)$,
\begin{equation}\label{eq:wasserstein-acceleration}
    \frac{1}{2} \frac{\ud^2}{\ud t^2} W_2^2( q_t\mux,\mux)  =  \frac{\ud}{\ud t}\int j_t\cdot\xi_{q_t}\ud\mux
    \le
    \int\frac{\abs{j_t}^2}{q_t}\ud\mux
    +\int
    \left(
        \partial_t j_t
        -\nabla_x^*\!\left(\frac{j_t\otimes j_t}{q_t}\right)
    \right)\cdot\xi_{q_t}\ud\mux .
\end{equation}
\end{lemma}

With $m_t=j_t/q_t$ as in \eqref{eq:velovityanda}, the second integral of the right-hand side of \eqref{eq:wasserstein-acceleration} is the pairing of the Brenier displacement $\xi_{q_t}$ with the Eulerian acceleration field $a_t$ in \eqref{eq:velovityanda}:
\begin{equation*}
    \int \left( \partial_t j_t
        -\nabla_x^*\!\left(\frac{j_t\otimes j_t}{q_t}\right)
    \right)\cdot\xi_{q_t}\ud\mux
    =
    \int a_t\cdot\xi_{q_t}\ud(q_t\mux).
\end{equation*}
Indeed, using the continuity equation \eqref{eq:continue}, we obtain
\begin{equation*}
    \partial_t j_t-\nabla_x^*\!\left(\frac{j_t\otimes j_t}{q_t}\right)
    =\partial_t(q_t m_t)-\nabla_x^*(q_t m_t\otimes m_t)
    =q_t\bigl(\partial_t m_t+(m_t\cdot\nabla_x)m_t\bigr)=q_ta_t.
\end{equation*}
Moreover, it is finite by (A1), \eqref{eq:A2-bounds}, and \eqref{eq:boundxiq}:
\begin{equation*}
    \Abs{\int a_t\cdot\xi_{q_t}\ud q_t\mux}
    \le\norm{a_t}_{L^2( q_t\mux)}\,\norm{\xi_{q_t}}_{L^2( q_t\mux)}
    =\norm{a_t}_{L^2( q_t\mux)}\,W_2( q_t\mux,\mux).
\end{equation*}

As mentioned above, under stronger regularity assumptions assumed in \cite{Lu2026}, the inequality \eqref{eq:wasserstein-acceleration} holds pointwise in time; the distributional formulation above allows us to work under the weaker condition~\textup{(A2)}. We defer the proof of \cref{lem:wasserstein-acceleration} to
\cref{app:wasserstein-acceleration}.

\begin{lemma}[Localized stress estimate \cite{Lu2026}*{(5.21)}]\label{lem:localized-stress} Suppose \cref{gm:ass:lsi}. Let $g$ be a probability density on $\RR^{2d}$ with respect to $\mu$, and set $q=\Piv g$ and $j=\Piv(vg)$. Assume that $q$ is strictly positive and belongs to $C^\infty(\RR^d)$, and write $g(x,v)=q(x)h_x(v)$. Assume that the conditional
mean and covariance
\begin{equation*}
    m(x)=\int v\,h_x(v)\ud\kv,
    \qquad
    \Sigma_g(x)=\int (v-m(x))\otimes(v-m(x))h_x(v)\ud\kv
\end{equation*}
exist for \textup{a.e.} $x$. Define the centered stress tensor by
\begin{equation} \label{def:tensor}
    \Theta_g
    :=
    \Piv(v\otimes v\,g)-\frac{j\otimes j}{q}-qI_d = q(\Sigma_g-I_d).
\end{equation}
Assume that $\Ent_{\mux}(q)$ and $\Entv(g)$ are finite, and that $\Theta_g\in C^1_{\mathrm{loc}}(\RR^d;\RR^{d\times d})$. Let $\xi_q=x-T_q$ be the Brenier displacement from $q\mux$ to $\mux$. Then, for any $\chi\in C_c^\infty(\RR^d)$ with $0\le\chi\le1$ and $0<\eta<1/2$,
\begin{equation}\label{eq:localized-stress}
    \int\xi_q\cdot\nabla_x^{*}(\chi\Theta_g)\ud\mux
    -\int\chi\,\nabla_xq\cdot\xi_q\ud\mux
    \;\le\;
    \eta^{-1}\Entv(g)
    -\int\chi\,q\log q\ud\mux
    +\int q\,\nabla\chi\cdot\xi_q\ud\mux .
\end{equation}
\end{lemma}

\begin{lemma}[Brenier stress estimate]\label{lem:brenier-stress}
Under the assumptions of \cref{lem:localized-stress}, assume in addition that there exists a sequence of cutoff functions $\chi_R\in C_c^\infty(\RR^d)$, $0\le\chi_R\le1$, $\chi_R\uparrow1$, $\abs{\nabla\chi_R}\le C/R$, such that $\xi_q\cdot\nabla_x q,\ \xi_q\cdot\nabla_x^*\Theta_g\in L^1(\mux)$ and, as $R\to\infty$,
\begin{equation} \label{eq:limitcutoff}
    \begin{aligned}
    \int q\,\nabla\chi_R\cdot\xi_q\ud\mux
        &\;\longrightarrow\;0,\\
    \int \xi_q\cdot\nabla_x^*(\chi_R\Theta_g)\ud\mux
        &\;\longrightarrow\;
        \int \xi_q\cdot\nabla_x^*\Theta_g\ud\mux.
\end{aligned}
\end{equation}
Then, for every $0<\eta<1/2$,
\begin{equation}\label{eq:brenier-stress}
    -\int\nabla_x q\cdot\xi_q\ud\mux
    +
    \int\xi_q\cdot\nabla_x^*\Theta_g\ud\mux
    \le
    \eta^{-1}\Entv(g)-\Ent_{\mux}(q).
\end{equation}
\end{lemma}

\begin{proof}
We apply \cref{lem:localized-stress} with $\chi=\chi_R$ and let $R\to\infty$. The second term in the left-hand side of \eqref{eq:localized-stress} converges by dominated convergence and $\xi_q\cdot\nabla q\in L^1(\mu_x)$, while the first converges by the assumed cutoff limit. Moreover, for the right-hand side, we have $\int\chi_R\,q\log q\ud\mux\to\Ent_{\mux}(q)$ by again dominated convergence. Indeed, the function $s\mapsto s\log s$ attains its global minimum $-\e^{-1}$ at $s=\e^{-1}$, so the negative part satisfies $(q\log q)_-\le\e^{-1}$ pointwise, and hence $\int(q\log q)_-\ud\mux\le\e^{-1}$. By $\Ent_{\mux}(q)=\int(q\log q)_+\ud\mux-\int(q\log q)_-\ud\mux<\infty$, the positive part $(q\log q)_+$ is integrable as well.
\end{proof}

\section{Space-time logarithmic Sobolev inequality}\label{sec:stlsi}

This section is devoted to the space-time logarithmic Sobolev inequality for the underdamped Langevin dynamics. To this end, we first introduce the time-augmented reference measure:
\begin{equation*}
    \ud\muT = \frac{1}{T}\ud t\ud\mu
\end{equation*}
on $[0,T]\times\RR^{2d}$. For a path of probability densities $\{g_t\}_{t \in [0,T]}$ on $\RR^{2d}$, the time-averaged entropy and time-averaged velocity Fisher information are defined by
\begin{equation*}
    \Ent_{\muT}(g) = \int g \log g \ud\muT,
    \qquad
    I_{v,T}(g) = \int \frac{\abs{\nabla_v g}^2}{g}\ud\muT.
\end{equation*}

Given a probability density $g$ with respect to $\mu$ and $r\in\mathcal D'(\RR^{2d})$, we define the weighted
velocity negative norm of $r$ with respect to $g$ by
\begin{equation}\label{eq:negative-norm}
    \norm{r}_{-1, g}^2
    = \sup_{\zeta \in C_c^\infty(\RR^{2d})}
    \left\{
        2 \langle r, \zeta \rangle - \int g \abs{\nabla_v \zeta}^2\ud\mu
    \right\},
\end{equation}
where $\langle\cdot,\cdot\rangle$ denotes the distributional pairing (which reduces to $\int r \zeta\ud\mu$ when $r$ is a function). Equivalently, by Legendre duality,
\begin{equation}\label{eq:negative-representation}
    \norm{r}_{-1, g}^2
    = \inf \left\{
        \int g \abs{A}^2\ud\mu \;:\; A \in L^2(g\ud\mu;\, \RR^d),\ r = \nabla_v^{*}(g A)
    \right\},
\end{equation}
with the convention that the infimum is $+\infty$ if no such velocity field $A$ exists.

\begin{remark}
The norm \eqref{eq:negative-norm} is the velocity-fiber analogue of the Wasserstein tangent metric \cite{ambrosio2005gradient}*{Section 8}. To see this, we consider the sufficiently regular data for simplicity. We write $g(x,v) = q(x)h_x(v)$, and on $\{q>0\}$ we define
\begin{equation*}
    \dot h_x(v) := r(x,v)/q(x),
\end{equation*}
with $\int \dot h_x\ud\kv=0$ for $\mu_x$-\textup{a.e.} $x$, when $\norm{r}_{-1, g} < \infty$, as shown in \cref{lem:negative-norm-zero-marginal} below. Then, the representation \eqref{eq:negative-representation} disintegrates as
\begin{equation*}
    \norm{r}_{-1,g}^2
    =
    \int_{\RR^d}
    q(x)\,
    \norm{\dot h_x}_{-1,h_x}^2
    \ud\mux(x),
\end{equation*}
where
\begin{equation*}
    \norm{\dot h}_{-1,h}^2
    :=
    \inf_{\dot h=\nabla_v^*(ha)}
    \int h\,\abs{a}^2\ud\kv
    =
    \sup_{\phi\in C_c^\infty(\RR^d)}
    \left\{
        2\int \dot h\,\phi\ud\kv
        -
        \int h\,\abs{\nabla_v\phi}^2\ud\kv
    \right\}.
\end{equation*}
Thus $\norm{r}_{-1,g}^2$ is nothing else than the average over $q\mux$ of the squared Wasserstein tangent norm at the conditional law $h_x\kv$.
\end{remark}

\begin{lemma}
\label{lem:negative-norm-zero-marginal}
Let $g$ be a probability density with respect to $\mu=\mux\otimes\kappa$,
and let $r\in\mathcal D'(\RR^{2d})$ satisfy
\begin{equation*}
        \|r\|_{-1,g}<\infty.
\end{equation*}
Then $r$ has vanishing spatial marginal:
\begin{equation*}
     \Piv r=0
    \qquad\text{in }\mathcal D'(\RR^d).
\end{equation*}
In particular, let $\{g_t\}_{t\in[0,T]}$ be a path such that $q_t=\Piv g_t$, $j_t=\Piv(vg_t)$, and
$\partial_tq_t-\nabla_x^*j_t$ are well-defined in the sense of distributions. If
\begin{equation*}
        \|(\partial_t+ \La)g_t\|_{-1,g_t}<\infty, \qquad\text{for \textup{a.e.} }t\in(0,T),
\end{equation*}
then
\begin{equation*}
        \partial_tq_t-\nabla_x^*j_t=0
\end{equation*}
in $\mathcal D'((0,T)\times\RR^d)$.
\end{lemma}

\begin{proof}
By~\eqref{eq:negative-representation}, for given $r$ with $\|r\|_{-1,g}<\infty$, choose
$A\in L^2(g \ud \mu;\RR^d)$ such that $\int g|A|^2 \ud \mu < \infty$ and
$$
    r=\nabla_v^*(gA)
$$
in the distributional sense. Since $g$ is a probability density, Cauchy--Schwarz gives
$$
    \int |gA| \ud \mu
    \le
    \left(\int g|A|^2 \ud \mu\right)^{1/2}
    \left(\int g \ud \mu\right)^{1/2}
    <\infty.
$$
Let $\phi(x)\in C_c^\infty(\RR^d)$, and let $\eta_R(v)\in C_c^\infty(\RR^d)$ be a standard velocity cutoff satisfying $\eta_R\to1$ pointwise and $\|\nabla_v\eta_R\|_\infty\to 0$. Then, as $R \to \infty$,
\begin{equation*}
     \langle r,\phi\eta_R\rangle
    =
    \int gA\cdot\nabla_v(\phi\eta_R) \ud \mu
    =
    \int \phi\,gA\cdot\nabla_v\eta_R \ud \mu \to 0\,.
\end{equation*}
Thus $\langle \Piv r,\phi\rangle := \lim_{R\to\infty}\langle r,\phi(x)\eta_R(v)\rangle = 0$, that is, $\Piv r = 0$.

We now consider $r_t := (\partial_t + \mathcal{L}_a)g_t$. From the definitions of $q_t$ and $j_t$, as well as $\La$, we readily have
\begin{equation*}
   0 = \Piv r_t  = \Piv\bigl((\partial_t+\La)g_t\bigr)
    =  \partial_tq_t-\nabla_x^*j_t,
\end{equation*}
by a direct computation. The proof is complete.
\end{proof}

\begin{remark}[The negative norm of $\Ls g$]\label{rem:Ls-negative-norm}
For a sufficiently regular positive density $g$ with $\Iv(g)<\infty$, one can compute
\begin{equation} \label{eq:identifylsg}
    \Ls g  = - \nabla_v^{*}\nabla_v g  = \nabla_v^{*}\bigl(-g\nabla_v\log g\bigr).
\end{equation}
Hence the representation \eqref{eq:negative-representation} of $\norm{\cdot}_{-1,g}$ gives
\begin{equation*}
    \norm{\Ls g}_{-1,g}^2 \le \int g\abs{\nabla_v\log g}^2\ud\mu = \Iv(g).
\end{equation*}
In fact equality holds in the regular class used below: the upper bound follows from the above admissible current, while the matching lower bound follows from the dual formula \eqref{eq:negative-norm} by testing with compactly supported cutoffs applied to bounded smooth truncations of $-\log g$. Thus
\begin{equation*}
    \norm{\Ls g}_{-1,g}^2=\Iv(g).
\end{equation*}
Consequently, along an underdamped Langevin trajectory satisfying
$(\partial_t+\La)g_t=\gamma\Ls g_t$, one has
\begin{equation*}
\norm{(\partial_t+\La)g_t}_{-1,g_t}^2 = \gamma^2\Iv(g_t).
\end{equation*}
In contrast, for a controlled underdamped Langevin trajectory \eqref{eq:controlled-kolmogorov}, we only have the inequality for the negative norm bound; see \eqref{eq:negativenormbound} below.
\end{remark}

\begin{lemma}[Entropy dissipation for regular finite-action paths]
\label{lem:entropy-chain-rule}
Suppose that \cref{gm:ass:lsi,gm:ass:standing} hold. Let $I\subset\RR$ be an open interval, and
\begin{equation*}
    h\in C^1\bigl(I;C^\infty(\RR^{2d})\bigr),
    \qquad h>0, \qquad \int h_t\ud\mu=1 \quad\text{for every }t\in I,
\end{equation*}
and let $W:I\times\RR^{2d}\to\RR^d$ be measurable. Assume that
\begin{equation}\label{eq:entropy-chain-graph}
    (\partial_t+\La)h_t=\nabla_v^*W_t
    \qquad\text{in }\mathcal D'(I\times\RR^{2d})
\end{equation}
and that
\begin{equation}\label{eq:entropy-chain-integrability}
    \Ent_\mu(h_t),\
    \Iv(h_t),\
    \int\frac{\abs{W_t}^2}{h_t}\ud\mu \in L^1_{\rm loc}(I).
\end{equation}
Then $t\mapsto\Ent_\mu(h_t)$ has a representative in $W^{1,1}_{\rm loc}(I)$, and
\begin{equation}\label{eq:entropy-chain-rule}
    \frac{\ud}{\ud t}\Ent_\mu(h_t) = \int W_t\cdot\nabla_v\log h_t\ud\mu  \qquad\text{for \textup{a.e.}\ }t\in I.
\end{equation}
\end{lemma}

\begin{proof}
By \cref{gm:ass:standing,gm:lem:exp-integrability}\textup{(i)}, the potential satisfies $U\ge0$ and is coercive. Hence the Hamiltonian $H(x,v):=U(x)+\abs v^2/2$ is nonnegative and has compact sublevel sets. Set $\Phi(r):=r\log r-r+1$ for $r>0$, with $\Phi(0):=1$. For $N\ge2$, let $\Phi_N$ be the convex $C^1$ function obtained by retaining $\Phi$ on $[N^{-1},N]$ and extending it affinely on $[0,N^{-1}]$ and $[N,\infty)$:
\begin{equation*}
    \Phi_N(r) =
    \begin{cases}
        1-N^{-1}-(\log N)r, & 0\le r\le N^{-1},\\  r\log r-r+1,     &N^{-1}<r<N,\\
       (\log N)r-N+1, & r\ge N.
    \end{cases}
\end{equation*}
Thus, there holds
\begin{equation}\label{eq:entropy-value-truncation}
    0\le\Phi_N\uparrow\Phi,\qquad
    \abs{\Phi_N'}\le\log N,\qquad
    \Phi_N''(r)=\frac1r\,
    \mathbf1_{\{N^{-1}<r<N\}}
    \quad\text{for \textup{a.e.}\ }r>0.
\end{equation}

We combine this truncation with a cutoff adapted to the Hamiltonian. Choose a nonincreasing $\chi\in C_c^\infty([0,\infty))$ such that $0\le\chi\le1$, $\chi=1$ on $[0,1]$, and $\operatorname{supp}\chi\subset[0,2]$, and define $\chi_R:=\chi(H/R)$. Then $\chi_R$ is compactly supported, $\chi_R\uparrow1$, and
\begin{equation}\label{eq:entropy-chain-cutoff}
    \La\chi_R=0,\qquad
    \nabla_v\chi_R=\frac1R\chi'(H/R)v.
\end{equation}
We now fix a compact interval $K\subset I$, and we test \eqref{eq:entropy-chain-graph} in the space variables with $\chi_R\Phi_N'(h_t)$, using smooth
convex approximations of $\Phi_N$. This gives, in $\mathcal D'(K)$,
\begin{align*}
    \frac{\ud}{\ud t}\int\chi_R\Phi_N(h_t)\ud\mu
    &=  -\int\chi_R\Phi_N'(h_t)\La h_t\ud\mu
 +\int\chi_R\Phi_N'(h_t)\nabla_v^*W_t\ud\mu.
\end{align*}
Because $\La$ is a first-order skew-adjoint operator, the chain rule and \eqref{eq:entropy-chain-cutoff} imply
\begin{equation*}
    -\int\chi_R\Phi_N'(h_t)\La h_t\ud\mu
    = -\int\chi_R\La\bigl(\Phi_N(h_t)\bigr)\ud\mu  = \int\Phi_N(h_t)\La\chi_R\ud\mu = 0.
\end{equation*}
For the rest term, the integration by parts gives
\begin{align*}
    \int\chi_R\Phi_N'(h_t)\nabla_v^*W_t\ud\mu
    &=\int W_t\cdot\nabla_v\bigl(\chi_R\Phi_N'(h_t)\bigr)\ud\mu\\
    &= \int\chi_R\Phi_N''(h_t)W_t\cdot\nabla_vh_t\ud\mu +\int\Phi_N'(h_t)W_t\cdot\nabla_v\chi_R\ud\mu.
\end{align*}
Combining these identities yields, in $\mathcal D'(K)$,
\begin{align}
    \frac{\ud}{\ud t}\int\chi_R\Phi_N(h_t)\ud\mu
    &= \int\chi_R\Phi_N''(h_t)\,  W_t\cdot\nabla_vh_t\ud\mu + \int\Phi_N'(h_t)\,W_t\cdot\nabla_v\chi_R\ud\mu.
    \label{eq:entropy-chain-localized}
\end{align}

We now let $R\to\infty$. By \eqref{eq:entropy-value-truncation}, we find
\begin{equation*}
    \abs{\Phi_N''(h_t)W_t\cdot\nabla_vh_t}
    \le
    \left(\frac{\abs{W_t}^2}{h_t}\right)^{1/2}
    \left(\frac{\abs{\nabla_vh_t}^2}{h_t}\right)^{1/2},
\end{equation*}
which is integrable on $K\times\RR^{2d}$ by
\eqref{eq:entropy-chain-integrability}. For the second cutoff error term, the support of $\nabla_v\chi_R$ is contained in $\{R\le H\le2R\}$, where $\abs v^2\le4R$. Hence, the Cauchy--Schwarz inequality gives
\begin{align*}
&\int_K \Abs{\int\Phi_N'(h_t)\,W_t\cdot\nabla_v\chi_R\ud\mu}\ud t \le \frac{2\norm{\chi'}_\infty\log N}{\sqrt R}\, \abs{K}^{1/2}
    \left(\int_K\int\frac{\abs{W_t}^2}{h_t}\ud\mu\ud t \right)^{1/2} \to 0.
\end{align*}
Since $\Phi_N$ has at most linear growth, passing to the limit $R \to \infty$ in
\eqref{eq:entropy-chain-localized} yields, in $\mathcal D'(K)$,
\begin{equation}\label{eq:entropy-chain-truncated}
    \frac{\ud}{\ud t}\int\Phi_N(h_t)\ud\mu  =
    \int\Phi_N''(h_t)\,W_t\cdot\nabla_vh_t\ud\mu.
\end{equation}

Finally, the right-hand side of \eqref{eq:entropy-chain-truncated} converges in $L^1(K)$, as $N\to\infty$, to $\int W_t\cdot\nabla_v\log h_t\ud\mu$, again by \eqref{eq:entropy-value-truncation} and $
(|W_t|^2/h_t)^{1/2}(|\nabla_vh_t|^2/h_t)^{1/2} \in L^1(K \times \RR^{2d})$. The first condition in \eqref{eq:entropy-chain-integrability} provides a time $t_0\in K$ with $\Ent_\mu(h_{t_0})<\infty$. Integrating \eqref{eq:entropy-chain-truncated} from $t_0$ to $t$, then using monotone convergence and $\int\Phi(h_t)\ud\mu=\Ent_\mu(h_t)$ proves that the entropy $\Ent_\mu(h_t)$ is finite and absolutely continuous on $K$ and satisfies \eqref{eq:entropy-chain-rule}. Since $K\subset I$ was arbitrary, the proof is complete.
\end{proof}

\subsection{Space-time log-Sobolev inequality for regular paths}
In this section, we establish the space-time LSI, \cref{thm:stlsi}, in three steps.

The first step, \cref{lem:forced-corrector}, is a controlled version of the entropy-current calculation in~\cite{Lu2026}*{Section 6}: it bounds the evolution of the corrector $\COT(g_t)$ along a path driven by an additional velocity forcing. The second step, \cref{prop:controlled}, combines this corrector inequality with the entropy dissipation identity, \cref{lem:entropy-chain-rule}, for the controlled equation to yield a differential inequality for a modified entropy, at the cost of a quadratic penalty in the forcing. The third step integrates this differential inequality on interior time windows and lets the windows increase to $(0,T)$, thereby proving the space-time LSI \eqref{eq:stlsi}.

\begin{lemma}[Controlled corrector differential inequality]\label{lem:forced-corrector}
Let $\{g_t\}_{t\in(0,T)}$ be a path of probability densities with respect to $\mu$ satisfying
\begin{equation}\label{eq:reg-gt-smooth}
    g\in C^1\bigl((0,T);\,C^\infty(\RR^{2d})\bigr)
    \quad\text{and}\quad
    \inf_{(t,x,v)\in K\times\RR^{2d}} g_t(x,v)>0
    \quad\text{for every compact }K\subset(0,T),
\end{equation}
and let $u:(0,T)\times\RR^{2d}\to\RR^d$ be a measurable velocity-control field with finite action
\begin{equation}\label{eq:finite-action-u}
    \int_0^T\int g_t\abs{u_t}^2\ud\mu\ud t<\infty.
\end{equation}
Assume that $g_t$ satisfies the controlled forward Kolmogorov equation
\begin{equation}\label{eq:controlled-kolmogorov}
    (\partial_t+\La)g_t  = \gamma\Ls g_t+\nabla_v^{*}(g_tu_t)
\end{equation}
in the sense of distributions on  $(0,T)\times\RR^{2d}$, and there holds
\begin{equation} \label{eq:controlled-local-integrability}
     \Ent_\mu(g_t), \ \Iv(g_t) \in L^1_{\rm loc}((0,T)).
\end{equation}
Assume that $\{(q_t,j_t)= (\Piv g_t, \Piv(vg_t))\}_{t\in(0,T)}$ is a regular density-current pair in the sense of \cref{def:regular-pair}.  Assume moreover that, for \textup{a.e.} $t\in(0,T)$, the centered stress tensor $\Theta_t$, defined by \eqref{def:tensor}, satisfies all the assumptions of \cref{lem:brenier-stress}, including the cutoff limits \eqref{eq:limitcutoff}.

Define the spatial source induced by the velocity control:
\begin{equation}\label{eq:Bt-def}
    B_t(x) := \Piv(g_tu_t)(x) = \int g_t(x,v)u_t(x,v)\,\kv(\ud v).
\end{equation}
Then, for \textup{a.e.}\ $t\in(0,T)$, $B_t\in L^2(q_t^{-1}\mux;\RR^d)$ with
\begin{equation*}
    \int\frac{\abs{B_t}^2}{q_t}\ud\mux \le \int g_t\abs{u_t}^2\ud\mu,
\end{equation*}
and hence
\begin{equation}\label{eq:Bt-xi-bound}
    \Abs{\int B_t\cdot \xi_{q_t}\ud\mux} \le \left(\frac2\rho\,\Ent_{\mux}(q_t)\right)^{1/2}
    \left(\int g_t\abs{u_t}^2\ud\mu\right)^{1/2}.
\end{equation}
Moreover, the controlled corrector differential inequality
\begin{equation}\label{eq:forced-cot-diff}
    \frac{\ud}{\ud t}\COT(g_t) \le -\Ent_{\mux}(q_t) -\gamma\COT(g_t) + 3\Iv(g_t) + \int B_t\cdot\xi_{q_t}\ud\mux
\end{equation}
holds in the sense of distributions on $(0,T)$.
\end{lemma}

\begin{proof}
All identities below are understood weakly in the spatial variables and in the sense of distributions in time; the estimates on individual time slices hold for \textup{a.e.}\ $t$. The local integrability assumption \eqref{eq:controlled-local-integrability}, together with \eqref{est:corrector}, the finite-action assumption \eqref{eq:finite-action-u}, and the source estimate \eqref{eq:Bt-xi-bound}, ensures that all terms in \eqref{eq:forced-cot-diff} are locally integrable
functions of $t$.

We first record an elementary bound on $B_t(x)$ defined in \eqref{eq:Bt-def}. A direct application of Jensen's inequality gives, for \textup{a.e.}\ $t$,
\begin{equation}\label{eq:Bt-action-proof}
    |B_t(x)|^2 = q_t(x)^2
    \left|  \int u_t(x,v) h_{t,x}(v)\,\kv(\ud v)
    \right|^2 \le q_t(x) \int g_t(x,v)|u_t(x,v)|^2\,\kv(\ud v).
\end{equation}
We then have
\begin{equation}\label{eq:Bt-L2q-proof}
    \int \frac{|B_t|^2}{q_t}\ud\mux  \le
    \int g_t |u_t|^2\ud\mu < \infty.
\end{equation}
In particular, the pairing with the Brenier displacement $\xi_{q_t}$ is well-defined by
\begin{equation*}
    \left|\int B_t\cdot \xi_{q_t}\ud\mux\right|
    \le
    \left(\int \frac{|B_t|^2}{q_t}\ud\mux\right)^{1/2}
    \left(\int q_t |\xi_{q_t}|^2\ud\mux\right)^{1/2},
\end{equation*}
where the second term is bounded by \eqref{eq:boundxiq}:
\begin{equation*}
      \int q_t\abs{\xi_{q_t}}^2\ud\mux
        \le \frac2\rho\,\Ent_{\mux}(q_t).
\end{equation*}
Combining the last two estimates with \eqref{eq:Bt-L2q-proof} gives \eqref{eq:Bt-xi-bound}.

We next show that the velocity control does not change the spatial continuity
equation. Set $r_t:=(\partial_t+\La)g_t$. By the controlled equation \eqref{eq:controlled-kolmogorov} and the identity \eqref{eq:identifylsg}, we have
\begin{equation*}
    r_t = \nabla_v^{*}\bigl(g_t(u_t-\gamma\nabla_v\log g_t)\bigr).
\end{equation*}
Hence the representation formula \eqref{eq:negative-representation} gives
\begin{equation} \label{eq:negativenormbound}
    \norm{(\partial_t+\La)g_t}_{-1,g_t}^2 \le \int g_t\abs{u_t-\gamma\nabla_v\log g_t}^2\ud\mu
    \le  2 \int g_t\abs{u_t}^2\ud\mu+2\gamma^2\Iv(g_t) < \infty
\end{equation}
for \textup{a.e.}\ $t$. By \cref{lem:negative-norm-zero-marginal}, this implies
\begin{equation} \label{eq:qt-eqn-controlled}
    \Piv\bigl((\partial_t+\La)g_t\bigr) = \partial_t q_t-\nabla_x^*j_t =  0.
\end{equation}

We now compute the equation for the first velocity moment $j_t=\Piv(vg_t)$. For the control term, testing against a smooth vector field $a=a(x)$ gives
\begin{equation} \label{eq:sourcefirst}
    \begin{aligned}
         \int a(x)\cdot
    \Piv\bigl(v\,\nabla_v^{*}(g_tu_t)\bigr)\ud\mux   &=
    \int (v\cdot a(x))\,\nabla_v^{*}(g_tu_t)\ud\mu \\
    &= \int g_tu_t\cdot a(x)\ud\mu  =  \int a(x)\cdot B_t(x)\ud\mux,
    \end{aligned}
\end{equation}
by a standard velocity cutoff approximation as in the proof of \cref{lem:negative-norm-zero-marginal}.
Thus, the control contributes a source term $B_t$ to the first-moment equation. The remaining terms follow from the same computation as in the uncontrolled case \cite{Lu2026}*{(4.3)}. It follows that
\begin{equation}\label{eq:jt-eqn-controlled}
    \partial_t j_t  =  -\nabla_x q_t +\nabla_x^{*}\!\left(\frac{j_t\otimes j_t}{q_t}\right)
    +\nabla_x^{*}\Theta_t  -\gamma j_t +B_t.
\end{equation}

Since the marginal equation~\eqref{eq:qt-eqn-controlled} is unchanged by the velocity control, the Wasserstein acceleration inequality~\eqref{eq:wasserstein-acceleration} applies to $(q_t,j_t)$. Substituting~\eqref{eq:jt-eqn-controlled} into it and recalling $\COT(g_t)=\int j_t\cdot \xi_{q_t}\ud\mux$, we obtain
\begin{align*}
    \frac{\ud}{\ud t}\COT(g_t)
    &\le
    \int\frac{\abs{j_t}^2}{q_t}\ud\mux
    -\int\nabla_x q_t\cdot\xi_{q_t}\ud\mux
    +\int\xi_{q_t}\cdot\nabla_x^*\Theta_t\ud\mux
    -\gamma\int j_t\cdot\xi_{q_t}\ud\mux
    +\int B_t\cdot\xi_{q_t}\ud\mux .
\end{align*}
The Brenier stress estimate \eqref{eq:brenier-stress}, applied with
$\eta=1/4$, gives
\begin{equation*}
    - \int\nabla_x q_t\cdot\xi_{q_t}\ud\mux
    + \int\xi_{q_t}\cdot\nabla_x^*\Theta_t\ud\mux \le 4\Entv(g_t)-\Ent_{\mux}(q_t).
\end{equation*}
The remaining kinetic-current term satisfies
\begin{equation*}
    \int\frac{\abs{j_t}^2}{q_t}\ud\mux+4\Entv(g_t) \le  6\Entv(g_t) \le  3\Iv(g_t),
\end{equation*}
by \eqref{eq:imported-size} and the Gaussian LSI \eqref{eq:gausslsi}, while the friction term is exactly the corrector:
\begin{equation*}
    -\gamma \int j_t\cdot \xi_{q_t}\ud\mux
    =
    -\gamma \COT(g_t).
\end{equation*}
Consequently, the controlled corrector inequality is
\begin{equation*}
    \frac{\ud}{\ud t}\COT(g_t)
    \le
    -\Ent_{\mux}(q_t)  -\gamma \COT(g_t) + 3\Iv(g_t)  +\int B_t\cdot \xi_{q_t}\ud\mux ,
\end{equation*}
in the sense of distributions on $(0,T)$.
\end{proof}

The estimate \eqref{eq:forced-cot-diff} differs from the uncontrolled corrector inequality only by the source pairing $\int B_t\cdot\xi_{q_t}\ud\mux$, where $B_t=\Piv(g_tu_t)$. The next proposition combines this corrector estimate with the entropy dissipation identity (see \eqref{eq:controlled-entropy-diff}) for the controlled underdamped Langevin
equation, yielding a differential inequality (see \eqref{eq:pre-absorb-controlled}) for the modified entropy functional of~\cite{Lu2026},
\begin{equation}\label{def:modifiedentropy}
    \mathcal H_\theta(g_t):=\Ent_\mu(g_t)+\theta\sqrt\rho\,\COT(g_t).
\end{equation}
With the hypocoercive weight $\theta\sqrt\rho$ and the friction $\gamma=\Gamma\sqrt\rho$ chosen appropriately, the spatial entropy $\Ent_{\mux}(q_t)$ and the velocity Fisher information $\Iv(g_t)$ absorb the corrector contribution $-\theta\sqrt\rho\,\gamma\,\COT(g_t)$, while the terms created by the control are bounded by the quadratic action $\rho^{-1/2}\int g_t\abs{u_t}^2\ud\mu$.

\begin{proposition}[Controlled modified entropy estimate]\label{prop:controlled}
Suppose that \cref{gm:ass:lsi,gm:ass:standing} hold. Fix $\Gamma>0$ and set
\begin{equation}\label{eq:gammaprop}
    \gamma=\Gamma\sqrt\rho,
    \qquad
    \theta=\frac{1}{8}\min\bigl\{\Gamma,\Gamma^{-1}\bigr\}.
\end{equation}
Let $(g_t,u_t)$ satisfy the assumptions of \cref{lem:forced-corrector}, and $\mathcal H_\theta(g_t)$ be the modified entropy \eqref{def:modifiedentropy}.
Then, in the sense of distributions on $(0,T)$,
\begin{equation}\label{eq:controlled-diff}
    \frac{\ud}{\ud t}\mathcal H_\theta(g_t)
    \le
    -\frac{\theta}{2}\sqrt\rho\,\Ent_\mu(g_t)
    +C_\Gamma\,\rho^{-1/2}\int g_t\abs{u_t}^2\ud\mu,
\end{equation}
where the constant $C_\Gamma:=\Gamma^{-1}+2\theta$ depends only on $\Gamma$.
\end{proposition}

\begin{proof}
The estimates below are time-slice estimates for \textup{a.e.}\ $t$. Whenever a time derivative appears, the corresponding identity or inequality is understood in the sense of distributions on $(0,T)$, \textup{i.e.},~tested against nonnegative $\varphi\in C_c^\infty((0,T))$. The assumptions \eqref{eq:finite-action-u}  and \eqref{eq:controlled-local-integrability} ensure the required local time integrability.

First, \cref{lem:forced-corrector} applied with $\gamma=\Gamma\sqrt\rho$ from \eqref{eq:gammaprop} gives the controlled corrector inequality
\begin{equation}\label{eq:controlled-cot-diff}
    \frac{\ud}{\ud t}\COT(g_t)
    \le
    -\Ent_{\mux}(q_t)-\gamma\COT(g_t)+3\Iv(g_t)
    +\int B_t\cdot\xi_{q_t}\ud\mux.
\end{equation}
To obtain the entropy dissipation identity, rewrite
\eqref{eq:controlled-kolmogorov} as
\begin{equation*}
    (\partial_t+\La)g_t=\nabla_v^* W_t,
    \qquad W_t :=g_tu_t-\gamma\nabla_vg_t
    = g_t\bigl(u_t-\gamma\nabla_v\log g_t\bigr),
\end{equation*}
where $\int\frac{\abs{ W_t}^2}{g_t}\ud\mu$ is locally integrable in time. Thus, \cref{lem:entropy-chain-rule}, together with \eqref{eq:controlled-local-integrability}, gives
\begin{equation}\label{eq:controlled-entropy-diff}
    \frac{\ud}{\ud t}\Ent_\mu(g_t)
    =-\gamma\Iv(g_t)
    +\int g_tu_t\cdot\nabla_v\log g_t\ud\mu.
\end{equation}

We next estimate the two $u_t$-linear terms in \eqref{eq:controlled-cot-diff} and \eqref{eq:controlled-entropy-diff} arising from the control. For the entropy cross term in \eqref{eq:controlled-entropy-diff}, Young's inequality with weight $\gamma/2$ gives
\begin{equation}\label{eq:entropy-control-bound}
    \int g_tu_t\cdot\nabla_v\log g_t\ud\mu
    \le
    \frac{\gamma}{4}\Iv(g_t)
    +\frac{1}{\gamma}\int g_t\abs{u_t}^2\ud\mu.
\end{equation}
On the other hand, for the source pairing in \eqref{eq:controlled-cot-diff}, combining the bound \eqref{eq:Bt-xi-bound} from \cref{lem:forced-corrector} with Young's inequality yields
\begin{equation}\label{eq:cot-control-bound}
    \Abs{\int B_t\cdot\xi_{q_t}\ud\mux}
    \le
    \frac{1}{4}\Ent_{\mux}(q_t)
    +\frac{2}{\rho}\int g_t\abs{u_t}^2\ud\mu.
\end{equation}

We now bound the time derivative of the modified entropy functional $\mathcal{H}_\theta(g_t)$ in \eqref{def:modifiedentropy}. For this, it suffices to combine \eqref{eq:controlled-cot-diff} and \eqref{eq:controlled-entropy-diff}, and substitute the estimates \eqref{eq:entropy-control-bound} and \eqref{eq:cot-control-bound}. A straightforward computation, using $\gamma=\Gamma\sqrt\rho$, leads to
\begin{equation}\label{eq:pre-absorb-controlled}
    \frac{\ud}{\ud t}\mathcal H_\theta(g_t)
    \le  -\frac{3\theta}{4}\sqrt\rho\,\Ent_{\mux}(q_t) - \theta\Gamma\rho\,\COT(g_t) - \left(\frac{3\Gamma}{4}-3\theta\right)\sqrt\rho\,\Iv(g_t) + C_\Gamma\,\rho^{-1/2}\int g_t\abs{u_t}^2\ud\mu,
\end{equation}
with $C_\Gamma=\Gamma^{-1}+2\theta$.

It remains to estimate the indefinite-sign term $-\theta\Gamma\rho\,\COT(g_t)$ in \eqref{eq:pre-absorb-controlled}. By \eqref{eq:imported-cot-product} in \cref{lem:imported} and Young's inequality again, we have
\begin{equation}\label{eq:cot-absorption}
    \theta\Gamma\rho\,\abs{\COT(g_t)}
    \le
    \theta\Gamma\sqrt{2\rho\,\Ent_{\mux}(q_t)\,\Iv(g_t)}
    \le
    \frac{\theta}{4}\sqrt\rho\,\Ent_{\mux}(q_t)
    +2\theta\Gamma^2\sqrt\rho\,\Iv(g_t).
\end{equation}
We take $\theta=\frac{1}{8}\min\{\Gamma,\Gamma^{-1}\}$ as in \eqref{eq:gammaprop}, ensuring $\theta\le\Gamma/8$ and $\theta\Gamma^2\le\Gamma/8$, so that
\begin{equation}\label{eq:constantest}
    \frac{3\Gamma}{4}-3\theta-2\theta\Gamma^2
    \ge
    \frac{3\Gamma}{4}-\frac{3\Gamma}{8}-\frac{2\Gamma}{8}
    = \frac{\Gamma}{8}
    \ge \frac{\theta}{4}.
\end{equation}
Substituting \eqref{eq:cot-absorption} into \eqref{eq:pre-absorb-controlled} and using \eqref{eq:constantest}, we arrive at
\begin{equation*}
    \frac{\ud}{\ud t}\mathcal H_\theta(g_t)
    \le
    -\frac{\theta}{2}\sqrt\rho\,\Ent_{\mux}(q_t)
    -\frac{\theta}{4}\sqrt\rho\,\Iv(g_t)
    +C_\Gamma\,\rho^{-1/2}\int g_t\abs{u_t}^2\ud\mu.
\end{equation*}
Finally, the Gaussian LSI \eqref{eq:gausslsi} gives $\Entv(g_t)\le\frac12\Iv(g_t)$, and the entropy decomposition \eqref{eq:ent-decomp} then yields
\begin{equation*}
    -\frac{\theta}{2}\sqrt\rho\,\Ent_{\mux}(q_t)-\frac{\theta}{4}\sqrt\rho\,\Iv(g_t)
    \le
    -\frac{\theta}{2}\sqrt\rho\,\bigl(\Ent_{\mux}(q_t)+\Entv(g_t)\bigr)
    =-\frac{\theta}{2}\sqrt\rho\,\Ent_\mu(g_t),
\end{equation*}
which establishes \eqref{eq:controlled-diff}.
\end{proof}

We now derive the space-time LSI for a
regular density path $g_t$ from the controlled modified-entropy estimate~\eqref{eq:controlled-diff}. In fact, the current representation \eqref{eq:negative-representation} allows $(\partial_t+\La)g_t$ to be realized, up to an arbitrarily small error in its quadratic cost, as a velocity-divergence forcing. The path $g_t$ may therefore be viewed as a controlled underdamped Langevin trajectory \eqref{eq:controlled-kolmogorov} for a suitable control $u_t$, whose cost $\int g_t\abs{u_t}^2\ud\mu$ is bounded by the negative-norm action and the velocity Fisher information of $g_t$. Applying the resulting estimate~\eqref{eq:controlled-diff} on the time window $(0,T)$ with ballistic scale $T \sim\rho^{-1/2}$ and controlling the boundary entropy then yields a coercive estimate, from where the space-time LSI follows.

\begin{theorem}[Space-time logarithmic Sobolev inequality]\label{thm:stlsi}
Suppose that \cref{gm:ass:lsi,gm:ass:standing} hold. Fix $\Gamma>0$, and set
\begin{equation} \label{eq:constantsl}
    \gamma:=\Gamma\sqrt\rho,
    \qquad
    M:=\max\{\Gamma,\Gamma^{-1}\}\ge 1.
\end{equation}
There exist constants
\begin{equation*}
    \tau_{\rm ST}(\Gamma) = 32M+4,
    \qquad
    C_{\rm ST}(\Gamma) = 80M^2+16M+2,
\end{equation*}
depending only on $\Gamma$, such that the following holds. Let $T\ge\tau_{\rm ST}\,\rho^{-1/2}$, and let $\{g_t\}_{t\in (0,T)}$ be a path of probability densities on $\RR^{2d}$ with respect to $\mu$ satisfying the regularity assumptions on $\{g_t\}$ in \cref{lem:forced-corrector}. Then the space-time logarithmic Sobolev inequality holds:
\begin{equation}\label{eq:stlsi}
    \Ent_{\muT}(g) \le C_{\rm ST} \left[I_{v,T}(g)  +\frac{1}{\rho}\frac{1}{T}\int_0^T\norm{(\partial_t+\La)g_t}_{-1,g_t}^2\ud t \right].
\end{equation}
\end{theorem}

\begin{proof}
If the right-hand side of~\eqref{eq:stlsi} is infinite, the inequality holds trivially; we therefore assume it is finite. From the constant choice \eqref{eq:constantsl}, we have $\theta=1/(8M)$ in \cref{prop:controlled}.

First, by the time-integrated version of \eqref{eq:negative-representation}, for every $\delta>0$
one can choose a jointly measurable velocity field $A$, with slices $A_t$, such that
\begin{equation}\label{eq:st-current}
    (\partial_t+\La)g_t=\nabla_v^{*}(g_tA_t)
\end{equation}
holds in $\mathcal D'((0,T)\times\RR^{2d})$, and hence in
$\mathcal D'(\RR^{2d})$ for \textup{a.e.}\ $t\in(0,T)$, with
\begin{equation}\label{eq:st-current-cost}
    \int g\abs{A}^2\ud\muT  \le  \frac{1}{T}\int_0^T\norm{(\partial_t+\La)g_t}_{-1,g_t}^2\ud t+\delta.
\end{equation}
Since $g_t$ is a probability density for every $t\in(0,T)$, we have
\begin{equation}\label{eq:space-time-entropy-average}
    \Ent_{\muT}(g)=\frac{1}{T}\int_0^T\Ent_\mu(g_t)\ud t.
\end{equation}

We next introduce $$u_t:=A_t+\gamma\nabla_v\log g_t.$$ Since $\Ls=-\nabla_v^{*}\nabla_v$, the divergence equation \eqref{eq:st-current} becomes
\begin{equation*}
    (\partial_t+\La)g_t=\gamma\Ls g_t+\nabla_v^{*}(g_tu_t),
\end{equation*}
and
\begin{equation*}
    \int g\abs{u}^2\ud\muT  \le 2\int g\abs{A}^2\ud\muT+2\gamma^2 I_{v,T}(g)<\infty.
\end{equation*}
Thus, \cref{prop:controlled} applies. Fix $\varepsilon\in(0,T/4)$ such that
both $\varepsilon$ and $T-\varepsilon$ are Lebesgue points of
$t\mapsto\mathcal H_\theta(g_t)$; this holds for \textup{a.e.}\ $\varepsilon$. Set
\begin{equation*}
    I_\varepsilon:=(\varepsilon,T-\varepsilon),
    \qquad
    L_\varepsilon:=T-2\varepsilon.
\end{equation*}
Integrating~\eqref{eq:controlled-diff} over $I_\varepsilon$ gives
\begin{equation}\label{est:modifyentropy}
    \frac{\theta}{2}\sqrt\rho\int_{I_\varepsilon}\Ent_\mu(g_t)\ud t
    \le
   \mathcal H_\theta(g_\varepsilon)-\mathcal H_\theta(g_{T - \varepsilon})
    +(\Gamma^{-1}+2\theta)\,\rho^{-1/2}\int_{I_\varepsilon}\int g_t\abs{u_t}^2\ud\mu\ud t.
\end{equation}
By the corrector bound~\eqref{est:corrector} and $0<\theta\le1/8$, we have
\begin{equation}\label{eq:regular-modified-entropy-comparison}
    0\le(1-\theta) \Ent_\mu(g_t) \le\mathcal H_\theta(g_t)  \le(1+\theta) \Ent_\mu(g_t).
\end{equation}
Consequently, \begin{equation} \label{est:modifyentropy2}
    \mathcal H_\theta(g_\varepsilon)-\mathcal H_\theta(g_{T - \varepsilon}) \le  \mathcal H_\theta(g_\varepsilon) \le (1+\theta)\Ent_\mu(g_\varepsilon).
\end{equation}
Moreover, using the elementary bound $\abs{u_t}^2\le 2\abs{A_t}^2+2\gamma^2\abs{\nabla_v\log g_t}^2$ and $\gamma^2=\Gamma^2\rho$, one can estimate the $u_t$-action cost term in \eqref{est:modifyentropy} as follows:
\begin{equation} \label{est:modifyentropy3}
    \rho^{-1/2}\int g_t\abs{u_t}^2\ud\mu  \le  2\rho^{-1/2}\int g_t\abs{A_t}^2\ud\mu +2\Gamma^2\sqrt\rho\,\Iv(g_t).
\end{equation}
Define
\begin{equation}\label{eq:K1K2-def}
    K_1 := 2(\Gamma^{-1}+2\theta)\Gamma^2 = 2\Gamma+4\theta\Gamma^2,
    \qquad
    K_2 := 2(\Gamma^{-1}+2\theta) = 2\Gamma^{-1}+4\theta.
\end{equation}
Combining \eqref{est:modifyentropy}, \eqref{est:modifyentropy2}, and
\eqref{est:modifyentropy3} yields
\begin{equation}\label{eq:regular-window-pre-average}
    \frac{\theta}{2}\sqrt\rho\int_{I_\varepsilon} \Ent_\mu(g_t)\ud t
    \le (1+\theta)\Ent_\mu(g_\varepsilon) + K_1\sqrt\rho\int_{I_\varepsilon}\Iv(g_t)\ud t + \frac{K_2}{\sqrt\rho}\int_{I_\varepsilon} \int g_t\abs{A_t}^2\ud\mu\ud t.
\end{equation}

We now estimate the initial entropy $\Ent_\mu(g_\varepsilon)$ in \eqref{eq:regular-window-pre-average}. Applying \cref{lem:entropy-chain-rule} to \eqref{eq:st-current}, with current $g_tA_t$, shows that $t\mapsto\Ent_\mu(g_t)$ is locally absolutely continuous and
\begin{equation*}
    \frac{\ud}{\ud t}\Ent_\mu(g_t)
    =\int g_tA_t\cdot\nabla_v\log g_t\ud\mu,
\end{equation*}
which, by Young's inequality with weight $\sqrt\rho$, gives
\begin{equation*}
    \Abs{\frac{\ud}{\ud t}\Ent_\mu(g_t)} \le\frac{\sqrt\rho}{2}\Iv(g_t)  +\frac{1}{2\sqrt\rho}\int g_t\abs{A_t}^2\ud\mu.
\end{equation*}
Then, we have, for every $s\in I_\varepsilon$,
\begin{equation*}
    \Ent_\mu(g_\varepsilon)
    \le\Ent_\mu(g_s)+\int_{I_\varepsilon}\Abs{\frac{\ud}{\ud\ell}\Ent_\mu(g_\ell)}\ud\ell,
\end{equation*}
and after averaging it over $s\in I_\varepsilon$,
\begin{equation}\label{eq:regular-window-left-average}
    \Ent_\mu(g_\varepsilon)
    \le
    \frac{1}{L_\varepsilon}\int_{I_\varepsilon} \Ent_\mu(g_t)\ud t
    +\frac{\sqrt\rho}{2}\int_{I_\varepsilon}\Iv(g_t)\ud t
    +\frac{1}{2\sqrt\rho}\int_{I_\varepsilon} \int
        g_t\abs{A_t}^2\ud\mu\ud t.
\end{equation}
Combining \eqref{eq:regular-window-pre-average} and
\eqref{eq:regular-window-left-average} gives
\begin{multline}\label{eq:estintentropy}
    \left(\frac{\theta}{2}\sqrt\rho-\frac{1+\theta}{L_\varepsilon}\right)
    \int_{I_\varepsilon} \Ent_\mu(g_t)\ud t \\
    \le
    \left(K_1+\frac{1+\theta}{2}\right)\sqrt\rho
        \int_{I_\varepsilon}\Iv(g_t)\ud t
    +\frac{K_2+(1+\theta)/2}{\sqrt\rho}
        \int_{I_\varepsilon} \int g_t\abs{A_t}^2\ud\mu\ud t.
\end{multline}

Define
\begin{equation}\label{eq:tau-ST-def}
    \tau_{\rm ST}:=\frac{4(1+\theta)}{\theta}=32M+4.
\end{equation}
The choice $T\ge\tau_{\rm ST}\rho^{-1/2}$ gives
\begin{equation*}
        \frac{\theta}{2}\sqrt\rho-\frac{1+\theta}{T} \ge\frac{\theta}{4}\sqrt\rho.
\end{equation*}
It also ensures that the coefficient on the left of
\eqref{eq:estintentropy} is positive for every
$\varepsilon\in(0,T/4)$, because
\begin{equation*}
    L_\varepsilon>\frac{T}{2}
    \ge\frac{2(1+\theta)}{\theta\sqrt\rho}.
\end{equation*}
Choose a decreasing sequence $\varepsilon_n\downarrow0$ such that both $\varepsilon_n$ and $T-\varepsilon_n$ are Lebesgue points of $t\mapsto\mathcal H_\theta(g_t)$. Since all the integrands in \eqref{eq:estintentropy} are nonnegative, monotone convergence gives
\begin{equation} \label{eq:entropyconverg}
    \int_{I_{\varepsilon_n}}\Ent_\mu(g_t)\ud t \to T\Ent_{\muT}(g)
\end{equation}
and the two integrals on the right-hand side converge to their counterparts over $(0,T)$. Letting $\varepsilon_n\downarrow0$ in \eqref{eq:estintentropy} therefore yields
\begin{equation}\label{eq:estintentropy2}
    \Ent_{\muT}(g)
    \le \frac{4}{\theta}\!\left(K_1+\frac{1+\theta}{2}\right)I_{v,T}(g)
    +\frac{4}{\theta}\!\left(K_2+\frac{1+\theta}{2}\right)\frac{1}{\rho}\int g\abs{A}^2\ud\muT.
\end{equation}
A direct computation from~\eqref{eq:K1K2-def} with $\theta=1/(8M)$ gives $K_1,K_2\le 5M/2$, and then the two coefficients on the right-hand side of \eqref{eq:estintentropy2} are both bounded by
\begin{equation}\label{eq:CST-explicit}
    C_{\rm ST} :=\frac{4}{\theta}\!\left(\frac{5M}{2}+\frac{1+\theta}{2}\right) =\frac{10M}{\theta}+\frac{2(1+\theta)}{\theta}  =80M^2+16M+2.
\end{equation}
Therefore
\begin{equation*}
    \Ent_{\muT}(g) \le C_{\rm ST} \left[I_{v,T}(g) + \frac{1}{\rho}\int g\abs{A}^2\ud\muT  \right].
\end{equation*}
Recalling \eqref{eq:st-current-cost} and letting $\delta\downarrow0$ yields \eqref{eq:stlsi}. The proof is complete.
\end{proof}

\begin{remark}[Explicit constants and friction scaling]\label{rem:stlsi-constants}
With $M:=\max\{\Gamma,\Gamma^{-1}\}\ge 1$, the constants in \cref{thm:stlsi} read
\begin{equation*}
    \tau_{\rm ST}(\Gamma)=32M+4, \qquad  C_{\rm ST}(\Gamma)=80M^2+16M+2\le 98M^2,
\end{equation*}
and both attain their minimum at the friction $\Gamma=1$, where $\tau_{\rm ST}(1)=36$ and $C_{\rm ST}(1)=98$. Moreover, as $\Gamma\to 0$ or $\Gamma\to\infty$, we have
\begin{equation*}
    \tau_{\rm ST}(\Gamma)\asymp\max\{\Gamma,\Gamma^{-1}\},  \qquad C_{\rm ST}(\Gamma)\asymp\max\{\Gamma^2,\Gamma^{-2}\}.
\end{equation*}
Hence, the space-time LSI constant degrades only polynomially in $\max\{\Gamma,\Gamma^{-1}\}$, identifying $\Gamma\sim 1$ (equivalently, $\gamma\sim\sqrt\rho$) as the optimal friction-scaling regime.
\end{remark}

\subsection{Extension to finite-action paths} \label{sec:finite-action-extension}
We next extend the space-time LSI to finite-action paths, removing the regularity assumption in \cref{thm:stlsi}.
We will first record a windowed regular-path estimate in a formulation suited to approximation (\cref{gm:lem:window}).  The main technical result is the recovery theorem (\cref{gm:thm:recovery}), which approximates an admissible finite-action pair on interior time windows by regular density-current pairs with the required action and Fisher-information limsup bounds. Applying the windowed estimate to the recovery sequence and passing to the limit by lower semicontinuity of the entropy yields the finite-action space-time LSI (\cref{gm:cor:stlsi,gm:cor:stlsi-action}).

For the finite-action extension, it is useful to rewrite the action term $
    \frac{1}{T}\int_0^T \norm{(\partial_t+\La)g_t}_{-1,g_t}^2\ud t$
in \eqref{eq:stlsi} as a convex least-action problem. This amounts to passing from the velocity field $A$ in the representation \eqref{eq:negative-representation} to the current $W=gA$. Thus, for a path
$\{g_t\}_{t\in[0,T]}$ of probability densities with respect to $\mu$ and a
measurable current $W:(0,T)\times\RR^{2d}\to\RR^d$, we introduce the cost
\begin{equation}\label{gm:eq:functionals}
    J_T(g,W) = \int\frac{\abs{W}^2}{g}\ud\muT,
\end{equation}
with the convention that $\abs W^2/g=0$ on $\{g=0,W=0\}$ and $\abs W^2/g=+\infty$ on $\{g=0,W\neq0\}$. The admissible currents are those satisfying
\begin{equation}\label{gm:eq:graph}
        (\partial_t+\La)g_t=\nabla_v^{*}W_t \qquad\text{in } \ \mathcal D'((0,T)\times\RR^{2d}).
\end{equation}
The finite-action cost of $g$ is then defined by the least-action problem
\begin{equation}\label{eq:finite-action-cost}
\begin{aligned}
    \mathcal A_T(g)
    &:= \inf_W
    \left\{   J_T(g,W):\  (\partial_t+\La)g_t=\nabla_v^{*}W_t\ \text{ in }\ \mathcal D'((0,T)\times\RR^{2d})
    \right\}.
    \end{aligned}
\end{equation}
Equivalently, by the current representation $W=gA$ in
\eqref{eq:negative-representation},
\begin{equation*}
    \mathcal A_T(g)
    =\frac{1}{T}\int_0^T
    \norm{(\partial_t+\La)g_t}_{-1,g_t}^2\ud t.
\end{equation*}

We next specify the class of admissible probability paths for which the finite-action space-time LSI will be proved.

\begin{definition}[Admissible pair]\label{gm:ass:path}
Let $\{g_t\}_{t\in(0,T)}$ be a path of probability densities on $\RR^{2d}$ with respect to $\mu$, where $\int g_t\ud\mu=1$ for \textup{a.e.}\ $t\in(0,T)$, and let $W:(0,T)\times\mathbb R^{2d}\to\mathbb R^d$ be a measurable current. We say that the pair $(g,W)$ is admissible if
\begin{enumerate}[label=(P\arabic*)]
\item\label{gm:P1} $\Ent_{\muT}(g)=\frac{1}{T}\int_0^T\Ent_\mu(g_t)\ud t<\infty$;
\item\label{gm:P2} $I_{v,T}(g)<\infty$;
\item\label{gm:P3} the equation \eqref{gm:eq:graph} holds and $J_T(g,W)<\infty$.
\end{enumerate}
In addition, we impose the following conditions:
\begin{enumerate}[label=(Q\arabic*)]
\item\label{gm:Q1} writing $q_t=\Piv g_t$ and $j_t=\Piv(vg_t)$, the conditional mean velocity $m_t=j_t/q_t$, defined to be zero on $\{q_t=0\}$, is locally bounded: for every compact interval $I\subset(0,T)$ there exists $M_I<\infty$ such that, for \textup{a.e.}\ $(t,x)\in I\times\RR^d$,
\begin{equation}\label{gm:eq:Q1}
    \abs{j_t(x)}\le M_I\,q_t(x);
\end{equation}
\item\label{gm:Q2} the integrated fourth velocity moment is finite:
\begin{equation}\label{gm:eq:Q2}
    \int_0^T \int(1+\abs{v}^4)\,g_t\ud\mu\ud t<\infty;
\end{equation}
\item\label{gm:Q3} the integrated gradient moment is finite:
\begin{equation}\label{gm:eq:Q3}
    \int_0^T \int \bigl(1+\abs{\nabla U(x)}\bigr)^2\,g_t\ud\mu\ud t<\infty\,.
\end{equation}
\end{enumerate}
\end{definition}

\begin{remark}\label{gm:rem:QQ}
The conditions \ref{gm:P1}--\ref{gm:P3} are the intrinsic assumptions for the well-posedness of the space-time LSI. The additional conditions \ref{gm:Q1}--\ref{gm:Q3} are technical assumptions used to construct the recovery sequence in \cref{gm:thm:recovery}. More precisely, conditions \ref{gm:Q1} and \ref{gm:Q2} are used only to verify the controlled corrector inequality \ref{gm:R3}.  Condition \ref{gm:Q3} is used to establish the current regularity in \ref{gm:R1}, the slice bounds entering \ref{gm:R2}, the acceleration estimate for \ref{gm:R3}, and the force-commutator estimate for the action limsup bound in \ref{gm:R5}. The $L^1$ convergence in \ref{gm:R4} does not require any of \ref{gm:Q1}--\ref{gm:Q3}, and the Fisher-information part of \ref{gm:R5} uses only \ref{gm:P2}. If the potential satisfies the gradient bound \eqref{eq:gradidentpotential}: $\abs{\nabla U}^2\le C(1+U)$, then \ref{gm:P1} implies \ref{gm:Q3}. We refer to \cref{gm:rem:hypotheses} for more detailed discussions.
\end{remark}

We first isolate from the proof of \cref{thm:stlsi} the regular-path estimate that will be transferred to finite-action paths. It is formulated on an arbitrary interior time window and uses only the divergence equation, the entropy dissipation identity, and the controlled corrector inequality. Its proof is precisely the derivation of \eqref{eq:estintentropy}, with $g_tA_t$ replaced by $W_t$, and is omitted.

\begin{lemma}[Windowed regular-path estimate]\label{gm:lem:window}
Suppose that \cref{gm:ass:lsi,gm:ass:standing} hold. Fix $\Gamma>0$, set
$\gamma:=\Gamma\sqrt\rho$, and write
\begin{equation} \label{eq:parameter}
     \theta:=\frac18\min\{\Gamma,\Gamma^{-1}\},\qquad   K_1:=2\Gamma+4\theta\Gamma^2,\qquad   K_2:=2\Gamma^{-1}+4\theta.
\end{equation}
Consider the interval $I=(t_0,t_1)\subset(0,T)$ with length $L:=t_1-t_0$ such that
\begin{equation*}
    c(L):=\frac{\theta}{2}\sqrt\rho-\frac{1+\theta}{L}>0.
\end{equation*}
Let the density path $h\in C^1\bigl(I;C^\infty(\RR^{2d})\bigr)$ be strictly positive, with
$\int h_t\ud\mu=1$ for every $t\in I$, and let $W:I\times\RR^{2d}\to\RR^d$
be measurable. Assume that
\begin{enumerate}[label=\textup{(W\arabic*)}]
\item\label{gm:W1}
the equation
\begin{equation*}
        (\partial_t+\La)h_t=\nabla_v^{*}W_t
\end{equation*}
holds in $\mathcal D'(I\times\RR^{2d})$, and
\begin{equation*}
        \int_I \int\frac{\abs{W_t}^2}{h_t}\ud\mu\ud t<\infty, \qquad \int_I\Iv(h_t)\ud t<\infty;
\end{equation*}
\item\label{gm:W2}
the function $t\mapsto\Ent_\mu(h_t)$ belongs to $L^1_{\rm loc}(I)$;
\item\label{gm:W3}
writing $q_t:=\Piv h_t$, $j_t:=\Piv(vh_t)$, and
\begin{equation*}
        B_t:=\Piv W_t+\gamma j_t,
\end{equation*}
all terms in the following inequality are locally integrable on $I$, and
\begin{equation*}
    \frac{\ud}{\ud t}\mathcal C_{\rm OT}(h_t) \le - \Ent_{\mux}(q_t)-\gamma\mathcal C_{\rm OT}(h_t)+3\Iv(h_t)
    +\int B_t\cdot\xi_{q_t}\ud\mux
\end{equation*}
holds in $\mathcal D'(I)$.
\end{enumerate}
Then, \cref{lem:entropy-chain-rule} gives
\begin{equation*}
    \frac{\ud}{\ud t}\Ent_\mu(h_t)
    =\int W_t\cdot\nabla_v\log h_t\ud\mu
    \qquad\text{for \textup{a.e.}\ }t\in I.
\end{equation*}
Moreover, $t\mapsto\Ent_\mu(h_t)$ has a unique absolutely continuous
extension to $[t_0,t_1]$, because the two integrability conditions in
\ref{gm:W1} make its derivative integrable up to the endpoints, and
\begin{equation}\label{gm:eq:window-inequality}
    c(L)\int_{t_0}^{t_1}\Ent_\mu(h_t)\ud t \le  \Bigl(K_1+\frac{1+\theta}{2}\Bigr)\sqrt\rho\int_{t_0}^{t_1}\Iv(h_t)\ud t  +\frac{K_2+\frac{1+\theta}{2}}{\sqrt\rho}  \int_{t_0}^{t_1} \int\frac{\abs{W_t}^2}{h_t}\ud\mu\ud t.
\end{equation}
\end{lemma}

We now construct regular approximants of an admissible pair $(g,W)$ in the sense of \cref{gm:ass:path}. The following recovery theorem yields, on every interior time window, regular pairs satisfying \ref{gm:W1}--\ref{gm:W3}, so that \cref{gm:lem:window} applies, together with the $L^1$ convergence and the action and Fisher-information limsup bounds required for the limiting argument. Note that the subsequent argument uses \cref{lem:forced-corrector} only through its conclusion, namely the controlled corrector inequality. Therefore, rather than checking all the hypotheses of that lemma for the approximants, in particular the
cutoff-limit conditions, we will verify the inequality (\textup{i.e.},  \ref{gm:R3}
) directly.
For ease of exposition, the detailed construction (the proof of \cref{gm:thm:recovery}) is deferred to \cref{app:recovery}.

\begin{theorem}[Recovery on interior windows]\label{gm:thm:recovery}
Let \cref{gm:ass:lsi,gm:ass:standing} hold, and let $(g,W)$ be admissible in the sense of \cref{gm:ass:path}. Fix $\Gamma>0$, set $\gamma:=\Gamma\sqrt\rho$, and let $0<\varepsilon<T/4$. Write
$$
I_\varepsilon := (\varepsilon,T-\varepsilon),
\qquad I_{\varepsilon/2} := \left(\frac{\varepsilon}{2}, T-\frac{\varepsilon}{2}\right).
$$
Then there is a sequence of density-current pairs $(g^n,W^n)$ on $I_{\varepsilon/2}\times\RR^{2d}$ such that:
\begin{enumerate}[label=(R\arabic*)]
\item\label{gm:R1}
$g^n\in C^1\bigl(I_{\varepsilon/2};C^\infty(\RR^{2d})\bigr)$ and
$W^n\in C\bigl(I_{\varepsilon/2};C^\infty(\RR^{2d};\RR^d)\bigr)$. Moreover,
\begin{equation*}
    \int g^n_t\ud\mu=1\quad\text{for every }t\in I_{\varepsilon/2}, \qquad g^n_t\ge\delta_n>0\quad\text{on }I_{\varepsilon/2}\times\RR^{2d};
\end{equation*}
\item\label{gm:R2} on the smaller window $I_\varepsilon$, the equation
\begin{equation*}
    (\partial_t+\La)g^n_t=\nabla_v^{*}W^n_t
\end{equation*}
holds pointwise. On the larger window $I_{\varepsilon/2}$, the map $t\mapsto\Ent_\mu(g^n_t)$ is locally absolutely continuous and
\begin{equation*}
    \frac{\ud}{\ud t}\Ent_\mu(g^n_t)
    =\int W^n_t\cdot\nabla_v\log g^n_t\ud\mu
\end{equation*}
for \textup{a.e.}\ $t\in I_{\varepsilon/2}$. The quantities $\Ent_\mu(g^n_t)$ and $\Iv(g^n_t)$ are locally bounded on $I_{\varepsilon/2}$;
\item\label{gm:R3}
writing $q^n_t:=\Piv g^n_t$ and $j^n_t:=\Piv(vg^n_t)$, and setting
\begin{equation*}
    B^n_t:=\Piv(W^n_t)+\gamma j^n_t,
\end{equation*}
the controlled corrector differential inequality holds in $\mathcal D'(I_\varepsilon)$
\begin{equation}\label{gm:eq:weak-corrector}
    \frac{\ud}{\ud t}\,\mathcal C_{\rm OT}(g^n_t) \le -\Ent_{\mux}(q^n_t)-\gamma\,\mathcal C_{\rm OT}(g^n_t)+3\Iv(g^n_t)
    +\int B^n_t\cdot\xi_{q^n_t}\ud\mux\,,
\end{equation}
where $\mathcal C_{\rm OT}(g^n_t)=\int j^n_t\cdot\xi_{q^n_t}\ud\mux$ and $\xi_{q^n_t}$ is the Brenier displacement from $q^n_t\mux$ to $\mux$;
\item\label{gm:R4}
the convergence holds
\begin{equation*}
    g^n\to g \quad \text{in $\ L^1\bigl(I_\varepsilon\times\RR^{2d};\ud t\ud\mu\bigr)$;}
\end{equation*}
\item\label{gm:R5}
the action and velocity Fisher information, given by,
\begin{equation*}
J_\varepsilon(h,V):=\int_{I_\varepsilon} \int\frac{\abs{V_t}^2}{h_t}\ud\mu\ud t, \qquad
I_{v,\varepsilon}(h):=\int_{I_\varepsilon} \int\frac{\abs{\nabla_vh_t}^2}{h_t}\ud\mu\ud t,
\end{equation*}
satisfy
\begin{equation*}
    \limsup_{n\to\infty}J_\varepsilon(g^n,W^n)\le T\,J_T(g,W),
    \qquad
    \limsup_{n\to\infty}I_{v,\varepsilon}(g^n)\le T\,I_{v,T}(g).
\end{equation*}
\end{enumerate}
\end{theorem}

We next prove the space-time LSI for every fixed admissible pair $(g,W)$. The intrinsic formulation in terms of $\mathcal A_T(g)$ then follows by minimizing over the current $W$.

\begin{theorem}[Finite-action space-time LSI for admissible pairs]\label{gm:cor:stlsi}
Let \cref{gm:ass:lsi,gm:ass:standing} hold. Fix $\Gamma>0$, set $\gamma:=\Gamma\sqrt\rho$, and let $T\ge\tau_{\rm ST}(\Gamma)\rho^{-1/2}$. Let $(g,W)$ be an admissible pair in the sense of \cref{gm:ass:path}. Then
\begin{equation} \label{eq:stlsi_pair}
    \Ent_{\muT}(g) \le C_{\rm ST}(\Gamma) \left[ I_{v,T}(g) + \frac1\rho J_T(g,W) \right],
\end{equation}
where $\tau_{\mathrm{ST}}$ and $C_{\mathrm{ST}}$ are the constants from \cref{thm:stlsi}.
\end{theorem}

\begin{proof}
Fix $\varepsilon\in(0,T/4)$, set $ I_\varepsilon:=(\varepsilon,T-\varepsilon)$ with $L_\varepsilon:=T-2\varepsilon$,  and let $(g^n,W^n)$ be the recovery sequence from \cref{gm:thm:recovery}. Let $\theta$, $K_1$ and $K_2$ be as in \eqref{eq:parameter}. As in the proof of \cref{thm:stlsi}, we have
\begin{equation*}
     c(L_\varepsilon)=\frac{\theta}{2}\sqrt\rho -\frac{1+\theta}{L_\varepsilon}>0\,, \qquad \lim_{\varepsilon \to 0} c(L_\varepsilon) \ge\frac{\theta}{4}\sqrt\rho.
\end{equation*}
From \ref{gm:R5}, without loss of generality, one can assume $J_\varepsilon(g^n,W^n)+I_{v,\varepsilon}(g^n)<\infty$ for all $n$, after discarding finitely many terms. Then, \cref{gm:lem:window} applies to $(h,W)=(g^n,W^n)$ by \ref{gm:R1}--\ref{gm:R3}, and yields
\begin{equation*}
   c(L_\varepsilon) \int_{I_\varepsilon}\Ent_\mu(g^n_t)\ud t \le
    \Bigl(K_1+\frac{1+\theta}2\Bigr)\sqrt\rho\;I_{v,\varepsilon}(g^n) +\frac{K_2+\tfrac{1+\theta}2}{\sqrt\rho}\,J_\varepsilon(g^n,W^n).
\end{equation*}
By \ref{gm:R4}, the convergence $g^n\to g$ also holds in $L^1$ with respect to the probability measure $L_\varepsilon^{-1}\ud t\otimes\mu$ on
$I_\varepsilon\times\RR^{2d}$, and both $g^n$ and $g$ are probability densities with respect to this measure. Their relative entropies with respect to this measure are, respectively,
$L_\varepsilon^{-1}\int_{I_\varepsilon}\Ent_\mu(g^n_t)\ud t$ and
$L_\varepsilon^{-1}\int_{I_\varepsilon}\Ent_\mu(g_t)\ud t$. Lower semicontinuity of relative entropy, together with the two limsup bounds in \ref{gm:R5}, therefore gives
\begin{equation*}
    c(L_\varepsilon) \int_{I_\varepsilon}\Ent_\mu(g_t)\ud t   \le \Bigl(K_1+\frac{1+\theta}2\Bigr)\sqrt\rho\;T\,I_{v,T}(g) +\frac{K_2+\tfrac{1+\theta}2}{\sqrt\rho}\,T J_T(g,W)\,.
\end{equation*}
By the convergence \eqref{eq:entropyconverg}, letting $\varepsilon\downarrow0$, and dividing by $\frac{\theta}{4}\sqrt\rho\,T$, we obtain \eqref{eq:stlsi_pair}.
\end{proof}

\begin{corollary}[Finite-action space-time LSI]
\label{gm:cor:stlsi-action}
Let \cref{gm:ass:lsi,gm:ass:standing} hold. Fix $\Gamma>0$, set $\gamma:=\Gamma\sqrt\rho$, and let $T\ge\tau_{\rm ST}(\Gamma)\rho^{-1/2}$. Let $g=\{g_t\}_{t\in(0,T)}$ be a measurable path of probability densities with respect to $\mu$. Assume that
\begin{equation*}
    \Ent_{\muT}(g)<\infty, \qquad
I_{v,T}(g)<\infty, \qquad \mathcal A_T(g)<\infty,
\end{equation*}
and that $g$ satisfies \ref{gm:Q1}--\ref{gm:Q3} of \cref{gm:ass:path}. Then
\begin{equation*}
     \Ent_{\muT}(g) \le C_{\rm ST}(\Gamma) \left[I_{v,T}(g)  +\frac{1}{\rho}\mathcal A_T(g) \right].
\end{equation*}
\end{corollary}

\begin{proof}
Fix $\delta>0$. By the definition of the finite-action cost \eqref{eq:finite-action-cost}, there exists a measurable current $W$ such that
the equation \eqref{gm:eq:graph} holds and
\begin{equation*}
    J_T(g,W)\le\mathcal A_T(g)+\delta .
\end{equation*}
Together with the assumptions of the corollary, this shows that $(g,W)$ is admissible in the sense of \cref{gm:ass:path}. Applying Theorem~\ref{gm:cor:stlsi} and letting $\delta\downarrow0$ proves the result.
\end{proof}

\section{The time-interpolation estimate}\label{sec:interpolation}

Throughout this section and the next, we assume that
\cref{gm:ass:lsi,gm:ass:standing} hold. We fix a $\Gamma>0$, set the friction parameter $\gamma:=\Gamma\sqrt\rho$, and choose a time period $T\ge\tau_{\rm ST}(\Gamma)\,\rho^{-1/2}$, so that the finite-action space-time LSI \cref{gm:cor:stlsi} is available with constant $C_{\rm ST}=C_{\rm ST}(\Gamma)$ for any probability path satisfying \cref{gm:ass:path}. For ease of exposition, we write
\begin{equation*}
    \Pt_t:=\e^{t(-\La+\gamma\Ls)},   \qquad  \Pt_t^\ast:=\e^{t(\La+\gamma\Ls)}
\end{equation*}
for the density semigroup of the underdamped Langevin dynamics~\eqref{eq:underdamped-Langevin} and its $L^2(\mu)$-adjoint, respectively. Both are positivity-preserving Markov contractions on $L^p(\mu)$ for any $1\le p\le\infty$~\cite{BakryGentilLedoux2014}.

The goal of this section is to control the bilinear pairing $\int\Pt_T\varphi\cdot\psi\ud\mu$. Taking the supremum over $\psi\in L^{q'}(\mu)$ with $\norm{\psi}_{L^{q'}(\mu)}=1$ then recovers the $L^q(\mu)$-norm of $\Pt_T\varphi$. It is the main step for proving hypocoercive hypercontractivity in \cref{sec:hypercontractivity}. Here and in what follows, $q'$ denotes the conjugate exponent of $q$, $1/q+1/q'=1$. We will establish the estimate for smooth test functions $\varphi,\psi$ bounded above and below by positive constants. The approximation argument in the proof of \cref{prop:central} then extends the result to general data.

To control such pairing, we construct a forward/backward time interpolation between $\varphi$ and $\psi$, following the
semigroup-interpolation method of Bakry--\'Emery~\cite{BakryEmery1985} and Neveu~\cite{Neveu1976}, with one essential adaptation: the underdamped generator $-\La+\gamma\Ls$ has a skew-adjoint part $-\La$, so the density semigroup $\Pt_t$ and its $L^2(\mu)$-adjoint $\Pt^{*}_t$ do not coincide. We therefore need to propagate $\varphi$ and $\psi$ under $\Pt_t$ and $\Pt_t^{*}$ simultaneously.

Specifically, for smooth functions $\varphi,\psi$ bounded above and below by positive constants, define the \emph{forward/backward interpolation density}
$G_s$ by
\begin{equation}\label{def:interpolationdensity}
G_s=\frac{\varphi_s\psi_s}{Z}, \qquad  Z=\int\varphi_s\psi_s\ud\mu,
      \qquad \text{with} \qquad  \varphi_s=\Pt_s\varphi,\qquad  \psi_s=\Pt^\ast_{T-s}\psi.
\end{equation}
By the skew-adjointness $\La^{*}=-\La$ and the symmetry $\Ls^{*}=\Ls$ in $L^2(\mu)$,
\begin{equation*}
        \frac{\ud}{\ud s}\int\varphi_s\psi_s\ud\mu
        =\int(-\La+\gamma\Ls)\varphi_s\cdot\psi_s\ud\mu
        -\int\varphi_s\cdot(\La+\gamma\Ls)\psi_s\ud\mu=0,
\end{equation*}
so the normalization $Z$ is independent of $s$, and each $G_s$ is indeed a probability density with respect to $\mu$. In particular, at $s=T$, we have $\varphi_T=\Pt_T\varphi$ and $\psi_T=\psi$, so that $Z$ recovers the pairing of interest:
\begin{equation*}
        Z=\int\Pt_T\varphi\cdot\psi\ud\mu.
\end{equation*}
To bound it from above, we first apply the finite-action space-time LSI to the path $\{G_s\}_{s\in[0,T]}$ and propagate the resulting bound on $\Ent_{\muT}(G)$ to the endpoint entropies $\Ent_\mu(G_0)$ and $\Ent_\mu(G_T)$; see \cref{prop:interpolation-endpoint}. We then translate these endpoint estimates into an upper bound on $Z$ in \cref{prop:closure}.

We start by identifying the divergence equation satisfied by $G_s$ and bound its velocity Fisher information and the current cost of $(G,W)$ in the following lemma.

\begin{lemma}\label{lem:interpolation-current}
Let $\varphi,\psi$ be smooth functions bounded above and below by positive constants, and let $G_s$ be the forward/backward interpolation density defined in \eqref{def:interpolationdensity}. Define
\begin{equation} \label{eq:currentgsw}
     A_s:=-\gamma\bigl(\nabla_v\log\varphi_s-\nabla_v\log\psi_s\bigr)\,, \quad  W_s = G_s A_s\,.
\end{equation}
Then $G_s$ satisfies the controlled equation:
\begin{equation}\label{eq:interpolation-current}
   (\partial_s+\mathcal L_a)G_s = \nabla_v^*W_s = \nabla_v^*(G_s A_s).
\end{equation}
Moreover, the velocity Fisher information and the finite-action cost of $G$ are
both bounded by the \emph{two-sided Fisher information}
\begin{equation}\label{eq:interpolation-energy}
        \mathcal D_T(\varphi,\psi)
        :=\frac{1}{T}\int_0^T\int \left(
\abs{\nabla_v\log\varphi_s}^2+\abs{\nabla_v\log\psi_s}^2  \right) G_s \ud\mu\ud s,
\end{equation}
namely
\begin{equation}\label{eq:interpolation-norms}
I_{v,T}(G)\le 2\,\mathcal D_T(\varphi,\psi), \qquad J_T(G,W)  = \frac{1}{T}\int_0^T\int G_s|A_s|^2 \ud \mu \ud s  \le 2\gamma^2\,\mathcal D_T(\varphi,\psi).
\end{equation}
In particular, $\mathcal{A}_T(G) \le J_T(G,W) \le 2 \gamma^2 \mathcal{D}_T(\varphi, \psi)$.
\end{lemma}

\begin{proof}
Since $\mathcal L_a$ is a first-order differential operator, we have
\begin{equation*}
    \mathcal L_a(\varphi_s\psi_s) = (\mathcal L_a\varphi_s)\psi_s + \varphi_s\mathcal L_a\psi_s.
\end{equation*}
Combining this with the forward and backward equations:
\begin{equation*}
    \partial_s\varphi_s = (-\mathcal L_a+\gamma\mathcal L_s)\varphi_s, \qquad \partial_s\psi_s = -(\mathcal L_a+\gamma\mathcal L_s)\psi_s,
\end{equation*}
we obtain
\begin{equation}\label{eq:partial-La-product}
    (\partial_s+\La)(\varphi_s\psi_s)
    =\gamma\bigl(\psi_s\Ls\varphi_s-\varphi_s\Ls\psi_s\bigr).
\end{equation}
Using $\mathcal L_s=-\nabla_v^*\nabla_v$ and
\begin{equation*}
    \nabla_v^*(fF) = f\nabla_v^*F-\nabla_vf\cdot F,
\end{equation*}
the mixed terms cancel and
\begin{equation*}
    \psi_s\mathcal L_s\varphi_s - \varphi_s\mathcal L_s\psi_s =
-\nabla_v^* \left[ \varphi_s\psi_s \bigl(\nabla_v\log\varphi_s-\nabla_v\log\psi_s\bigr)\right].
\end{equation*}
Substituting into~\eqref{eq:partial-La-product} and dividing by the time-independent normalization $Z$ yields~\eqref{eq:interpolation-current}.

We next bound the velocity Fisher information $\Iv(G_s)$. Recalling that $G_s = \varphi_s \psi_s / Z$ and applying the product rule
\begin{equation*}
    \nabla_v \log G_s = \nabla_v \log \varphi_s + \nabla_v \log \psi_s,
\end{equation*}
the elementary inequality $\abs{a + b}^2 \le 2\abs{a}^2 + 2\abs{b}^2$ gives
\begin{equation*}
    \Iv(G_s) = \int G_s \abs{\nabla_v \log G_s}^2\ud\mu
    \le 2\int G_s\bigl(\abs{\nabla_v \log \varphi_s}^2 + \abs{\nabla_v \log \psi_s}^2\bigr)\ud\mu.
\end{equation*}
Time-averaging over $[0, T]$ yields the first bound in~\eqref{eq:interpolation-norms}.

For the second bound in~\eqref{eq:interpolation-norms}, we apply the variational characterization~\eqref{eq:negative-representation} with the explicit current $A_s$ given in~\eqref{eq:interpolation-current}, which yields
\begin{align*}
    \norm{(\partial_s + \La)G_s}_{-1, G_s}^2
    \le \int G_s \abs{A_s}^2\ud\mu
    &= \gamma^2 \int G_s \abs{\nabla_v \log \varphi_s - \nabla_v \log \psi_s}^2\ud\mu \\
    &\le 2\gamma^2 \int G_s \bigl(\abs{\nabla_v \log \varphi_s}^2 + \abs{\nabla_v \log \psi_s}^2\bigr)\ud\mu.
\end{align*}
Again, time-averaging over $[0, T]$ completes the proof.
\end{proof}

We next show that the interpolation path $G$ and the associated current $W$ indeed satisfy \cref{gm:ass:path} so that the space-time LSI in \cref{gm:cor:stlsi} can apply.

\begin{lemma}[Admissibility of the interpolation pair]
\label{lem:interpolation-admissible}
Let $\varphi,\psi$ be smooth and satisfy
\begin{equation*}
    0<m_0\le \varphi, \psi\le M_0.
\end{equation*}
Let $G_s$ be defined by \eqref{def:interpolationdensity}, and let $W_s = G_s A_s$ be the current \eqref{eq:currentgsw} from \cref{lem:interpolation-current}. Then $G$ satisfies \ref{gm:P1} and \ref{gm:Q1}--\ref{gm:Q3} of \cref{gm:ass:path}. Moreover, $(G,W)$ is an admissible pair of the space-time LSI \eqref{eq:stlsi_pair}.
\end{lemma}

\begin{proof}
Since $\Pt_s$ and $\Pt^*_{T-s}$ are positivity-preserving Markov operators, we have, for all $s \in [0,T]$,
\begin{equation*}
    m_0 \le \varphi_s, \psi_s \le M_0.
\end{equation*}
Moreover, the normalization $Z=\int\varphi_s\psi_s\ud\mu$ is independent of $s$ and satisfies $m_0^2\le Z\le M_0^2$, implying
\begin{equation}\label{eq:G-uniform-bounds}
  \frac{m_0^2}{M_0^2} \le \frac{m_0^2}{Z}\le G_s\le \frac{M_0^2}{Z} \le \frac{M_0^2}{m_0^2}.
\end{equation}
It follows that \ref{gm:P1} of \cref{gm:ass:path} holds:
\begin{equation*}
    \Ent_\mu(G_s)\le \log\left(M_0^2\right) - \log\left(m_0^2\right).
\end{equation*}

We next verify \ref{gm:Q1}. Let $q_s=\Piv G_s$ and $j_s=\Piv(vG_s)$. For each $x$, set the conditional density $h_{s,x}(v)=G_s(x,v)/q_s(x)$  with respect to $\kv$. Since $q_s(x)\ge m_0^2/Z$, \eqref{eq:G-uniform-bounds} implies
\begin{equation*}
    h_{s,x}(v)\le \left(\frac{M_0}{m_0}\right)^2.
\end{equation*}
Hence, we have the property \ref{gm:Q1}:
\begin{equation*}
    \abs{j_s(x)} \le q_s(x)\int \abs v\,h_{s,x}(v)\kv(\ud v) \le q_s(x)\left(\frac{M_0}{m_0}\right)^2  \int \abs v\,\kv(\ud v).
\end{equation*}
The same uniform upper bound gives \ref{gm:Q2}:
\begin{equation*}
    \int_0^T \int(1+\abs v^4)G_s\ud\mu \ud s  \le  T\,\frac{M_0^2}{Z}\int(1+\abs v^4) \kv(\ud v)<\infty.
\end{equation*}
Moreover, since $\int \abs{\nabla U(x)}^2\ud\mux < \infty$ (see  \eqref{gm:eq:static-gradient}), \ref{gm:Q3} is verified:
\begin{equation*}
    \int_0^T \int(1+\abs{\nabla U(x)})^2G_s \ud \mu \ud s  \le  T\,\frac{M_0^2}{Z} \int(1+\abs{\nabla U(x)})^2\ud\mux < \infty.
\end{equation*}

It remains to verify that $\mathcal D_T(\varphi,\psi)<\infty$. By the uniform bounds on $\varphi_s$ and $\psi_s$,
\begin{equation*}
    T\mathcal D_T(\varphi,\psi)
    \le
    \frac{M_0}{m_0Z}
    \int_0^T
    \left(
        \norm{\nabla_v\varphi_s}_{L^2(\mu)}^2
        +
        \norm{\nabla_v\psi_s}_{L^2(\mu)}^2
    \right)\ud s.
\end{equation*}
The $L^2(\mu)$ energy identities for $\Pt_s$ and $\Pt_s^*$ give
\begin{equation*}
    2\gamma\int_0^T\norm{\nabla_v\varphi_s}_{L^2(\mu)}^2\ud s
    =
    \norm{\varphi}_{L^2(\mu)}^2-\norm{\varphi_T}_{L^2(\mu)}^2\,, \quad   2\gamma\int_0^T\norm{\nabla_v\psi_s}_{L^2(\mu)}^2\ud s
    =
    \norm{\psi}_{L^2(\mu)}^2-\norm{\psi_0}_{L^2(\mu)}^2.
\end{equation*}
Hence $\mathcal D_T(\varphi,\psi)<\infty$. Finally, \cref{lem:interpolation-current} gives $I_{v,T}(G)<\infty$ and $J_T(G,W)<\infty$, as well as the controlled equation $(\partial_s+\mathcal L_a)G_s=\nabla_v^*W_s$. Thus, \ref{gm:P2} and \ref{gm:P3} also hold, and the pair $(G,W)$ is admissible.
\end{proof}

We can now apply the finite-action space-time LSI to the interpolation path $G$. A direct estimate of $\frac{\ud}{\ud s}\Ent_\mu(G_s)$ then converts the resulting time-averaged entropy bound into a bound on the endpoint entropies $\Ent_\mu(G_0)$ and $\Ent_\mu(G_T)$.

\begin{proposition}\label{prop:interpolation-endpoint}
Let $\varphi,\psi$ be smooth and bounded above and below by positive constants, and let $G_s$ be defined by \eqref{def:interpolationdensity}. Then $G$ satisfies the time-averaged entropy
bound
\begin{equation}\label{eq:interpolation-average}
    \Ent_{\muT}(G) \le  2C_{\mathrm{ST}}\!\left(1 + \frac{\gamma^2}{\rho}\right)
    \mathcal D_T(\varphi, \psi),
\end{equation}
and the endpoint entropy bound
\begin{equation}\label{eq:interpolation-endpoint}
    \Ent_\mu(G_0) + \Ent_\mu(G_T)
    \le K_T\,\mathcal D_T(\varphi, \psi),
\end{equation}
where
\begin{equation}\label{eq:KT}
    K_T := 4C_{\mathrm{ST}}\!\left(1 + \frac{\gamma^2}{\rho}\right) + 2\gamma T.
\end{equation}
\end{proposition}
\begin{proof}
By \cref{lem:interpolation-current,lem:interpolation-admissible}, the pair $(G,W)$ is admissible. The time-averaged bound~\eqref{eq:interpolation-average} follows from \cref{gm:cor:stlsi} applied to $G$, together with \eqref{eq:interpolation-norms}:
\begin{equation*}
    \Ent_{\muT}(G) \le 2 C_{\rm ST}\left(1+\frac{\gamma^2}{\rho}\right) \mathcal D_T(\varphi,\psi).
\end{equation*}

For the endpoint bound~\eqref{eq:interpolation-endpoint}, \cref{lem:entropy-chain-rule}, applied to the interpolation equation \eqref{eq:interpolation-current}, shows that $s\mapsto\Ent_\mu(G_s)$ is absolutely continuous and
\begin{equation*}
    \frac{\ud}{\ud s}\Ent_\mu(G_s)
    = \int G_s A_s \cdot \nabla_v \log G_s\ud\mu.
\end{equation*}
Substituting $A_s = -\gamma(\nabla_v \log \varphi_s - \nabla_v \log \psi_s)$
and $\nabla_v \log G_s = \nabla_v \log \varphi_s + \nabla_v \log \psi_s$, the integrand simplifies to
\begin{equation*}
  G_s A_s \cdot \nabla_v \log G_s =   -\gamma G_s\bigl(\abs{\nabla_v \log \varphi_s}^2
    - \abs{\nabla_v \log \psi_s}^2\bigr),
\end{equation*}
and hence
\begin{equation}\label{eq:entropy-derivative-bound}
    \Abs{\frac{\ud}{\ud s}\Ent_\mu(G_s)}
    \le \gamma \int G_s\bigl(\abs{\nabla_v \log \varphi_s}^2
    + \abs{\nabla_v \log \psi_s}^2\bigr)\ud\mu.
\end{equation}
Integrating over $[0, T]$ gives
\begin{equation*}
    \int_0^T \Abs{\frac{\ud}{\ud s}\Ent_\mu(G_s)}\ud s
    \le \gamma T\,\mathcal D_T(\varphi, \psi).
\end{equation*}
For each endpoint $a \in \{0, T\}$ and every $r \in [0, T]$,
\begin{equation*}
    \Ent_\mu(G_a) \le \Ent_\mu(G_r)
    + \int_0^T \Abs{\frac{\ud}{\ud s}\Ent_\mu(G_s)}\ud s
    \le \Ent_\mu(G_r) + \gamma T\,\mathcal D_T(\varphi, \psi).
\end{equation*}
Averaging in $r$ over $[0, T]$ yields
\begin{equation*}
    \Ent_\mu(G_a) \le \Ent_{\muT}(G) + \gamma T\,\mathcal D_T(\varphi, \psi).
\end{equation*}
Adding the above two endpoint inequalities ($a = 0$ and $a = T$) and applying~\eqref{eq:interpolation-average} gives~\eqref{eq:interpolation-endpoint}, with $K_T$ as in~\eqref{eq:KT}. The proof is complete.
\end{proof}

After normalizing $\varphi$ and $\psi$ in $L^p(\mu)$ and $L^{q'}(\mu)$, respectively, the following lemma bounds $\log Z$ by a weighted combination of $\Ent_\mu(G_0)$ and $\Ent_\mu(G_T)$, minus the dissipation term $\gamma T\,\mathcal D_T(\varphi,\psi)$. In \cref{sec:hypercontractivity}, this bound will be combined with \cref{prop:interpolation-endpoint} to conclude $Z\le1$ for an appropriate choice of exponents $p$ and $q$ (see \cref{prop:central} below).

\begin{proposition}\label{prop:closure}
Let $1<p,q<\infty$, let $q'=q/(q-1)$, and let $\varphi,\psi$ be smooth functions
bounded above and below by positive constants satisfying
\begin{equation*}
    \norm{\varphi}_{L^p(\mu)}=1,
    \qquad
    \norm{\psi}_{L^{q'}(\mu)}=1.
\end{equation*}
Then
\begin{equation}\label{eq:closure}
    2\log Z \le  \left(\frac{2}{p} - 1\right)\Ent_\mu(G_0)  + \left(1 - \frac{2}{q}\right)\Ent_\mu(G_T)
    - \gamma T\,\mathcal D_T(\varphi, \psi).
\end{equation}
\end{proposition}
\begin{proof}
First,  by the definition of $G_s$, we find
\begin{equation} \label{eq:entropyGs}
    \Ent_\mu(G_s) = \int G_s\log\varphi_s\ud\mu + \int G_s\log\psi_s\ud\mu - \log Z.
\end{equation}

Using $\partial_sG_s=-\La G_s+\nabla_v^{*}(G_sA_s)$ and $\partial_s\varphi_s=-\La\varphi_s+\gamma\Ls\varphi_s$, we compute
\begin{equation} \label{eq:timederivative}
    \begin{aligned}
              \frac{\ud}{\ud s}\int G_s\log\varphi_s\ud\mu
        &=
        \int(-\La G_s)\log\varphi_s\ud\mu
        +\int\nabla_v^{*}(G_sA_s)\log\varphi_s\ud\mu        \\
        &\quad
        +\int G_s\frac{-\La\varphi_s+\gamma\Ls\varphi_s}{\varphi_s}\ud\mu.
    \end{aligned}
\end{equation}
From the chain rule:
\begin{equation*}
    \La(\log f) = \frac{\La f}{f}\,,
\end{equation*}
and the skew-adjointness of $\La$, the two $\La$ contributions cancel:
\begin{align*}
     \int(-\La G_s)\log\varphi_s\ud\mu + \int G_s\frac{-\La\varphi_s}{\varphi_s}\ud\mu = \int G_s\frac{\La\varphi_s}{\varphi_s}\ud\mu + \int G_s\frac{-\La\varphi_s}{\varphi_s}\ud\mu = 0\,.
\end{align*}
The identity \eqref{eq:timederivative} thus simplifies to
\begin{equation*}
        \frac{\ud}{\ud s}\int G_s\log\varphi_s\ud\mu
        =
        \int G_sA_s\cdot\nabla_v\log\varphi_s\ud\mu
        +\gamma\int G_s\frac{\Ls\varphi_s}{\varphi_s}\ud\mu.
\end{equation*}
Next, using $\Ls=-\nabla_v^{*}\nabla_v$ and $G_s/\varphi_s=\psi_s/Z$, we compute
\begin{equation*}
        \int G_s\frac{\Ls\varphi_s}{\varphi_s}\ud\mu
        =-\int\nabla_v\!\left(\frac{\psi_s}{Z}\right)\cdot\nabla_v\varphi_s\ud\mu
        =-\int G_s\nabla_v\log\varphi_s\cdot\nabla_v\log\psi_s\ud\mu,
\end{equation*}
and since
$A_s=-\gamma(\nabla_v\log\varphi_s-\nabla_v\log\psi_s)$, we obtain
\begin{equation} \label{eq:gsvp}
        \frac{\ud}{\ud s}\int G_s\log\varphi_s\ud\mu
        =-\gamma\int G_s\abs{\nabla_v\log\varphi_s}^2\ud\mu.
\end{equation}
The analogous computation with
$\partial_s\psi_s=-\La\psi_s-\gamma\Ls\psi_s$ gives
\begin{equation} \label{eq:gspsi}
        \frac{\ud}{\ud s}\int G_s\log\psi_s\ud\mu
        =
        \gamma\int G_s\abs{\nabla_v\log\psi_s}^2\ud\mu.
\end{equation}

Integrating these two identities \eqref{eq:gsvp}-\eqref{eq:gspsi} over $[0,T]$ and then subtracting gives
\begin{equation}\label{eq:closure-energy-identity}
\gamma T\mathcal D_T(\varphi,\psi)  = \int G_0\log\varphi\ud\mu  - \int G_T\log\varphi_T\ud\mu + \int G_T\log\psi\ud\mu - \int G_0\log\psi_0\ud\mu,
\end{equation}
due to the definition \eqref{eq:interpolation-energy} of $\mathcal D_T(\varphi,\psi)$ and endpoints $\varphi_0=\varphi$ and $\psi_T=\psi$. Then, recalling \eqref{eq:entropyGs}, the identity \eqref{eq:closure-energy-identity} can be rewritten as
\begin{equation} \label{eq:energy_zgvp}
\gamma T\mathcal D_T(\varphi,\psi) = 2\int G_0\log\varphi\ud\mu  + 2\int G_T\log\psi\ud\mu -\Ent_\mu(G_0)-\Ent_\mu(G_T)-2\log Z.
\end{equation}

Finally, to relate $\int G_0\log\varphi\ud\mu$ and $\int G_T\log\psi\ud\mu$ with entropies
$\Ent_\mu(G_0)$ and $\Ent_\mu(G_T)$, respectively, we use the variational inequality for $\Ent_\mu(\cdot)$: for any probability density $g$ with respect to $\mu$ and any measurable function $\phi$,
\begin{equation}\label{eq:variational}
    \int \phi\, g\ud\mu \le \Ent_\mu(g) + \log\int \e^{\phi}\ud\mu.
\end{equation}
Since $\norm{\varphi}_{L^p(\mu)} = \norm{\psi}_{L^{q'}(\mu)} = 1$, applying~\eqref{eq:variational} to $\phi = p\log\varphi$ and $g = G_0$ gives
\begin{equation*}
    p\int G_0\log\varphi\ud\mu  \le \Ent_\mu(G_0) + \log\int \varphi^p\ud\mu = \Ent_\mu(G_0),
\end{equation*}
which yields, by dividing by $p$,
\begin{equation} \label{auxeq1}
    \int G_0\log\varphi\ud\mu \le \frac{1}{p}\Ent_\mu(G_0).
\end{equation}
The analogous application to $\phi = q'\log\psi$ and $g = G_T$, combined with $\norm{\psi}_{L^{q'}(\mu)}^{q'} = 1$, gives
\begin{equation} \label{auxeq2}
    \int G_T\log\psi\ud\mu \le \frac{1}{q'}\Ent_\mu(G_T).
\end{equation}
Substituting \eqref{auxeq1} and \eqref{auxeq2} into \eqref{eq:energy_zgvp} and using $1/q' = 1 - 1/q$ gives~\eqref{eq:closure}.
\end{proof}

\section{Hypocoercive hypercontractivity}\label{sec:hypercontractivity}

We now utilize the endpoint entropy bound \eqref{eq:interpolation-endpoint} in \cref{prop:interpolation-endpoint} and the normalization bound \eqref{eq:closure} in \cref{prop:closure} to establish the hypercontractivity property for $\Pt_T$. We first prove the contraction at the critical pair $(p_{\rm c},q_{\rm c})$, at which the two bounds \eqref{eq:interpolation-endpoint} and \eqref{eq:closure} balance exactly and force the normalization constant to satisfy $Z\le1$.
The remaining $(p,q)$ exponents then follow by Riesz--Thorin interpolation with the $L^1\to L^1$ and $L^\infty\to L^\infty$ Markov contraction bounds for $\Pt_T$ (see \cref{thm:one-slab}).

\begin{proposition}[Critical hypercontractive estimate]\label{prop:central}
Let $K_T$ be the constant \eqref{eq:KT} for the endpoint entropy estimate \eqref{eq:interpolation-endpoint}. Introduce constants
\begin{equation}  \label{eq:criexponents}
        \eta_T=\frac{\gamma T}{K_T}, \qquad    p_{\rm c}=\frac2{1+\eta_T},   \qquad
        q_{\rm c}=\frac2{1-\eta_T}.
\end{equation}
Then $0<\eta_T<1/2$, and
\begin{equation*}
    1<p_{\rm c}<2<q_{\rm c}<\infty, \qquad q_{\rm c}'=p_{\rm c}.
\end{equation*}
Moreover, one has the hypercontractivity:
\begin{equation} \label{eq:critest}
    \norm{\Pt_Tf}_{L^{q_{\rm c}}(\mu)} \le \norm f_{L^{p_{\rm c}}(\mu)}\,, \quad \forall\, f\in L^{p_{\rm c}}(\mu)\,.
\end{equation}
\end{proposition}

\begin{proof}
From
\begin{equation*}
    K_T = 4C_{\mathrm{ST}} \left(1+\frac{\gamma^2}{\rho}\right)
+ 2\gamma T > 2\gamma T,
\end{equation*}
we immediately obtain
\begin{equation*}
    0<\eta_T=\frac{\gamma T}{K_T}<\frac12,
\end{equation*}
and that the critical exponents in \eqref{eq:criexponents} satisfy $1<p_{\rm c}<2<q_{\rm c}<\infty$.

We first prove \eqref{eq:critest} on positive smooth functions. Specifically, let
\begin{equation*}
    \mathcal C_+ := \left\{ \varepsilon+\phi:\varepsilon>0,\  \phi\in C_c^\infty(\mathbb R^{2d}),\ \phi\ge 0 \right\}.
\end{equation*}
Since $\mu$ is a probability measure with a smooth strictly positive Lebesgue density, $\mathcal C_+$ is dense in the nonnegative cone of $L^r(\mu)$ for any $1\le r<\infty$.

Let $\varphi\in\mathcal C_+$ satisfy $\|\varphi\|_{L^{p_{\rm c}}(\mu)}=1$. Since $\varphi$ is bounded and $\Pt_T$ is a Markov contraction, we have
\begin{equation*}
    0\le \Pt_T\varphi\le \|\varphi\|_{L^\infty(\mu)},
\end{equation*}
and hence $\Pt_T\varphi\in L^{q_{\rm c}}(\mu)$. By $L^{q_{\rm c}}$-duality and density of $\mathcal C_+$ in $L^{q_{\rm c}'}(\mu)$, there holds
\begin{equation}\label{eq:critical-duality}
\|\Pt_T\varphi\|_{L^{q_{\rm c}}(\mu)} =
\sup_{\psi\in\mathcal C_+}\left\{\int \Pt_T\varphi\,\psi \ud \mu: \  \psi\ge0,\ \|\psi\|_{L^{q_{\rm c}'}(\mu)}=1
\right\}.
\end{equation}
Fix such a $\psi$, and let $Z := \int \Pt_T\varphi\,\psi \ud \mu$ and
\begin{equation*}
    G_s = \frac{\varphi_s\psi_s}{Z}, \qquad \varphi_s=\Pt_s\varphi,
\qquad \psi_s=\Pt_{T-s}^*\psi,
\end{equation*}
be the forward/backward interpolation defined in
\eqref{def:interpolationdensity}. Applying \cref{prop:closure} with
$p=p_{\rm c}$ and $q=q_{\rm c}$, we obtain
\begin{align*}
2\log Z
&\le \left(\frac{2}{p_{\rm c}}-1\right) \Ent_\mu(G_0) + \left(1-\frac{2}{q_{\rm c}}\right) \Ent_\mu(G_T) - \gamma T\mathcal D_T(\varphi,\psi)
\\ & = \eta_T \bigl[ \Ent_\mu(G_0)+\Ent_\mu(G_T) \bigr] - \gamma T\mathcal D_T(\varphi,\psi).
\end{align*}
The endpoint entropy estimate \eqref{eq:interpolation-endpoint} then gives
\begin{align*}
2\log Z &\le \eta_TK_T\mathcal D_T(\varphi,\psi) - \gamma T\mathcal D_T(\varphi,\psi)  = 0,
\end{align*}
and hence
\begin{equation*}
    Z\le1.
\end{equation*}
Taking the supremum in \eqref{eq:critical-duality} yields, by homogeneity
\begin{equation}\label{eq:critical-core}
\|\Pt_T\varphi\|_{L^{q_{\rm c}}(\mu)} \le \|\varphi\|_{L^{p_{\rm c}}(\mu)}\,,  \quad \forall\,\varphi\in\mathcal C_+.
\end{equation}

We next extend \eqref{eq:critical-core} to a nonnegative
$f\in L^{p_{\rm c}}(\mu)$. Consider the sequence $\varphi_n\in\mathcal C_+$ such that $\varphi_n \to f$ in $L^{p_{\rm c}}(\mu)$. Since $\Pt_T$ is an $L^{p_{\rm c}}(\mu)$-contraction, we also have $\Pt_T\varphi_n \to \Pt_Tf$ in $L^{p_{\rm c}}(\mu)$. Up to a subsequence, one can assume $\Pt_T\varphi_n \to \Pt_Tf$ for $\mu$-\textup{a.e.} $(x,v)$. Fatou's lemma and \eqref{eq:critical-core} then imply
\begin{align*}
\|\Pt_Tf\|_{L^{q_{\rm c}}(\mu)}^{q_{\rm c}}
 = \int |\Pt_Tf|^{q_{\rm c}}\ud\mu & \le \liminf_{n\to\infty} \int |\Pt_T\varphi_n|^{q_{\rm c}}\ud\mu \\
& \le \liminf_{n\to\infty} \|\varphi_n\|_{L^{p_{\rm c}}(\mu)}^{q_{\rm c}} = \|f\|_{L^{p_{\rm c}}(\mu)}^{q_{\rm c}}.
\end{align*}
It readily follows that $\|\Pt_Tf\|_{L^{q_{\rm c}}(\mu)} \le \|f\|_{L^{p_{\rm c}}(\mu)}$ for any nonnegative $f\in L^{p_{\rm c}}(\mu)$.

Finally, for a general real-valued $f\in L^{p_{\rm c}}(\mu)$, positivity of the Markov semigroup gives
\begin{equation*}
    |\Pt_Tf| \le \Pt_T|f|.
\end{equation*}
Applying the contraction \eqref{eq:critest} to $|f|$, we conclude that
\begin{equation*}
    \|\Pt_Tf\|_{L^{q_{\rm c}}(\mu)} \le \|\Pt_T|f|\|_{L^{q_{\rm c}}(\mu)} \le \||f|\|_{L^{p_{\rm c}}(\mu)} = \|f\|_{L^{p_{\rm c}}(\mu)}.
\end{equation*}
The proof is complete.
\end{proof}

\begin{theorem}[One-step hypocoercive hypercontractivity]\label{thm:one-slab}
With $\eta_T$ as in \cref{prop:central}, set
\begin{equation} \label{def:alpha}
        \alpha_T=\frac{1+\eta_T}{1-\eta_T} > 1.
\end{equation}
Then for any $p>1$ and $f\in L^p(\mu)$,
\begin{equation}\label{eq:one-slab}
    \norm{\Pt_Tf}_{L^{1+\alpha_T(p-1)}(\mu)}  \le  \norm f_{L^p(\mu)}.
\end{equation}
Consequently, whenever $p\le q\le 1+\alpha_T(p-1)$, we have
\begin{equation} \label{eq:one-slab2}
    \norm{\Pt_Tf}_{L^q(\mu)}  \le  \norm f_{L^p(\mu)}.
\end{equation}
\end{theorem}

\begin{proof}
The proof is based on the critical estimate of \cref{prop:central} together with the
$L^1\to L^1$ and $L^\infty\to L^\infty$ Markov contraction bounds for
$\Pt_T$.

First let $1<p\le p_{\rm c}$. By the Markov contraction property and \cref{prop:central}, we have
\begin{equation*}
        \norm{\Pt_Tf}_{L^1(\mu)}
        \le
        \norm f_{L^1(\mu)}
        \qquad\text{and}\qquad
        \norm{\Pt_Tf}_{L^{q_{\rm c}}(\mu)}
        \le
        \norm f_{L^{p_{\rm c}}(\mu)}.
\end{equation*}
Hence, the Riesz--Thorin theorem applied to $\Pt_T$ with the endpoints $L^1 \to L^1$ and
$L^{p_{\rm c}} \to L^{q_{\rm c}}$ gives, for $0\le\beta\le1$,
\begin{equation} \label{eq:interpolation}
        \norm{\Pt_Tf}_{L^{q_\beta}(\mu)}
        \le
        \norm f_{L^{p_\beta}(\mu)},
        \qquad
        \frac1{p_\beta}=1-\beta+\frac{\beta}{p_{\rm c}},
        \quad
        \frac1{q_\beta}=1-\beta+\frac{\beta}{q_{\rm c}}.
\end{equation}
Since $p_{\rm c}>1$, the map $\beta\mapsto p_\beta$ is increasing from $1$ to
$p_{\rm c}$. Thus, for a given $p\in(1,p_{\rm c}]$, there is a unique
$\beta = \tfrac{2(p-1)}{p(1-\eta_T)}\in(0,1]$ such that $p=p_\beta$. Moreover, using \eqref{eq:criexponents}, or equivalently
$1/p_{\rm c}=(1+\eta_T)/2$ and $1/q_{\rm c}=(1-\eta_T)/2$, one can directly compute the interpolated exponents as
\begin{equation*}
    \frac1{p_\beta} = 1-\frac{\beta(1-\eta_T)}2,
\qquad \frac1{q_\beta} =  1-\frac{\beta(1+\eta_T)}2.
\end{equation*}
With the choice of \eqref{def:alpha}, we then have
\begin{equation*}
        \alpha_T(p_\beta-1)
        =
        \frac{\beta(1+\eta_T)}{2-\beta(1-\eta_T)} \le \frac{\beta(1+\eta_T)}{2-\beta(1+\eta_T)} =  q_\beta-1.
\end{equation*}
Therefore, at $p=p_\beta$, by the interpolation \eqref{eq:interpolation} and the monotonicity of $\norm{\cdot}_{L^r(\mu)}$,
\begin{equation*}
        \norm{\Pt_Tf}_{L^{1+\alpha_T(p-1)}(\mu)}
        \le
        \norm{\Pt_Tf}_{L^{q_\beta}(\mu)}
        \le
        \norm f_{L^p(\mu)}.
\end{equation*}
This proves \eqref{eq:one-slab} for $1<p\le p_{\rm c}$.

We now consider the case $p\ge p_{\rm c}$; the analysis is similar. Interpolating the critical estimate \eqref{eq:critest} with the
$L^\infty\to L^\infty$ contraction gives, for $0<\beta\le1$,
\begin{equation*}
        \norm{\Pt_Tf}_{L^{q_\beta}(\mu)}
        \le
        \norm f_{L^{p_\beta}(\mu)},
        \qquad
        \frac1{p_\beta}=\frac{\beta}{p_{\rm c}},
        \quad
        \frac1{q_\beta}=\frac{\beta}{q_{\rm c}}.
\end{equation*}
Now $p_\beta$ runs from $p_{\rm c}$ to $\infty$ as $\beta$ runs from $1$ down
to $0$, so we choose $\beta$ with $p=p_\beta$.  Since
$q_{\rm c}/p_{\rm c}=\alpha_T$, we have
$q_\beta=\alpha_T p_\beta$, and hence
\begin{equation*}
        q_\beta\ge 1+\alpha_T(p_\beta-1).
\end{equation*}
Monotonicity of $L^q(\mu)$ norms again gives \eqref{eq:one-slab}.

The final claim \eqref{eq:one-slab2} follows directly from \eqref{eq:one-slab} and the norm monotonicity:
\begin{equation*}
        \norm{\Pt_Tf}_{L^q(\mu)}  \le  \norm{\Pt_Tf}_{L^{1+\alpha_T(p-1)}(\mu)}    \le    \norm f_{L^p(\mu)}, \quad \text{for $1\le q\le 1+\alpha_T(p-1)$}. \qedhere
\end{equation*}
\end{proof}

By iteration, it is easy to extend the hypercontractivity \eqref{eq:one-slab2} for a fixed $T$ to any $t$.

\begin{corollary}[Iterated hypercontractivity]\label{cor:renyi}
Let $p>1$ and, for $n=0,1,2,\ldots$, set
\begin{equation} \label{def:pn}
        p_n=1+\alpha_T^n(p-1).
\end{equation}
Then, for any $n \ge 0$ and $f\in L^p(\mu)$,
\begin{equation*}
        \norm{\Pt_{nT}f}_{L^{p_n}(\mu)}
        \le
        \norm f_{L^p(\mu)}.
\end{equation*}
More generally, if $t\ge0$ and $n=\lfloor t/T\rfloor$, then
\begin{equation} \label{eq:targecontract}
        \norm{\Pt_tf}_{L^q(\mu)}
        \le
        \norm f_{L^p(\mu)}
\end{equation}
for every $1\le q\le p_n$.

In particular, fix $\Gamma>0$ and $\tau\ge\tau_{\rm ST}(\Gamma)$, and impose the kinetic scaling
\begin{equation*}
        \gamma=\Gamma\sqrt\rho,
        \qquad
        T=\tau\rho^{-1/2}.
\end{equation*}
Writing $\alpha_T$ as $\alpha_{\Gamma,\tau}$ and setting $\lambda_{\Gamma,\tau}:=(\log\alpha_{\Gamma,\tau})/\tau$, the largest exponent reached by the iteration at time $t$:
\begin{equation*}
    q_*(t) := 1+(p-1)\,\alpha_{\Gamma,\tau}^{\lfloor\sqrt\rho\,t/\tau\rfloor},
\end{equation*}
satisfies the two-sided bound
\begin{equation} \label{eq:twosidedebound}
  (p-1)\alpha_{\Gamma,\tau}^{-1} \exp\!\left(\lambda_{\Gamma,\tau}\sqrt\rho\,t\right) \le q_*(t)-1 \le  (p-1)\exp\!\left(\lambda_{\Gamma,\tau}\sqrt\rho\,t\right).
\end{equation}
\end{corollary}

\begin{proof}
The definition~\eqref{def:pn} gives the recursion
\begin{equation*}
        p_{k+1}=1+\alpha_T(p_k-1).
\end{equation*}
Applying \eqref{eq:one-slab} with input exponent $p_k$ to the function
$\Pt_{kT}f$ gives
\begin{equation*}
        \norm{\Pt_{(k+1)T}f}_{L^{p_{k+1}}(\mu)} = \norm{\Pt_T(\Pt_{kT}f)}_{L^{p_{k+1}}(\mu)}
        \le
        \norm{\Pt_{kT}f}_{L^{p_k}(\mu)}.
\end{equation*}
Iterating this inequality for $k=0,\ldots,n-1$ yields
\begin{equation*}
        \norm{\Pt_{nT}f}_{L^{p_n}(\mu)}
        \le
        \norm f_{L^p(\mu)}.
\end{equation*}
For arbitrary $t\ge0$, write $t=nT+r$ with
$n=\lfloor t/T\rfloor$ and $0\le r<T$.  The semigroup property gives
$\Pt_t=\Pt_r\Pt_{nT}$, and the Markov contraction property gives
\begin{equation*}
        \norm{\Pt_tf}_{L^{p_n}(\mu)}
        =
        \norm{\Pt_r(\Pt_{nT}f)}_{L^{p_n}(\mu)}
        \le
        \norm{\Pt_{nT}f}_{L^{p_n}(\mu)}
        \le
        \norm f_{L^p(\mu)}.
\end{equation*}
Finally, since the $L^s(\mu)$-norm is nondecreasing in $s$ on the probability space $(\RR^{2d},\mu)$, the same estimate extends to every output exponent $1\le q\le p_n$, which is the desired~\eqref{eq:targecontract}.

Under the kinetic scaling, the identities $\gamma T=\Gamma\tau$ and $\gamma^2/\rho=\Gamma^2$, together with \eqref{eq:KT} and \eqref{def:alpha}, show that $\alpha_T$ depends only on $\Gamma$ and $\tau$. We denote this value by $\alpha_{\Gamma,\tau}$. Since $t/T=\sqrt\rho\,t/\tau$, the exponent $p_n$ from \eqref{def:pn} with $n=\lfloor t/T\rfloor$ is exactly $q_*(t)$. The two-sided bound \eqref{eq:twosidedebound} on $q_*(t)-1$ then follows from
\begin{equation*}
        x-1\le\lfloor x\rfloor\le x,
        \qquad
        x=\frac{\sqrt\rho\,t}{\tau},
\end{equation*}
together with the definition of $\lambda_{\Gamma,\tau}$. The proof is complete.
\end{proof}

\begin{remark}[Kinetic scaling]\label{rem:kinetic-scaling}
Recall the kinetic time scale $ \gamma = \Gamma\sqrt{\rho}$ and $T = \tau\rho^{-1/2}$ on which \cref{cor:renyi} applies. Under this scaling, the two dimensionless constants entering $K_T$ in~\eqref{eq:KT} are $\gamma T = \Gamma\tau$ and $ \frac{\gamma^2}{\rho} = \Gamma^2$. Hence, the constants
\begin{equation*}
    K_T = K_{\Gamma,\tau} := 4C_{\mathrm{ST}}(\Gamma)(1 + \Gamma^2) + 2\Gamma\tau,
    \qquad
    \eta_T = \eta_{\Gamma,\tau} := \frac{\Gamma\tau}{K_{\Gamma,\tau}},
\end{equation*}
and
\begin{equation*}
    \alpha_T = \alpha_{\Gamma,\tau}
    := \frac{1 + \eta_{\Gamma,\tau}}{1 - \eta_{\Gamma,\tau}} > 1,
    \qquad
    \lambda_{\Gamma,\tau} := \frac{\log\alpha_{\Gamma,\tau}}{\tau},
\end{equation*}
depend only on $\Gamma$ and $\tau$, and in particular are independent of $\rho$.

Thus the LSI constant $\rho$ enters the iterated hypercontractive exponents in \cref{cor:renyi} only through the kinetic clock $\sqrt{\rho}\,t$. The resulting integrability gain therefore occurs on the ballistic time scale $t\sim\rho^{-1/2}$, consistent with the sharp hypocoercive entropy-decay rate of~\cite{Lu2026} and faster than the diffusive scale $\rho^{-1}$ that the LSI~\eqref{eq:lsi} yields in the overdamped case.
\end{remark}

We now translate the above hypocoercive hypercontractivity into statements about
the R\'enyi divergence decay. Recall that, for $r>1$, the R\'enyi divergence of order $r$ of a probability measure $\nu$ with respect to $\mu$ is
\begin{equation*}
        \Ren_r(\nu\mid\mu)
        =
        \frac1{r-1}\log\int\left(\frac{\ud\nu}{\ud\mu}\right)^r\ud\mu
        =
        \frac r{r-1}\log\norm{\ud\nu/\ud\mu}_{L^r(\mu)}
\end{equation*}
if $\nu\ll\mu$, and $+\infty$ otherwise.
The next corollary is essentially a reformulation of our hypercontractivity result in \cref{cor:renyi}, simultaneously yielding the order-improvement (hypercontractive) bound and the long-time decay estimate for $\Ren_r(\nu_t\mid\mu)$.

\begin{corollary}[Hypocoercive R\'enyi hypercontractivity and decay]\label{cor:renyi-decay}
Let $h$ be a probability density, set $\nu_t=(\Pt_t h)\,\mu$, and let $n=\lfloor t/T\rfloor$.
\begin{enumerate}
\item\emph{(R\'enyi hypercontractivity.)} For any $p>1$ and any $q$ with $1<p\le q\le p_n =1+\alpha_T^n(p-1)$,
\begin{equation}\label{eq:renyi}
        \Ren_q(\nu_t\mid\mu)
        \le
        \frac{q(p-1)}{p(q-1)}
        \Ren_p(\nu_0\mid\mu).
\end{equation}
\item\emph{(R\'enyi divergence decay.)} For every $q>1$ and every $t\ge0$,
\begin{equation}\label{eq:renyi-long-time}
        \Ren_q(\nu_t\mid\mu)
        \le
        q\alpha_T
        \e^{-(\log\alpha_T)t/T}
        \Ren_q(\nu_0\mid\mu).
\end{equation}
In particular, if $T=\tau\rho^{-1/2}$ with $\tau\ge\tau_{\rm ST}(\Gamma)$, then
\begin{equation}\label{eq:renyi-kinetic-decay}
\Ren_q(\nu_t\mid\mu) \le q\alpha_{\Gamma,\tau}\, \e^{-\lambda_{\Gamma,\tau}\sqrt\rho\,t} \Ren_q(\nu_0\mid\mu),
\end{equation}
where $\alpha_T = \alpha_{\Gamma,\tau}$ and $ \lambda_{\Gamma,\tau}=\frac{\log\alpha_{\Gamma,\tau}}{\tau}$ are as in \cref{cor:renyi}.
\end{enumerate}
\end{corollary}

\begin{proof}
\emph{(i).} By \cref{cor:renyi} and monotonicity of $L^q$ norms on the probability space,
\begin{equation*}
    \norm{\Pt_t h}_{L^q(\mu)}\le\norm{h}_{L^p(\mu)}\,, \quad \text{whenever $q\le p_n$}.
\end{equation*}
Using $\Ren_r(\nu\mid\mu)=\tfrac r{r-1}\log\norm{\ud\nu/\ud\mu}_{L^r(\mu)}$, we obtain \eqref{eq:renyi}:
\begin{equation*}
        \Ren_q(\nu_t\mid\mu)
        \le
        \frac q{q-1}\log\norm h_{L^p(\mu)}
        =
        \frac{q(p-1)}{p(q-1)}
        \Ren_p(\nu_0\mid\mu).
\end{equation*}

\medskip

\noindent
\emph{(ii).} Given $q>1$, choose $p=1+\alpha_T^{-n}(q-1)$, so that $q=1+\alpha_T^n(p-1)\le p_n$. Applying (i) at this $(p,q)$ gives
\begin{equation*}
        \Ren_q(\nu_t\mid\mu)
        \le
        \frac{q\alpha_T^{-n}}{1+(q-1)\alpha_T^{-n}}
        \Ren_p(\nu_0\mid\mu).
\end{equation*}
Since $p\le q$, monotonicity of R\'enyi divergence in the order gives $\Ren_p(\nu_0\mid\mu)\le\Ren_q(\nu_0\mid\mu)$. Moreover, $\alpha_T^{-n}\le\alpha_T\e^{-(\log\alpha_T)t/T}$, proving \eqref{eq:renyi-long-time}. Under the kinetic scaling $T=\tau\rho^{-1/2}$, the identities $\gamma T=\Gamma\tau$ and $\gamma^2/\rho=\Gamma^2$ make $\alpha_T$ a function only of $\Gamma$ and $\tau$, yielding \eqref{eq:renyi-kinetic-decay}.
\end{proof}

\begin{remark}[Rate degradation at extreme friction]\label{rem:rate-asymp}
We trace the asymptotic behavior of the hypocoercive R\'enyi-decay rate $\lambda_{\Gamma,\tau}\sqrt\rho$ at extreme friction, taking $\tau$ at its minimal admissible value $\tau=\tau_{\rm ST}(\Gamma)$.

With $M:=\max\{\Gamma,\Gamma^{-1}\}$, substituting the asymptotics $\tau_{\rm ST}(\Gamma)\asymp M$ and $C_{\rm ST}(\Gamma)\asymp M^2$ from \cref{rem:stlsi-constants} into the constants $K_T=4C_{\rm ST}(\Gamma)(1+\Gamma^2)+2\Gamma\tau$ and $\eta_T=\Gamma\tau/K_T$ of \cref{rem:kinetic-scaling} yields
\begin{equation*}
    K_T \asymp
    \begin{cases}
        M^4 & \text{as }\Gamma\to\infty\ (M=\Gamma),\quad\text{with }1+\Gamma^2\asymp M^2,\ \Gamma\tau\asymp M^2,\\
        M^2 & \text{as }\Gamma\to 0\ (M=\Gamma^{-1}),\quad\text{with }1+\Gamma^2\asymp 1,\ \Gamma\tau\asymp 1.
    \end{cases}
\end{equation*}
While $K_T$ scales asymmetrically in the two extremes, in both cases
\begin{equation*}
    \eta_T=\frac{\Gamma\tau}{K_T}\asymp M^{-2},
    \qquad
    \log\alpha_T=\log\frac{1+\eta_T}{1-\eta_T}=2\eta_T+O(\eta_T^3)\asymp M^{-2},
    \qquad
    \lambda_{\Gamma,\tau_{\rm ST}(\Gamma)}
    =\frac{\log\alpha_T}{\tau_{\rm ST}(\Gamma)}
    \asymp M^{-3}.
\end{equation*}
The hypocoercive R\'enyi-decay rate $\lambda_{\Gamma,\tau}\sqrt\rho$ therefore degrades cubically in $\max\{\Gamma,\Gamma^{-1}\}$ at extreme friction, identifying
$\gamma\sim \sqrt{\rho}$ as the optimal friction scaling again (see also \cref{rem:stlsi-constants}).

\end{remark}

\appendix
\crefalias{section}{appendix}

\section{Proof of the weak Wasserstein acceleration inequality}
\label{app:wasserstein-acceleration}

\begin{proof}[Proof of \cref{lem:wasserstein-acceleration}]
Let
\begin{equation*}
    D(t):=\frac{1}{2} W_2^2( q_t\mux,\mux),
    \qquad
    \mathcal C(t):=\int\xi_{q_t}\cdot m_t\ud q_t\mux=\int j_t\cdot\xi_{q_t}\ud\mux.
\end{equation*}
Since $m$ is smooth and bounded by \eqref{eq:A2-bounds}, the flow $\dot X_s=m_{t+s}(X_s)$, $X_0=\operatorname{id}$, is globally well posed and transports $ q_t\mux$ to $q_{t+s} \mux$, by the method of characteristics for \eqref{eq:continue}. Hence $t\mapsto q_t\mux$ is $W_2$-absolutely continuous with metric derivative bounded by $\norm{m_t}_{L^2( q_t\mux)}$, and the first-variation formula for the squared distance $W_2^2$ \cite{ambrosio2005gradient}*{Theorem 8.4.7} gives $$D'(t)=\mathcal C(t)$$ for \textup{a.e.}\ $t$, with $D$ locally Lipschitz. Thus, it suffices to prove, in $\mathcal D'((0,T))$, that
\begin{equation} \label{targeteq}
    \frac{\ud}{\ud t}\mathcal C(t)\le\int\abs{m_t}^2\ud q_t\mux+\int\xi_{q_t}\cdot a_t\ud q_t\mux =: Q(t)
\end{equation}
\emph{Step 1: Flow competitor and second-order expansion.}
Fix $t$ and take $s$ sufficiently small. Noting that $(X_s)_\# (q_t\mux)= q_{t+s} \mux$ and $(T_{q_t})_\# (q_t\mux)=\mux$, the measure
$(X_s,T_{q_t})_\# (q_t\mux)$ couples $ q_{t+s} \mux$ and $\mux$, implying
\begin{equation}\label{eq:wa-coupling}
    D(t+s)\le\Phi_t(s):=\frac12\int\abs{X_s-T_{q_t}}^2\ud q_t\mux,
    \qquad
    \Phi_t'(s)=\int(X_s-T_{q_t})\cdot m_{t+s}(X_s)\ud q_t\mux.
\end{equation}
Since $s\mapsto m_{t+s}(X_s)$ is $C^1$ with derivative $a_{t+s}(X_s)$, and \eqref{eq:A2-bounds} bounds $\|a_r\|_{L^2(q_r\mu_x)}$ and $\|m_r\|_{L^\infty}$ locally uniformly for $r$ in compact time intervals, $\Phi_t'$ is absolutely continuous. Using $(X_s)_\# (q_t\mux)= q_{t+s} \mux$ and $X_s-T_{q_t} = X_s- T_{q_{t+s}}(X_s) +  T_{q_{t+s}}(X_s) - T_{q_t}$, we have, for \textup{a.e.}\ $s$,
\begin{equation}\label{eq:wa-error}
    \Phi_t''(s)=Q(t+s)+e_t(s),
    \quad
    e_t(s):=\int\Delta_{t,s}\cdot a_{t+s}(X_s)\ud q_t\mux,
\end{equation}
with $\Delta_{t,s}:=T_{q_{t+s}}(X_s) - T_{q_t}$ and
\begin{equation} \label{est:es}
    \abs{e_t(s)}\le\norm{\Delta_{t,s}}_{L^2( q_t\mux)}\norm{a_{t+s}}_{L^2( q_{t+s} \mux)}.
\end{equation}
Here $\norm{\Delta_{t,s}}_{L^2( q_t\mux)}\le2(\int\abs y^2\mux(\ud y))^{1/2} < + \infty$ for all $(t,s)$, since both $T_{q_{t+s}}(X_s)$ and $T_{q_t}$ push $q_t \mux$ onto $\mux$, whose second moment is finite.

\medskip
\noindent \emph{Step 2: Stability of the transported Brenier maps.}
We claim $\norm{\Delta_{t,s}}_{L^2( q_t\mux)}\to0$ as $s\to0$. Set $Y_s:=T_{q_{t+s}}\circ X_s$ and $\pi_s:=(\operatorname{id},Y_s)_\# q_t\mux$. Then $\pi_s$ couples $( q_t\mux,\mux)$, and by the triangle inequality in $L^2( q_t\mux)$,
\begin{align*}
    W_2( q_t\mux,\mux) \le  \operatorname{cost}(\pi_s)^{1/2} & := \Big( \int |x - y|^2 \ud \pi_s(x,y) \Big)^{1/2} = \norm{x-  T_{q_{t+s}}(X_s)}_{L^2( q_t \mux)} \\
    & \le\norm{x-X_s}_{L^2( q_t\mux)}+\sqrt{2D(t+s)}\,.
\end{align*}
Hence, the $\pi_s$ are asymptotically optimal couplings of the pair $( q_t\mux,\mux)$:
\begin{equation} \label{eq:conver_couple}
    \lim_{s\to0}\operatorname{cost}(\pi_s) =  W_2^2( q_t\mux,\mux).
\end{equation}
Since the marginals of $\pi_s$ are fixed and equal to $q_t\mux,\mux$ for every $s$, the family $\pi_s$ is tight, uniformly in $s$, and hence narrowly relatively compact, by Prokhorov's theorem.
Let $\pi$ be any narrow limit point along a
sequence $s_k\to0$. By continuity of the marginal maps under narrow convergence, $\pi$ is itself a coupling of $(q_t\mux,\mux)$. Moreover, by \eqref{eq:conver_couple}, we have
\begin{equation*}
\int|x-y|^2\ud\pi \le \liminf_{k\to\infty}\operatorname{cost}(\pi_{s_k}) = W_2^2(q_t\mux,\mux),
\end{equation*}
and hence $\pi$ is the optimal coupling $\pi = (\operatorname{id},T_{q_t})_\#\,q_t\mux$. It follows that the whole tight family $\{\pi_s\}$ converges: as $s \to 0$, $\pi_s \to (\operatorname{id},T_{q_t})_\#\,q_t\mux$ narrowly, and $\operatorname{cost}(\pi_s) \to W_2^2(q_t\mux,\mux)$. To upgrade this to convergence of the maps, note that $Y_s$ pushes $q_t\mux$ forward to $\mux$ and
\begin{equation*}
\|Y_s\|_{L^2(q_t\mux)}^2 = \|T_{q_t}\|^2_{L^2(q_t\mux)} = \int|y|^2\,\mux(\ud y)  < \infty.
\end{equation*}
This uniform second-moment bound, along with the narrow convergence of $\pi_s$, implies $Y_s \rightharpoonup T_{q_t}$ weakly in $L^2(q_t\mux;\mathbb R^d)$, which further gives the strong $L^2$ convergence:
\begin{equation*}
\|Y_s - T_{q_t}\|^2_{L^2(q_t\mux)} = \|Y_s\|^2_{L^2(q_t\mux)} - 2\langle Y_s,T_{q_t}\rangle_{L^2(q_t\mux)} + \|T_{q_t}\|^2_{L^2(q_t\mux)} \to 0,
\end{equation*}
proving the claim.

\medskip
\noindent
\emph{Step 3: Symmetric difference quotient.}
We now integrate \eqref{eq:wa-error} twice in $s$ and combine the forward and backward expansions to cancel the indefinite first-order term $\mathcal C(t)$. Specifically, for $|s|$ sufficiently small,
\begin{equation*}
    \Phi_t(s)=D(t)+s\mathcal C(t)+\int_0^s\int_0^\sigma\bigl[Q(t+\tau)+e_t(\tau)\bigr]\ud\tau\ud\sigma,
\end{equation*}
where we used $\Phi_t(0)=D(t)$, $\Phi_t'(0)=\mathcal C(t)$, and the definition of $Q$ in \eqref{targeteq}. Taking $s=h>0$ sufficiently small and using $D(t+h)\le\Phi_t(h)$ from \eqref{eq:wa-coupling} gives the forward estimate
\begin{equation}\label{eq:forward-bound}
    D(t+h) \le D(t)+h\mathcal C(t)+\int_0^h(h-\tau)\bigl[Q(t+\tau)+e_t(\tau)\bigr]\ud\tau.
\end{equation}
Similarly, taking $s = -h$ with $D(t-h)\le\Phi_t(-h)$ yields the backward estimate
\begin{equation}\label{eq:backward-bound}
    D(t-h) \le D(t)-h\mathcal C(t)+\int_0^h(h-\tau)\bigl[Q(t-\tau)+e_t(-\tau)\bigr]\ud \tau.
\end{equation}
Adding \eqref{eq:forward-bound} and \eqref{eq:backward-bound} gives
\begin{equation*}
    D(t+h)+D(t-h)-2D(t)\le\int_0^h(h-s)\bigl[Q(t+s)+Q(t-s)\bigr]\ud s+\veps_t(h),
\end{equation*}
where $\veps_t(h):=\int_0^h(h-s)\bigl[\abs{e_t(s)}+\abs{e_t(-s)}\bigr]\ud s$ denotes the error term.
Test against $0\le\omega\in C^\infty_c((0,T))$ and choose $h_0>0$ small enough that the closed $h_0$-neighborhood of $\operatorname{supp}\omega$ is contained in $(0,T)$. Dividing by $h^2$ and letting $h\downarrow0$, the left side tends to $\int D\,\omega''$, while
\begin{equation*}
    \frac{1}{h^2} \int_0^T \int_0^h(h-s)\bigl[Q(t+s)+Q(t-s)\bigr]\ud s \, \omega(t) \ud t\to  \int_0^T Q\,\omega \ud t,
\end{equation*}
and, by \eqref{est:es} and $\norm{\Delta_{t,s}}_{L^2( q_t\mux)}\to0$ as $s\to0$,
\begin{equation*}
   \frac1{h^2}\int\omega\,\veps_t(h)\ud t \le \sup_{\operatorname{dist}(t,\operatorname{supp}\omega)\le h_0}\norm{a_t}_{L^2(q_t\mux)}\int_0^1(1-\sigma)\int\omega(t)\Bigl[\norm{\Delta_{t,h\sigma}}_{L^2(q_t\mux)}+\norm{\Delta_{t,-h\sigma}}_{L^2(q_t\mux)}\Bigr]\ud t\ud\sigma \to 0.
\end{equation*}
Therefore, $D''\le Q$ in $\mathcal D'((0,T))$, and then the desired inequality \eqref{targeteq} is by $D'=\mathcal C$.
\end{proof}

\section{Finite-action recovery sequences}\label{app:recovery}
This appendix is devoted to the proof of the recovery theorem \cref{gm:thm:recovery}.

\subsection{Setting and elementary lemmas}\label{gm:sec:setting}

We use the setting and notation of \cref{sec:setting,sec:stlsi}. We first pass to Lebesgue densities. In these variables, Euclidean Gaussian convolution preserves mass and positivity, while the action and the velocity Fisher information have convex formulations controlled by the convolution estimate \eqref{gm:eq:kernel-jensen} below.

Let $\varrho_\infty(x,v) = Z_\infty^{-1}\e^{-U(x)-\abs v^2/2}$ denote the Lebesgue density of $\mu$, and define the corresponding Lebesgue density $\varrho_t$ and velocity current $F_t$ by
\begin{equation}\label{gm:eq:dictionary}
    \varrho_t = g_t\varrho_\infty,\qquad F_t = W_t\varrho_\infty .
\end{equation}
With $\varrho_t$ and $F_t$, one can rewrite the divergence equation \eqref{gm:eq:graph} as
\begin{equation}\label{gm:eq:lebesgue}
\partial_t\varrho + \divg_x(v\varrho) + \divg_v\bigl(-\nabla U\,\varrho+F\bigr)=0 \qquad\text{in }\mathcal D'((0,T)\times\RR^{2d}),
\end{equation}
and moreover
\begin{equation}\label{gm:eq:functionals-lebesgue}
    \int\frac{\abs{W_t}^2}{g_t}\ud\mu=\int\frac{\abs{F_t}^2}{\varrho_t}\dz,
    \qquad
     \Iv(g_t) = \int\frac{\abs{\nabla_vg_t}^2}{g_t}\ud\mu
    =\int\frac{\abs{\nabla_v\varrho_t+v\varrho_t}^2}{\varrho_t}\dz .
\end{equation}
Throughout, we denote by
\begin{equation*}
    \eta_a(z)= \frac{1}{(2\pi a^2)^{d/2}} \e^{-\abs z^2/(2a^2)}, \qquad z\in\RR^d\,,\quad a>0,
\end{equation*}
the centered Gaussian kernel on $\RR^d$. Convolution
in the $x$- and $v$-variables is denoted by $*_x$ and $*_v$, respectively. We recall the identity
\begin{equation}\label{gm:eq:gaussian-identity}
    \nabla\eta_a(z)=-\frac{z}{a^2}\eta_a(z),
    \qquad\text{equivalently,}\qquad  z\,\eta_a(z)=-a^2\nabla\eta_a(z).
\end{equation}

For ease of exposition, we collect some elementary lemmas for later use.

\begin{lemma}[Convolution decreases the quadratic action]\label{gm:lem:kernel-cs}
Let $\eta\ge0$ be an integrable function on $\RR^\ell$ with $\int \eta = 1$. Let $f\in L^1(\RR^\ell)$ be nonnegative, and $F:\RR^\ell\to\RR^m$ be measurable, with $F=0$ \textup{a.e.}\ on $\{f=0\}$ and $\int_{\RR^\ell}\abs F^2/f\ud x<\infty$. Then, it holds that
\begin{equation}\label{gm:eq:kernel-cs}
    \frac{\abs{\eta*F}^2}{\eta*f} \le \eta*\left(\frac{\abs F^2}{f}\right) \qquad\textup{a.e.},
\end{equation}
and hence
\begin{equation}\label{gm:eq:kernel-jensen}
    \int_{\RR^\ell}\frac{\abs{\eta*F}^2}{\eta*f}\ud x \le \int_{\RR^\ell}\frac{\abs F^2}{f}\ud x.
\end{equation}
\end{lemma}

\begin{proof}
It suffices to note from the Cauchy--Schwarz inequality with respect to the measure $\eta(x-y)f(y)\ud y$ that for any $x \in \RR^\ell$,
\begin{equation*}
    \abs{(\eta*F)(x)}^2 \le (\eta*f)(x)\, \left(\eta*\frac{\abs F^2}{f}\right)(x). \qedhere
\end{equation*}
\end{proof}

\begin{lemma}[Derivative bound of convolution]\label{gm:lem:gaussian-derivative}
For $f\in L^1(\RR^d)$, $f\ge0$, and $a>0$,
\begin{equation}\label{gm:eq:gaussian-derivative}
    \abs{\nabla(\eta_a*f)(x)}^2
    \le
    \frac{1}{a^4}\,(\eta_a*f)(x)\int\abs{x-y}^2\eta_a(x-y)f(y)\ud y,
\end{equation}
and consequently
\begin{equation}\label{gm:eq:gaussian-derivative-int}
    \int\frac{\abs{\nabla(\eta_a*f)}^2}{\eta_a*f}\ud x
    \le  \frac{d}{a^{2}}\int f \ud x.
\end{equation}
\end{lemma}
\begin{proof}
By \eqref{gm:eq:gaussian-identity}, we compute
\begin{equation*}
    \nabla(\eta_a*f)(x)=-a^{-2}\int(x-y)\eta_a(x-y)f(y)\ud y.
\end{equation*}
Similarly to the proof of \cref{gm:lem:kernel-cs}, the Cauchy--Schwarz inequality with respect to the measure $\eta_a(x-y)f(y)\ud y$ gives \eqref{gm:eq:gaussian-derivative}. The inequality \eqref{gm:eq:gaussian-derivative-int} follows from \eqref{gm:eq:gaussian-derivative} and $\int\abs z^2\eta_a(z)\ud z=da^2$.
\end{proof}

\begin{lemma}[Growth of potential]\label{gm:lem:tame-growth}
Under \cref{gm:ass:standing}, for all $x,z\in\RR^d$,
\begin{equation}\label{gm:eq:gronwall-gradU}
    1+\abs{\nabla U(x+z)}\le\bigl(1+\abs{\nabla U(x)}\bigr)\e^{C_U\abs z}, \qquad  \abs{\nabla U(x+z)-\nabla U(x)}\le\bigl(1+\abs{\nabla U(x)}\bigr)\bigl(\e^{C_U\abs z}-1\bigr),
\end{equation}
and
\begin{equation}\label{gm:eq:U-growth}
    U(x+z)\le U(x)+\bigl(1+\abs{\nabla U(x)}\bigr)\,\abs z\,\e^{C_U\abs z}.
\end{equation}
\end{lemma}

\begin{proof}
Let $\phi(s)=1+\abs{\nabla U(x+sz)}$, $s\in[0,1]$. Then $\phi$ is locally Lipschitz with
\begin{equation*}
    \abs{\phi'}\le\abs z\,\abs{\nabla^2U(x+sz)}\le C_U\abs z\,\phi
\end{equation*}
\textup{a.e.}\ by \eqref{gm:eq:tame}. Gronwall's inequality therefore gives, for every $s\in[0,1]$,
\begin{equation}\label{gm:eq:phi-gronwall}
    \phi(s) \le \phi(0)\e^{C_U\abs z s} =\bigl(1+\abs{\nabla U(x)}\bigr)\e^{C_U\abs z s}.
\end{equation}
Taking $s=1$ proves the first bound in \eqref{gm:eq:gronwall-gradU}. The second one follows from
\eqref{gm:eq:tame} and \eqref{gm:eq:phi-gronwall}:
\begin{align*}
    \abs{\nabla U(x+z)-\nabla U(x)}
    \le C_U\abs z\,\phi(0)\int_0^1\e^{C_U\abs z s}\ud s =\bigl(1+\abs{\nabla U(x)}\bigr)
      \bigl(\e^{C_U\abs z}-1\bigr).
\end{align*}
Finally, using \eqref{gm:eq:phi-gronwall} once more, we have \eqref{gm:eq:U-growth}:
\begin{align*}
    U(x+z)-U(x) &=\int_0^1 z\cdot\nabla U(x+sz)\ud s \\
    & \le \abs z\int_0^1\phi(s)\ud s \le \abs z\,\bigl(1+\abs{\nabla U(x)}\bigr)\e^{C_U\abs z}.  \qedhere
\end{align*}
\end{proof}

\begin{lemma}\label{gm:lem:exp-integrability}
Under \cref{gm:ass:lsi,gm:ass:standing}, the following statements hold:
\begin{enumerate}[label=\textup{(\roman*)}]
\item there exist constants $\alpha>0$ and $\beta\ge0$ such that
\begin{equation*}
    U(x)\ge\alpha\abs{x}-\beta,
\end{equation*}
and consequently $\int\e^{-\veps U}\ud x<\infty$ for any $\veps>0$;
\item for every $0\le\lambda<1$,
\begin{equation*}
    \int\e^{\lambda U}\ud\mux<\infty;
\end{equation*}
\item the gradient moment satisfies
\begin{equation}\label{gm:eq:static-gradient}
    \int\abs{\nabla U}^2\ud\mux\le(dC_U+2)^2<\infty;
\end{equation}
\item for any $0<s<\rho/2$,
\begin{equation*}
    \int\e^{s\abs{x}^2}\ud\mux<\infty.
\end{equation*}
\end{enumerate}
\end{lemma}

\begin{proof}
For (i), for $K:=\{U\le1\}$, Markov's inequality gives
\begin{equation*}
    \abs{K}\le \e\int\e^{-U}\ud x<\infty.
\end{equation*}
Moreover, $K$ is convex and has nonempty interior due to $\inf U = 0 < 1$. Noting that an unbounded convex set with nonempty interior has infinite Lebesgue measure, $K$ is bounded and therefore compact, which  yields a point $x_\star\in K$ such that $U(x_\star)=0$.

Choose $R>0$ such that $K\subset \{|x| \le R\}$.  If $x\notin K$, a segment joining $x_\star$ to $x$ contains a point $y\in K$ such that $U(y)=1$. The convexity of $t\mapsto U(x_\star+t(x-x_\star))$ gives
\begin{equation*}
\frac{U(x)}{\abs{x-x_\star}} \ge \frac{U(y)}{\abs{y-x_\star}} \ge \frac{1}{2R}.
\end{equation*}
Noting $|x_\star| \le R$, it follows that
\begin{equation*}
    U(x)\ge\frac{\abs{x}-R}{2R}  \qquad \text{for every }x\in\RR^d.
\end{equation*}
Thus (i) holds with $\alpha=(2R)^{-1}$ and $\beta=1/2$, and $ \int\e^{-\veps U}\ud x \le \e^{\veps\beta}\int\e^{-\veps\alpha\abs{x}}\ud x
 <\infty$.

 For (ii), if $0\le\lambda<1$, then (i) applied with $\veps=1-\lambda$ yields
\begin{equation*}
    \int\e^{\lambda U}\ud\mux
    =\frac{1}{Z_x}\int\e^{-(1-\lambda)U}\ud x
    <\infty.
\end{equation*}

For (iii), choose $\chi\in C_c^\infty(\RR^d)$ such that $0\le\chi\le1$, $\chi=1$ on $\{|x| \le 1\}$, $\operatorname{supp}\chi\subset \{|x| \le 2\}$, and $\abs{\nabla\chi}\le C_0$. Set $\chi_R(x):=\chi(x/R)$, $\varphi_R:=\chi_R^2$, and
\begin{equation*}
    M_R:=\int\varphi_R\abs{\nabla U}^2\ud\mux =\int\varphi_R\,\Delta U\ud\mux+\int\nabla\varphi_R\cdot\nabla U\ud\mux,
\end{equation*}
where the second equality is from the integration by parts.
By \eqref{gm:eq:tame} and the Cauchy--Schwarz inequality, we estimate
\begin{equation*}
\int\varphi_R\Delta U\ud\mux  \le dC_U\int\varphi_R(1+\abs{\nabla U})\ud\mux \le dC_U\bigl(1+\sqrt{M_R}\bigr).
\end{equation*}
Moreover, $\abs{\nabla\varphi_R}\le(2C_0/R)\chi_R$, and hence $\left|\int\nabla\varphi_R\cdot\nabla U\ud\mux\right| \le\frac{2C_0}{R}\sqrt{M_R}$. Consequently, for $R\ge2C_0$,
\begin{equation*}
    M_R\le dC_U+(dC_U+1)\sqrt{M_R},
\end{equation*}
which implies $M_R\le(dC_U+2)^2$. Since $\varphi_R\to1$ pointwise, Fatou's lemma proves \eqref{gm:eq:static-gradient}.

Finally, (iv) follows from the Herbst argument. Indeed, the LSI \eqref{eq:lsi}, applied to the $1$-Lipschitz function $x\mapsto\abs{x}$, implies the Gaussian concentration $\mux\bigl(\abs{x}\ge m+r\bigr) \le \e^{-\rho r^2/2}$, where $m:=\int\abs{x}\ud\mux<\infty$ by (i). This tail estimate yields, when $s<\rho/2$,
\begin{equation*}
    \int\e^{s\abs{x}^2}\ud\mux
    =1+\int_0^\infty 2sr\e^{sr^2}\,
      \mux\bigl(\abs{x}>r\bigr)\ud r
    <\infty. \qedhere
\end{equation*}
\end{proof}

\subsection{The construction}\label{gm:sec:construction}
We now prove \cref{gm:thm:recovery}. The approximants are defined in the four steps below, which establish their smoothness, positivity, mass normalization, and divergence equation. Their $L^1$ convergence and the limsup bounds \ref{gm:R4}--\ref{gm:R5} are proved in \cref{gm:sec:limsup}. The remaining assertions in \ref{gm:R2}, together with the controlled corrector inequality \ref{gm:R3}, are verified in \cref{gm:sec:verification}.

We fix $(g,W)$ as in \cref{gm:ass:path} and set $\varepsilon\in(0,T/4)$. We let $(\varrho,F)$ be defined as above, satisfying the equation \eqref{gm:eq:lebesgue} with
\begin{equation*}
    \int_0^T \int\frac{\abs{F_t}^2}{\varrho_t}\dz\ud t=T\,J_T(g,W)<\infty\,, \qquad \int_0^T \int\frac{\abs{\nabla_v\varrho_t+v\varrho_t}^2}{\varrho_t}\dz\ud t=T\,I_{v,T}(g) < \infty.
\end{equation*}
The approximation consists of four steps: time mollification, Gaussian mollification in $x$, Gaussian mollification in $v$, and convex combination with equilibrium $\varrho_\infty$.

\subsubsection*{Step 1: time mollification}\label{gm:ssec:time}
Let $\zeta\in C_c^\infty((-1,1))$, $\zeta \ge 0$, $\int \zeta = 1$, and $\zeta_\tau(s)=\tau^{-1}\zeta(s/\tau)$ for $0<\tau<\varepsilon/2$. For $t\in(\varepsilon/2,\,T-\varepsilon/2)$, define
\begin{equation*}
    \varrho^{(\tau)}_t(x,v)=\int_0^T\zeta_\tau(t-r)\varrho_r(x,v)\ud r, \qquad F^{(\tau)}_t(x,v)=\int_0^T\zeta_\tau(t-r)F_r(x,v)\ud r.
\end{equation*}
The pair $(\varrho^{(\tau)},F^{(\tau)})$ satisfies \eqref{gm:eq:lebesgue} in $\mathcal D'((\varepsilon/2,T-\varepsilon/2)\times\RR^{2d})$. Moreover, $\int \varrho^{(\tau)}_t\dz=1$, and $t\mapsto \varrho^{(\tau)}_t$ is  $C^\infty$ from $(\varepsilon/2,T-\varepsilon/2)$ to $L^1(\RR^{2d})$, with $\partial_t\varrho^{(\tau)}_t=\int\zeta_\tau'(s)\varrho_{t-s}\ud s$.

Extend $\varrho$ and $F$ by zero outside $(0,T)$. For \textup{a.e.}\ $(x,v)\in\RR^{2d}$, applying \eqref{gm:eq:kernel-cs} in the time variable to the density-current pairs $(\varrho_t,F_t)$ and $(\varrho_t,\nabla_v\varrho_t + v\varrho_t)$
gives
\begin{equation*}
    \frac{\abs{F^{(\tau)}_t(x,v)}^2}
         {\varrho^{(\tau)}_t(x,v)}
    \le
    \int_0^T\zeta_\tau(t-r)
       \frac{\abs{F_r(x,v)}^2}{\varrho_r(x,v)}\ud r,
\end{equation*}
and
\begin{equation*}
    \frac{\abs{\nabla_v\varrho^{(\tau)}_t(x,v)
      +v\varrho^{(\tau)}_t(x,v)}^2}
         {\varrho^{(\tau)}_t(x,v)}
    \le
    \int_0^T\zeta_\tau(t-r)
       \frac{\abs{\nabla_v\varrho_r(x,v)+v\varrho_r(x,v)}^2}
            {\varrho_r(x,v)}\ud r,
\end{equation*}
respectively. Integrating over $(x,v)$ gives
\small
\begin{equation}\label{gm:eq:time-jensen}
    \int\frac{\abs{F^{(\tau)}_t}^2}{\varrho^{(\tau)}_t}\dz  \le  \Bigl(\zeta_\tau*\int\frac{\abs{F_\cdot}^2}{\varrho_\cdot}\dz\Bigr)(t),
    \quad
    \int\frac{\abs{\nabla_v\varrho^{(\tau)}_t+v\varrho^{(\tau)}_t}^2}{\varrho^{(\tau)}_t}\dz \le \Bigl(\zeta_\tau*\int\frac{\abs{\nabla_v\varrho_\cdot+v\varrho_\cdot}^2}{\varrho_\cdot}\dz\Bigr)(t),
\end{equation}
\normalsize
and for $\tau<\varepsilon/2$,
\begin{equation}\label{gm:eq:time-integrated}
    \int_{\varepsilon}^{T-\varepsilon} \int\frac{\abs{F^{(\tau)}_t}^2}{\varrho^{(\tau)}_t}\dz\ud t
    \le \int_{0}^{T} \int\frac{\abs{F_t}^2}{\varrho_t}\dz\ud t,
    \quad \int_{\varepsilon}^{T-\varepsilon} \int\frac{\abs{\nabla_v\varrho^{(\tau)}_t+v\varrho^{(\tau)}_t}^2}{\varrho^{(\tau)}_t}\dz\ud t
    \le T\,I_{v,T}(g).
\end{equation}

Moreover, time mollification commutes with nonnegative linear moments: for any $\phi\ge0$,
\begin{equation}\label{gm:eq:time-moments}
\begin{aligned}
    \int\phi(x,v)\varrho^{(\tau)}_t(x,v)\ud x\ud v
    &=\int_0^T\zeta_\tau(t-r)
      \left(\int\phi(x,v)\varrho_r(x,v)\ud x\ud v\right)\ud r \\
    &=\Bigl(\zeta_\tau*\textstyle\int
      \phi(x,v)\varrho_\cdot(x,v)\ud x\ud v\Bigr)(t).
\end{aligned}
\end{equation}
It also decreases the relative entropy:
\begin{equation}\label{gm:eq:time-entropy}
    \Ent_\mu(g^{(\tau)}_t) \le
    \bigl(\zeta_\tau*\Ent_\mu(g_\cdot)\bigr)(t),
\end{equation}
by
\begin{equation*}
    g^{(\tau)}_t\log g^{(\tau)}_t \le  \int_0^T\zeta_\tau(t-r)g_r\log g_r\ud r,
\end{equation*}
from the convexity of $s\mapsto s\log s$, where
\begin{equation*}
    g^{(\tau)}_t  =\frac{\varrho^{(\tau)}_t}{\varrho_\infty} = \int_0^T\zeta_\tau(t-r)g_r\ud r.
\end{equation*}
If $t\mapsto\int\phi(x,v)\varrho_t(x,v)\ud x\ud v$ belongs to $L^1(0,T)$, then the right-hand side of \eqref{gm:eq:time-moments} is finite and continuous. Similarly, \ref{gm:P1} gives the same conclusion for the right-hand side of \eqref{gm:eq:time-entropy}. Moreover, for fixed $\tau$, these time averages
are locally uniformly bounded on $(\varepsilon/2,T-\varepsilon/2)$ by standard convolution properties.

\subsubsection*{Step 2: Gaussian mollification in $x$}\label{gm:ssec:xstep}
For $a>0$, define the position convolution:
\begin{equation*}
    \varrho^{(\tau,a)}_t=\eta_a*_x\varrho^{(\tau)}_t,    \qquad  F^{(\tau,a)}_t=\eta_a*_xF^{(\tau)}_t .
\end{equation*}

\begin{proposition}\label{gm:prop:x-step}
The pair $(\varrho^{(\tau,a)},F^{(\tau,a)})$ satisfies, in $\mathcal D'((\varepsilon/2,T-\varepsilon/2)\times\RR^{2d})$,
\begin{equation}\label{gm:eq:x-step-equation}
    \partial_t\varrho^{(\tau,a)}+\divg_x\bigl(v\varrho^{(\tau,a)}\bigr)  +\divg_v\Bigl(-\nabla U\,\varrho^{(\tau,a)}+F^{(\tau,a)}+C^{(\tau,a)}\Bigr)=0,
\end{equation}
where the corrector current is
\begin{equation}\label{gm:eq:force-commutator}
    C^{(\tau,a)}_t(x,v)
    :=\int\eta_a(x-y)\bigl[\nabla U(x)-\nabla U(y)\bigr]\,\varrho^{(\tau)}_t(y,v)\ud y .
\end{equation}
Moreover, for $t\in(\varepsilon/2,T-\varepsilon/2)$,
\begin{equation}\label{gm:eq:force-commutator-action}
    \int\frac{\abs{C^{(\tau,a)}_t}^2}{\varrho^{(\tau,a)}_t}\dz
    \le C(d,C_U)a^2\int\bigl(1+\abs{\nabla U}\bigr)^2\varrho^{(\tau)}_t\dz
\end{equation}
for $0<a\le1$.
\end{proposition}

\begin{proof}
The time-moment estimate \eqref{gm:eq:time-moments}, assumption \ref{gm:Q3}, and \eqref{gm:eq:time-jensen} imply, for $t\in(\varepsilon/2,T-\varepsilon/2)$,
\begin{equation*}
    \int\abs{\nabla U}\,\varrho^{(\tau)}_t\dz \le  \left(\int\bigl(1+\abs{\nabla U}\bigr)^2\varrho^{(\tau)}_t\dz\right)^{1/2}, \qquad
    \int\abs{F^{(\tau)}_t}\dz  \le  \left(\int\frac{\abs{F^{(\tau)}_t}^2}{\varrho^{(\tau)}_t}\dz\right)^{1/2}.
\end{equation*}
Note that convolution in $x$ commutes, in the distributional sense, with $\partial_t$, $\divg_v$, and
$\divg_x(v\,\cdot)$. Applying $x$-convolution to the equation satisfied by $(\varrho^{(\tau)},F^{(\tau)})$, obtained by time-mollifying \eqref{gm:eq:lebesgue}, gives
\begin{equation*}
\partial_t\varrho^{(\tau,a)} +\divg_x\bigl(v\varrho^{(\tau,a)}\bigr) +\divg_v\Bigl(-\eta_a*_x\bigl(\nabla U\,\varrho^{(\tau)}\bigr) + F^{(\tau,a)}\Bigr)=0.
\end{equation*}
By \eqref{gm:eq:force-commutator}, we have
\begin{equation*}
 \eta_a*_x\bigl(\nabla U\,\varrho^{(\tau)}\bigr)(t,x,v)    =\nabla U(x)\,\varrho^{(\tau,a)}_t(x,v)-C^{(\tau,a)}_t(x,v),
\end{equation*}
and hence \eqref{gm:eq:x-step-equation} holds.

For the estimate \eqref{gm:eq:force-commutator-action}, we apply the Cauchy--Schwarz inequality to \eqref{gm:eq:force-commutator} with respect to the measure $\eta_a(x-y)\varrho^{(\tau)}_t(y,v)\ud y$, and obtain
\begin{equation} \label{auxeq3}
    \frac{\abs{C^{(\tau,a)}_t(x,v)}^2}
         {\varrho^{(\tau,a)}_t(x,v)} \le \int\eta_a(x-y)\abs{\nabla U(x)-\nabla U(y)}^2 \varrho^{(\tau)}_t(y,v)\ud y,
\end{equation}
using $\int\eta_a(x-y)\varrho^{(\tau)}_t(y,v)\ud y =\varrho^{(\tau,a)}_t(x,v)$. By \cref{gm:lem:tame-growth}, there holds
\begin{equation*}
    \abs{\nabla U(x)-\nabla U(y)}  \le \bigl(1+\abs{\nabla U(y)}\bigr)  \bigl(\e^{C_U\abs{x-y}}-1\bigr).
\end{equation*}
It follows that
\begin{equation*}
 \int\eta_a(x-y)\abs{\nabla U(x)-\nabla U(y)}^2\ud x
 \le C(d,C_U)a^2\bigl(1+\abs{\nabla U(y)}\bigr)^2.
\end{equation*}
Indeed, using $\e^s-1\le s\e^s$ for $s\ge0$ and changing variables $z=au$, for $0<a\le1$ we have
\begin{equation*}
\int\eta_a(z)\bigl(\e^{C_U\abs z}-1\bigr)^2\ud z
\le C_U^2a^2\int\eta_1(u)\abs u^2\e^{2C_U\abs u}\ud u
\le C(d,C_U)a^2.
\end{equation*}
Integrating \eqref{auxeq3} first in $x$ and then in $(y,v)$ proves \eqref{gm:eq:force-commutator-action}.
\end{proof}

\subsubsection*{Step 3: Gaussian mollification in $v$}\label{gm:ssec:vstep}
For $b>0$, define the velocity convolution: for $t \in (\varepsilon/2,T-\varepsilon/2)$,
\begin{equation*}
    \hat\varrho_t:=\eta_b*_v\varrho^{(\tau,a)}_t, \qquad
    \hat F_t:=\eta_b*_v\bigl(F^{(\tau,a)}_t + C^{(\tau,a)}_t\bigr)\,.
\end{equation*}
Gaussian convolution in $v$ does not commute with $\divg_x(v\,\cdot)$. However, the Gaussian identity turns the resulting commutator \eqref{eq:commuteinv} into the velocity-divergence term $\divg_v(b^2\nabla_x\hat\varrho)$, which can be absorbed into the current.

\begin{proposition} \label{gm:prop:v-step}
For every $b>0$, the corrected density-current pair
$(\hat\varrho,\hat F+b^2\nabla_x\hat\varrho)$ satisfies
\begin{equation}\label{gm:eq:v-step-equation}
    \partial_t\hat\varrho+\divg_x\bigl(v\hat\varrho\bigr)
    +\divg_v\Bigl(-\nabla U\,\hat\varrho+\hat F+b^2\nabla_x\hat\varrho\Bigr) = 0,
\end{equation}
in $\mathcal D'\bigl((\varepsilon/2,T-\varepsilon/2)\times\RR^{2d}\bigr)$.
\end{proposition}

\begin{proof}
Since the convolution in $v$ commutes with $\partial_t$
and $\divg_v$, as well as multiplication by $\nabla U(x)$, convolving \eqref{gm:eq:x-step-equation} in $v$ with the Gaussian kernel gives
\begin{equation} \label{eq:modifieddivergence}
    \partial_t\hat\varrho  +\eta_b*_v\divg_x\bigl(v\varrho^{(\tau,a)}\bigr) +\divg_v\bigl(-\nabla U\,\hat\varrho+\hat F\bigr)=0.
\end{equation}
It remains to identify the second term. Using \eqref{gm:eq:gaussian-identity}, we compute
\begin{align*}
\int\eta_b(v-w)\,w\,\varrho^{(\tau,a)}(t,x,w)\ud w
  & = v\hat\varrho(t,x,v) -\int\eta_b(v-w)(v-w)\varrho^{(\tau,a)}(t,x,w)\ud w\\
& = v\hat\varrho(t,x,v)+b^2\nabla_v\hat\varrho(t,x,v),
\end{align*}
and $\divg_v\bigl(b^2\nabla_x\hat\varrho\bigr) =  \divg_x\bigl(b^2\nabla_v\hat\varrho\bigr)$.
That is,
\begin{equation}\label{eq:commuteinv}
    \eta_b*_v\divg_x\bigl(v \varrho^{(\tau,a)}\bigr) - \divg_x\bigl(v \eta_b*_v\varrho^{(\tau,a)} \bigr) =  \divg_v\bigl(b^2\nabla_x\hat\varrho\bigr).
\end{equation}
Substituting it into \eqref{eq:modifieddivergence} proves \eqref{gm:eq:v-step-equation}.
\end{proof}

\begin{proposition}[Action of the correction term]\label{gm:prop:transport-action}
For every $t\in(\varepsilon/2,T-\varepsilon/2)$,
\begin{equation}\label{gm:eq:transport-action}
    \int\frac{\abs{b^2\nabla_x\hat\varrho_t(x,v)}^2}{\hat\varrho_t(x,v)}\ud x\ud v
    \le
    d\,\frac{b^4}{a^2}\,.
\end{equation}
\end{proposition}

\begin{proof}
It suffices to note
\begin{equation*}
    \hat\varrho_t =\eta_a*_x\bigl(\eta_b*_v\varrho^{(\tau)}_t\bigr),
\end{equation*}
and let $f_{t,v}(x):=(\eta_b*_v\varrho^{(\tau)}_t)(x,v)$. By \cref{gm:lem:gaussian-derivative}, we have
\begin{equation*}
    \int_{\RR^d} \frac{\abs{\nabla_x\hat\varrho_t(x,v)}^2}{\hat\varrho_t(x,v)}\ud x \le
    \frac d{a^2}\int_{\RR^d}f_{t,v}(x)\ud x.
\end{equation*}
Integrating in $v$ then gives \eqref{gm:eq:transport-action}:
\begin{align*}
 \int\frac{\abs{b^2\nabla_x\hat\varrho_t(x,v)}^2}
 {\hat\varrho_t(x,v)}\ud x\ud v
    & \le
    \frac{d\,b^4}{a^2}\int f_{t,v}(x)\ud x\ud v\\
    &=\frac{d\,b^4}{a^2}\int \varrho^{(\tau)}_t(x,v)\ud x\ud v
    =\frac{d\,b^4}{a^2}. \qedhere
\end{align*}
\end{proof}

\subsubsection*{Step 4: Convex combination with equilibrium}\label{gm:ssec:floor}
Let $\{\delta_n\} \subset (0,1)$ and
\begin{equation}\label{gm:eq:final-approximant}
    \varrho^{n}:=(1-\delta_n)\,\hat\varrho+\delta_n\,\varrho_\infty, \qquad F^{n}:=(1-\delta_n)\bigl(\hat F+b^2\nabla_x\hat\varrho\bigr),
\end{equation}
where the superscript $n$ is for the sequence of parameters $(\tau_n,a_n,b_n,\delta_n)\to0$ chosen in \cref{gm:sec:limsup}. Note that the equilibrium density is stationary for the Hamiltonian transport:
\begin{equation*}
    \divg_x\bigl(v\varrho_\infty\bigr) +\divg_v\bigl(-\nabla U\,\varrho_\infty\bigr)  =v\cdot\nabla_x\varrho_\infty -\nabla U\cdot\nabla_v\varrho_\infty  = 0.
\end{equation*}
Thus $(\varrho_\infty,0)$ solves \eqref{gm:eq:lebesgue}. Combining this with \eqref{gm:eq:v-step-equation} shows that $(\varrho^n,F^n)$ satisfies
\begin{equation}\label{gm:eq:final-equation}
    \partial_t\varrho^n+\divg_x\bigl(v\varrho^n\bigr)+\divg_v\bigl(-\nabla U\,\varrho^n+F^n\bigr)=0
    \quad\text{in }\mathcal D'\bigl((\varepsilon/2,T-\varepsilon/2)\times\RR^{2d}\bigr).
\end{equation}
With densities $g^n:=\varrho^n/\varrho_\infty$ and currents $W^n:=F^n/\varrho_\infty$ relative to $\mu$, \eqref{gm:eq:final-equation} is equivalent to
\begin{equation}\label{gm:eq:final-graph-equation}
    (\partial_t+\La)g^n_t=\nabla_v^*W^n_t
    \quad\text{in }\mathcal D'\bigl((\varepsilon/2,T-\varepsilon/2)\times\RR^{2d}\bigr).
\end{equation}
Both $\hat\varrho_t$ and $\varrho_\infty$ have unit mass. Hence, for every $t\in(\varepsilon/2,T-\varepsilon/2)$,
\begin{equation*}
    \int g^n_t\ud\mu  =\int \varrho^n_t(x,v)\ud x\ud v = 1.
\end{equation*}
Moreover, the functions $g^n$ are uniformly bounded below:
\begin{equation}\label{gm:eq:floor}
    g^n =(1-\delta_n)\frac{\hat\varrho}{\varrho_\infty}+\delta_n
    \ge \delta_n \qquad\text{on }(\varepsilon/2,T-\varepsilon/2)\times\RR^{2d}.
\end{equation}

\begin{lemma}[Regularity]\label{gm:lem:smoothness}
On $(\varepsilon/2,T-\varepsilon/2)$, we have
\begin{equation*}
    \varrho^n,g^n\in C^1\bigl((\varepsilon/2,T-\varepsilon/2);C^\infty(\RR^{2d})\bigr),
\end{equation*}
and
\begin{equation*}
    F^n,W^n\in C \bigl((\varepsilon/2,T-\varepsilon/2);C^\infty(\RR^{2d};\RR^d)\bigr).
\end{equation*}
Thus, \eqref{gm:eq:final-equation} and
\eqref{gm:eq:final-graph-equation} hold pointwise.
\end{lemma}

\begin{proof}
By the time mollification, $\varrho^{(\tau)}\in C^1((\varepsilon/2,T-\varepsilon/2);L^1(\RR^{2d}))$. Let $K_{a,b}(x,v):=\eta_a(x)\eta_b(v)$. For every multi-index $\alpha$, Young's inequality gives
\begin{equation*}
    \norm{\partial_{x,v}^\alpha\hat\varrho_t}_{L^\infty(\RR^{2d})}
    \le \norm{\partial_{x,v}^\alpha K_{a,b}}_{L^\infty(\RR^{2d})} \norm{\varrho^{(\tau)}_t}_{L^1(\RR^{2d})},
\end{equation*}
and the same estimate applies with $\varrho^{(\tau)}_t$ replaced by $\partial_t\varrho^{(\tau)}_t$. It follows that $\hat\varrho\in C^1((\varepsilon/2,T-\varepsilon/2);C^\infty(\RR^{2d}))$. Since $\varrho_\infty$ is smooth and strictly positive, we have proved the claimed regularity of $\varrho^n$ and $g^n$.

It remains to show the regularity of $F^n$ and $W^n$. By the Cauchy--Schwarz inequality, the finite action and the unit mass of $\varrho_t$ imply
\begin{equation*}
\int_0^T\!\int\abs{F_t(x,v)}\ud x\ud v\ud t \le
\sqrt{T} \left(\int_0^T\!\int\frac{\abs{F_t(x,v)}^2}{\varrho_t(x,v)} \ud x\ud v\ud t\right)^{1/2} < \infty.
\end{equation*}
Similarly, it follows from \ref{gm:Q3}  that
$(t,x,v)\mapsto\nabla U(x)\varrho_t(x,v)$ belongs to $L^1$. Consequently,
$F^{(\tau)}$ and $\nabla U\,\varrho^{(\tau)}$ are continuous in time with values in $L^1(\RR^{2d})$. Using \eqref{gm:eq:force-commutator}, we write
\begin{equation*}
    \hat F  =\eta_b*_v\eta_a*_xF^{(\tau)}
     +\nabla U\,\hat\varrho -\eta_b*_v\eta_a*_x\bigl(\nabla U\,\varrho^{(\tau)}\bigr).
\end{equation*}
The Gaussian convolution estimates used above, together with the smoothness of $U$ on compact sets, show that $\hat F\in C((\varepsilon/2,T-\varepsilon/2);C^\infty(\RR^{2d};\RR^d))$.
The same holds for $F^n=(1-\delta_n)(\hat F+b^2\nabla_x\hat\varrho)$ and for $W^n$, after division by $\varrho_\infty$. The proof is complete.
\end{proof}

\subsection{Limsup bounds and strong convergence}\label{gm:sec:limsup}
We fix the interior window $I_\varepsilon:=(\varepsilon,T-\varepsilon)$. In this section, under suitable conditions on the approximation scales, we prove the action and Fisher-information limsup bounds \ref{gm:R5} in \cref{gm:prop:limsup}, and the $L^1$ convergence  \ref{gm:R4} in \cref{gm:prop:L1}. We find that when
$\tau_n,a_n,b_n,\delta_n\to0$ and $b_n^2/a_n\to0$, \ref{gm:R4} and \ref{gm:R5} hold simultaneously.

For convenience, we denote, for a density path $h$ and a current $V$ that vanishes \textup{a.e.}\ on $\{h=0\}$,
\begin{equation*}
    \norm{V}_{h}^2 :=\int_{I_\varepsilon}\!\int  \frac{\abs{V_t(x,v)}^2}{h_t(x,v)}\ud x\ud v\ud t.
\end{equation*}

\begin{proposition}[Action and Fisher-information limsup bounds] \label{gm:prop:limsup}
Let $(\varrho^n,F^n)$ be the approximant defined in \eqref{gm:eq:final-approximant}. Assume that the associated  parameters $(\tau_n,a_n,b_n,\delta_n)$ satisfy
\begin{equation*}
    0 < \tau_n < \frac{\varepsilon}{2}, \qquad 0<a_n,b_n\le1, \qquad 0<\delta_n<1.
\end{equation*}
Then, we have
\begin{align}
    \norm{F^n}_{\varrho^n} & \le  \bigl(T\,J_T(g,W)\bigr)^{1/2}
    +C(d,C_U)a_nM_{\nabla U}^{1/2} +\sqrt{dT}\,\frac{b_n^2}{a_n},  \label{gm:eq:action-limsup}\\
    \norm{\nabla_v\varrho^n+v\varrho^n}_{\varrho^n}  &\le  \bigl(T\,I_{v,T}(g)\bigr)^{1/2} +\sqrt{dT}\,b_n, \label{gm:eq:fisher-limsup}
\end{align}
where
\begin{equation*}
    M_{\nabla U} = \int_0^T \int \bigl(1+\abs{\nabla U(x)}\bigr)^2\,g_t\ud\mu\ud t<\infty\,.
\end{equation*}
Consequently, if $a_n,b_n\to0$ and $b_n^2/a_n\to0$, then \ref{gm:R5} holds:
\begin{equation*}
    \limsup_{n\to\infty}J_\varepsilon(g^n,W^n) \le T\,J_T(g,W), \qquad
    \limsup_{n\to\infty}I_{v,\varepsilon}(g^n)  \le T\,I_{v,T}(g).
\end{equation*}
\end{proposition}

\begin{proof}
For convenience, we abbreviate $(\tau_n,a_n,b_n,\delta_n)$ by $(\tau,a,b,\delta)$.

\medskip
\noindent \emph{Action limsup bound.} Recalling the construction \eqref{gm:eq:final-approximant}, we have
\begin{equation*}
    \frac{\abs{F^n}^2}{\varrho^n}  \le(1-\delta) \frac{\abs{\hat F+b^2\nabla_x\hat\varrho}^2}{\hat\varrho},
\end{equation*}
and the triangle inequality gives
\begin{equation}\label{gm:eq:action-triangle}
    \norm{F^n}_{\varrho^n} \le \norm{\hat F}_{\hat\varrho} + b^2\norm{\nabla_x\hat\varrho}_{\hat\varrho}.
\end{equation}
By the convolution estimate \eqref{gm:eq:kernel-jensen} in $v$, there holds
\begin{equation*}
\begin{aligned}
    \norm{\hat F}_{\hat\varrho}
    &=\norm{\eta_b*_v\bigl(F^{(\tau,a)}+C^{(\tau,a)}\bigr)}_{\eta_b*_v\varrho^{(\tau,a)}}\\
    &\le\norm{F^{(\tau,a)}+C^{(\tau,a)}}_{\varrho^{(\tau,a)}} \le\norm{F^{(\tau,a)}}_{\varrho^{(\tau,a)}}
       +\norm{C^{(\tau,a)}}_{\varrho^{(\tau,a)}}.
\end{aligned}
\end{equation*}
Applying the same estimate in $x$ and then using
\eqref{gm:eq:time-integrated} yields
\begin{equation}\label{gm:eq:action-main-current}
    \norm{F^{(\tau,a)}}_{\varrho^{(\tau,a)}} \le \norm{F^{(\tau)}}_{\varrho^{(\tau)}} \le \left(\int_0^T\!\int
    \frac{\abs{F_t(x,v)}^2}{\varrho_t(x,v)}\ud x\ud v\ud t\right)^{1/2} = \bigl(T\,J_T(g,W)\bigr)^{1/2}.
\end{equation}
We now integrate \eqref{gm:eq:force-commutator-action} over $I_\varepsilon$, and then the identity \eqref{gm:eq:time-moments} gives
\begin{equation*}
\begin{aligned}
    \norm{C^{(\tau,a)}}_{\varrho^{(\tau,a)}}^2  &\le C(d,C_U)a^2\int_{I_\varepsilon}\!\int \bigl(1+\abs{\nabla U(x)}\bigr)^2 \varrho^{(\tau)}_t(x,v)\ud x\ud v\ud t\\
    &\le C(d,C_U)a^2M_{\nabla U}.
\end{aligned}
\end{equation*}
Finally, integrating \eqref{gm:eq:transport-action} in time gives
\begin{equation*}
    b^2\norm{\nabla_x\hat\varrho}_{\hat\varrho}
    \le\sqrt{d\abs{I_\varepsilon}}\,\frac{b^2}{a}
    \le\sqrt{dT}\,\frac{b^2}{a}.
\end{equation*}
Substitution into \eqref{gm:eq:action-triangle} proves \eqref{gm:eq:action-limsup}.
Keeping the time variable fixed in the same estimates, and using
\eqref{gm:eq:time-jensen}, \eqref{gm:eq:force-commutator-action}, and
\eqref{gm:eq:transport-action},
also gives, for every $t\in I_{\varepsilon/2}$,
\begin{align}
    \left(\int\frac{\abs{F^n_t}^2}{\varrho^n_t}\dz\right)^{1/2}
    &\le
    \left[\Bigl(\zeta_\tau* \int
        \frac{\abs{F_\cdot}^2}{\varrho_\cdot}\dz\Bigr)(t)\right]^{1/2}
    \notag\\
    &\quad+C(d,C_U)a
    \left[\Bigl(\zeta_\tau*\textstyle\int
        (1+\abs{\nabla U})^2\varrho_\cdot\dz\Bigr)(t)\right]^{1/2}
    +\sqrt d\,\frac{b^2}{a}.
    \label{gm:eq:slice-action-bound}
\end{align}
The right-hand side is locally bounded in $t$ because the two functions
being convolved belong to $L^1(0,T)$ by \ref{gm:P3} and \ref{gm:Q3}.

\medskip
\noindent \emph{Fisher information limsup bound.}
Noting $\nabla_v \varrho_\infty + v \varrho_\infty = 0$, it holds that
\begin{equation*}
    \nabla_v\varrho^n+v\varrho^n=(1-\delta)(\nabla_v\hat\varrho+v\hat\varrho),
    \qquad
    \frac{\abs{\nabla_v\varrho^n+v\varrho^n}^2}{\varrho^n}
    \le(1-\delta)\frac{\abs{\nabla_v\hat\varrho+v\hat\varrho}^2}{\hat\varrho}.
\end{equation*}
Thus
\begin{equation}\label{gm:eq:fisher-floor}
    \norm{\nabla_v\varrho^n+v\varrho^n}_{\varrho^n}\le\norm{\nabla_v\hat\varrho+v\hat\varrho}_{\hat\varrho}.
\end{equation}
The Gaussian identity \eqref{gm:eq:gaussian-identity} gives
\begin{equation*}
    \eta_b*_v\bigl(v\varrho^{(\tau,a)}\bigr)
    =v\hat\varrho+b^2\nabla_v\hat\varrho,
\end{equation*}
and hence
\begin{equation}\label{gm:eq:Gv-decomposition}
    \nabla_v\hat\varrho+v\hat\varrho
    =\eta_b*_v\bigl(\nabla_v\varrho^{(\tau,a)}+v\varrho^{(\tau,a)}\bigr)-b^2\nabla_v\hat\varrho.
\end{equation}
Using the triangle inequality, the convolution estimate \eqref{gm:eq:kernel-jensen} in $v$ and $x$, the commutation of $\nabla_v+v$ with $x$-convolution, and \eqref{gm:eq:time-integrated}, we obtain
\begin{equation*}
\begin{aligned}
    \norm{\eta_b*_v\bigl(\nabla_v\varrho^{(\tau,a)}+v\varrho^{(\tau,a)}\bigr)}_{\hat\varrho}
    &\le\norm{\nabla_v\varrho^{(\tau,a)}+v\varrho^{(\tau,a)}}_{\varrho^{(\tau,a)}}\\
    &\le\norm{\nabla_v\varrho^{(\tau)}+v\varrho^{(\tau)}}_{\varrho^{(\tau)}}
    \le\bigl(T\,I_{v,T}(g)\bigr)^{1/2}.
\end{aligned}
\end{equation*}
Moreover, similarly to \cref{gm:prop:transport-action}, by \cref{gm:lem:gaussian-derivative} in $v$ and integrating over $(t,x)\in I_\varepsilon\times\RR^d$, we have
\begin{equation*}
b^2\norm{\nabla_v\hat\varrho}_{\hat\varrho}    \le b\sqrt{d\abs{I_\varepsilon}} \le b\sqrt{dT}.
\end{equation*}
The corresponding fixed-time estimate of the Fisher information is
\begin{align}
    \left(\int
    \frac{\abs{\nabla_v\varrho^n_t+v\varrho^n_t}^2}{\varrho^n_t}
    \dz\right)^{1/2}
    &\le
    \left[\Bigl(\zeta_\tau* \int
    \frac{\abs{\nabla_v\varrho_\cdot+v\varrho_\cdot}^2}
         {\varrho_\cdot}\dz\Bigr)(t)\right]^{1/2}
    +\sqrt d\,b,
    \label{gm:eq:slice-fisher-bound}
\end{align}
for every $t\in I_{\varepsilon/2}$.
Its right-hand side is locally bounded in $t$ by \ref{gm:P2}.
Combining these estimates with \eqref{gm:eq:fisher-floor} and
\eqref{gm:eq:Gv-decomposition} proves \eqref{gm:eq:fisher-limsup}. The
limsup bounds follow from
$a_n\to0$ and from the identities
\begin{equation*}
J_\varepsilon(g^n,W^n)=\norm{F^n}_{\varrho^n}^2, \qquad
    I_{v,\varepsilon}(g^n)=\norm{\nabla_v\varrho^n+v\varrho^n}_{\varrho^n}^2,
\end{equation*}
by \eqref{gm:eq:functionals-lebesgue}. The proof is complete.
\end{proof}

\begin{proposition}[$L^1$ convergence]\label{gm:prop:L1}
Let $\tau_n,a_n,b_n,\delta_n\to0$, with $0<\tau_n<\varepsilon/2$, $0<a_n,b_n\le1$, and $0<\delta_n<1$. Then
\begin{equation*}
    \varrho^n\to \varrho
    \qquad\text{in }
    L^1\bigl(I_\varepsilon\times\RR^{2d};\ud t\ud x\ud v\bigr),
\end{equation*}
or, equivalently,
\begin{equation*}
    g^n\to g \qquad\text{in }  L^1\bigl(I_\varepsilon\times\RR^{2d};\ud t\ud\mu\bigr).
\end{equation*}
\end{proposition}

\begin{proof}
Denote by $\norm{\cdot}_1$ the $L^1$-norm on
$I_\varepsilon\times\RR^{2d}$ and extend $\varrho$ by zero outside $(0,T)$ in the
time variable. The standard properties of approximations to the identity give
\begin{equation*}
    \norm{\varrho^{(\tau)}-\varrho}_1\to 0,
    \qquad
    \norm{\eta_a*_x\varrho-\varrho}_1\to 0,
    \qquad
    \norm{\eta_b*_v\varrho-\varrho}_1\to 0
\end{equation*}
as $\tau,a,b\to0$. Since convolution with the Gaussian kernel is an $L^1$-contraction, the mollified density
\begin{equation*}
    \hat\varrho^n:=\eta_{b_n}*_v\eta_{a_n}*_x\varrho^{(\tau_n)}
\end{equation*}
satisfies, by the triangle inequality,
\begin{equation*}
\begin{aligned}
    \norm{\hat\varrho^n-\varrho}_1
    &\le\norm{\varrho^{(\tau_n)}-\varrho}_1
      +\norm{\eta_{a_n}*_x\varrho-\varrho}_1
      +\norm{\eta_{b_n}*_v\varrho-\varrho}_1
    \to 0.
\end{aligned}
\end{equation*}
Finally, both $\hat\varrho^n_t$ and $\varrho_\infty$ have unit mass, so
\begin{equation*}
    \norm{\varrho^n-\hat\varrho^n}_1
    =\delta_n\norm{\varrho_\infty-\hat\varrho^n}_1
    \le2\delta_n\abs{I_\varepsilon}
    \le2\delta_nT
    \to 0.
\end{equation*}
Thus, $\varrho^n\to \varrho$ in $L^1$. The equivalent statement for $g^n$ follows from
$\varrho^n=g^n\varrho_\infty$ and $\varrho=g\varrho_\infty$.
\end{proof}

\subsection{Verification of the corrector inequality}\label{gm:sec:verification}
We again fix $n$ and abbreviate $(\tau_n,a_n,b_n,\delta_n)$ by
$(\tau,a,b,\delta)$, where $0<\tau<\varepsilon/2$, $0<a,b\le1$, and
$0<\delta<1$. We retain the superscript $n$ on all quantities associated
with the approximant and write
$g^n=\varrho^n/\varrho_\infty$ and $W^n=F^n/\varrho_\infty$. Unless stated
otherwise, the estimates below hold on the larger window
$I_{\varepsilon/2}$ and are locally uniform in time there; the argument for
the corrector inequality \ref{gm:R3} is restricted to $I_\varepsilon$.

We first express the marginal quantities associated with $\varrho^n$ in
Lebesgue variables. For each fixed time slice, define
\begin{equation}\label{gm:eq:lebesgue-marginals}
    \check\varrho^n(x):=\int \varrho^n(x,v)\ud v,
    \qquad
    \check\jmath^n(x):=\int v\,\varrho^n(x,v)\ud v,
    \qquad
    S_L^n(x):=\int v\otimes v\,\varrho^n(x,v)\ud v.
\end{equation}
Let $\varrho_{\infty,x}:=Z_x^{-1}\e^{-U}$ be the Lebesgue density of $\mux$. The corresponding relative marginals and moments are
\begin{equation}\label{gm:eq:relative-marginals}
    q^n=\frac{\check\varrho^n}{\varrho_{\infty,x}},\qquad
    j^n=\frac{\check\jmath^n}{\varrho_{\infty,x}},\qquad
    \Piv\bigl(v\otimes v\,g^n\bigr)=\frac{S_L^n}{\varrho_{\infty,x}},\qquad
    m^n=\frac{j^n}{q^n}=\frac{\check\jmath^n}{\check\varrho^n}.
\end{equation}
The centered stress tensor is $\Theta^n=\Theta_L^n/\varrho_{\infty,x}$, where
\begin{equation}\label{gm:eq:stress-lebesgue}
    \Theta_L^n:=S_L^n-\frac{\check\jmath^n\otimes\check\jmath^n}{\check\varrho^n}
    -\check\varrho^n I_d
    =\int\bigl[(v-m^n)\otimes(v-m^n)-I_d\bigr]\varrho^n\ud v .
\end{equation}
We will repeatedly use the identity
\begin{equation}\label{gm:eq:adjoint-divergence}
    \nabla_x^{*}F
    =-\divg F+\nabla U\cdot F
    =-\frac{\divg(F\varrho_{\infty,x})}{\varrho_{\infty,x}}.
\end{equation}
In particular, $\nabla_x^{*}\Theta^n=-\divg(\Theta_L^n)/\varrho_{\infty,x}$, and
\begin{equation}\label{gm:eq:weighted-change}
    \int\frac{\abs{\nabla_x^{*}\Theta^n}^2}{q^n}\ud\mux
    =\int\frac{\abs{\divg\Theta_L^n}^2}{\check\varrho^n}\ud x,
    \qquad
    \int\frac{\abs{\nabla_xq^n}^2}{q^n}\ud\mux
    =\int\frac{\abs{\nabla\check\varrho^n+\check\varrho^n\nabla U}^2}{\check\varrho^n}\ud x .
\end{equation}
Finally, we shall repeatedly use the fact that whenever a slice quantity $\psi(t)$ is bounded by $(\zeta_\tau*\phi)(t)$ with some $\phi\in L^1(0,T)$, that bound is continuous in $t$ and therefore bounded on compact subsets of $I_{\varepsilon/2}$.

\subsubsection{Moment bounds}

We start with some useful consequences of the assumed moment bounds in \ref{gm:Q1}--\ref{gm:Q3} on the approximation sequences constructed here. For the original path \eqref{gm:eq:dictionary}, define its zeroth and first velocity moments in Lebesgue variables by
\begin{equation*}
    \bar\varrho_t(x):=\int \varrho_t(x,v)\ud v,
    \qquad
    \bar\jmath_t(x):=\int v\,\varrho_t(x,v)\ud v.
\end{equation*}

\begin{lemma}[Corollary of \ref{gm:Q1}]\label{gm:lem:fiber-transfer}
Set
\begin{equation*}
    \widetilde I
    :=\left(\frac{\varepsilon}{2}-\tau,
    T-\frac{\varepsilon}{2}+\tau\right)\subset(0,T),
\end{equation*}
and let $M_{\widetilde I}$ be the corresponding constant in \ref{gm:Q1}.
Then, for every $(t,x)\in I_{\varepsilon/2}\times\RR^d$,
\begin{equation}\label{gm:eq:fiber-transfer}
    \abs{\check\jmath^n_t(x)}\le M_{\widetilde I}\,\check\varrho^n_t(x),
    \qquad\text{equivalently}\qquad
    \abs{m^n_t(x)}=\frac{\abs{j^n_t(x)}}{q^n_t(x)}\le M_{\widetilde I}.
\end{equation}
\end{lemma}

\begin{proof}
Since $q_t=\bar\varrho_t/\varrho_{\infty,x}$ and $j_t=\bar\jmath_t/\varrho_{\infty,x}$, the common factor $\varrho_{\infty,x}$ cancels in \ref{gm:Q1}. Hence
\begin{equation}\label{gm:eq:original-fiber-bound}
    \abs{\bar\jmath_t(x)}
    \le M_{\widetilde I}\,\bar\varrho_t(x)
\end{equation}
for \textup{a.e.}\ $(t,x)\in\widetilde I\times\RR^d$.

\medskip
\noindent \emph{Step 1. Time mollification and $x$-convolution.}
Let $\bar\varrho^{(\tau,a)}$ and $\bar\jmath^{(\tau,a)}$ denote the zeroth and first velocity moments after these two steps:
\begin{equation*}
    \bar\varrho^{(\tau,a)}
    =\zeta_\tau*_t\eta_a*_x\bar\varrho,
    \qquad
    \bar\jmath^{(\tau,a)}
    =\zeta_\tau*_t\eta_a*_x\bar\jmath.
\end{equation*}
Since the convolution kernel is nonnegative, and the $\tau$-neighborhood of
$I_{\varepsilon/2}$ is contained in $\widetilde I$, \eqref{gm:eq:original-fiber-bound} gives, for every
$(t,x)\in I_{\varepsilon/2}\times\RR^d$,
\begin{equation*}
    \abs{\bar\jmath^{(\tau,a)}}
    \le\zeta_\tau*_t\eta_a*_x\abs{\bar\jmath}
    \le M_{\widetilde I}\,\bar\varrho^{(\tau,a)}.
\end{equation*}

\medskip
\noindent \emph{Step 2. Velocity mollification.}
For any nonnegative integrable $f$, the centered Gaussian convolution
preserves the zeroth and first moments:
\begin{equation*}
    \int(\eta_b*_vf)(v)\ud v=\int f(w)\ud w,
    \qquad
    \int v\,(\eta_b*_vf)(v)\ud v=\int w\,f(w)\ud w
\end{equation*}
because
\begin{equation*}
    \int\eta_b(v-w)\ud v=1,
    \qquad
    \int v\,\eta_b(v-w)\ud v=w.
\end{equation*}
Consequently, if $\hat\varrho=\eta_b*_v\varrho^{(\tau,a)}$ and we set
\begin{equation*}
    \widehat{\bar\varrho}(t,x):=\int\hat\varrho_t(x,v)\ud v,
    \qquad
    \widehat{\bar\jmath}(t,x):=\int v\,\hat\varrho_t(x,v)\ud v,
\end{equation*}
then $\widehat{\bar\varrho}=\bar\varrho^{(\tau,a)}$ and $\widehat{\bar\jmath}=\bar\jmath^{(\tau,a)}$, and hence
\begin{equation}\label{gm:eq:hat-fiber-bound}
    \abs{\widehat{\bar\jmath}}
    \le M_{\widetilde I}\,\widehat{\bar\varrho}.
\end{equation}

\medskip
\noindent \emph{Step 3. Equilibrium floor.}
Since the equilibrium density has spatial marginal
$\varrho_{\infty,x}$ and vanishing first velocity moment, the final
convex combination gives
\begin{equation*}
    \check\varrho^n
    =(1-\delta)\widehat{\bar\varrho}+\delta\varrho_{\infty,x},
    \qquad
    \check\jmath^n=(1-\delta)\widehat{\bar\jmath}.
\end{equation*}
Combining this with \eqref{gm:eq:hat-fiber-bound}, we obtain
\begin{equation*}
    \abs{\check\jmath^n}
    \le(1-\delta)M_{\widetilde I}\widehat{\bar\varrho}
    \le M_{\widetilde I}\check\varrho^n.
\end{equation*}
Finally, \eqref{gm:eq:relative-marginals} gives
$m^n=\check\jmath^n/\check\varrho^n=j^n/q^n$, proving the equivalent conditional-mean bound in \eqref{gm:eq:fiber-transfer}.
\end{proof}

\begin{lemma}[Integrated fourth moments]\label{gm:lem:fourth-transfer}
Under \ref{gm:Q2}, for every compact $K\subset I_{\varepsilon/2}$,
there exists $C_K<\infty$ such that
\begin{equation}\label{gm:eq:fourth-transfer}
    \sup_{t\in K}\int(1+\abs v^4)\,\varrho^n_t(x,v)\ud x\ud v
    \le C_K.
\end{equation}
\end{lemma}

\begin{proof}
By \ref{gm:Q2} and \eqref{gm:eq:time-moments},
\begin{equation*}
    t\longmapsto
    \int(1+\abs v^4)\varrho^{(\tau)}_t(x,v)\ud x\ud v
\end{equation*}
is bounded by the $\zeta_\tau$-average of an $L^1(0,T)$ function and is therefore locally bounded. The $x$-convolution preserves velocity moments, while for the $v$-convolution, the inequality $\abs{w+z}^4\le8(\abs w^4+\abs z^4)$ gives
\begin{equation*}
    \int(1+\abs v^4)\hat\varrho_t(x,v)\ud x\ud v
    \le8\int(1+\abs w^4)\varrho^{(\tau)}_t(x,w)\ud x\ud w
    +8c_db^4,
\end{equation*}
which is also locally bounded, where $c_d:=\int\abs z^4\eta_1(z)\ud z$. Finally, the convex combination with equilibrium adds the finite Gaussian moment $\delta\int(1+\abs v^4)\varrho_\infty\ud x\ud v$, proving the claim.
\end{proof}

\begin{lemma}[Stress energy bound]\label{gm:lem:stress-L2}
For every $t\in I_{\varepsilon/2}$,
\begin{equation}\label{gm:eq:stress-L2}
    \int\frac{\abs{\Theta^n_t}^2}{q^n_t}\ud\mux
    =\int\frac{\abs{\Theta^n_{L,t}}^2}{\check\varrho^n_t}\ud x
    \le C_d\int(1+\abs v^4)\varrho^n_t(x,v)\ud x\ud v.
\end{equation}
In particular, the stress energy is locally bounded in time.
\end{lemma}

\begin{proof}
The identity follows directly from
$\Theta^n=\Theta_L^n/\varrho_{\infty,x}$,
$q^n=\check\varrho^n/\varrho_{\infty,x}$, and
$\ud\mux=\varrho_{\infty,x}\ud x$. Let
\begin{equation*}
    \Sigma^n(x):=\frac{1}{\check\varrho^n(x)}
    \int(v-m^n(x))\otimes(v-m^n(x))\varrho^n(x,v)\ud v
\end{equation*}
be the conditional covariance, so that
$\Theta_L^n=\check\varrho^n(\Sigma^n-I_d)$. Then
\begin{equation*}
    \frac{\abs{\Theta_L^n}^2}{\check\varrho^n}
    \le2\check\varrho^n\,\abs{\Sigma^n}^2+2d\check\varrho^n.
\end{equation*}
Jensen's inequality gives
\begin{equation*}
    \check\varrho^n\,\abs{\Sigma^n}^2
    \le\int\abs{v-m^n}^4\varrho^n\ud v.
\end{equation*}
Moreover, $\abs{v-m^n}^4\le8(\abs v^4+\abs{m^n}^4)$, while by Jensen's inequality again, we have
$\abs{m^n}^4\check\varrho^n\le\int\abs v^4\varrho^n\ud v$. Integration in $x$ proves
\eqref{gm:eq:stress-L2}, for instance with $C_d=32+2d$.
\end{proof}

\begin{lemma}[Slice moments]\label{gm:lem:slice-moments}
For every compact $K\subset I_{\varepsilon/2}$,
\begin{equation*}
    \sup_{t\in K}\int
    \bigl(1+\abs x^2+U(x)+\abs{\nabla U(x)}^2\bigr)
    \varrho^n_t(x,v)\ud x\ud v<\infty.
\end{equation*}
\end{lemma}

\begin{proof}
Choose $0<s_0<\rho/2$. The exponential-integrability bounds in
\cref{gm:lem:exp-integrability}(ii),(iv) allow us to apply Gibbs variational inequality \eqref{eq:variational} separately with
\begin{equation*}
    \phi=s_0\abs{x}^2
    \qquad\text{and}\qquad
    \phi=\frac12U,
\end{equation*}
and with $g=g_t$ in both cases.
This gives
\begin{equation*}
    \int\bigl(\abs x^2+U(x)\bigr)g_t\ud\mu
    \le C\bigl(\Ent_\mu(g_t)+1\bigr)
\end{equation*}
for \textup{a.e.}\ $t$. The right-hand side is integrable in time by \ref{gm:P1},
while \ref{gm:Q3} gives the time integrability of
$\int(1+\abs{\nabla U})^2g_t\ud\mu$. Hence time mollification turns all
these moments into locally bounded $\zeta_\tau$-averages, by
\eqref{gm:eq:time-moments}. The $v$-convolution leaves spatial moments
unchanged.
For the $x$-convolution, \cref{gm:lem:tame-growth} gives
\begin{align*}
    \int\abs x^2\,\eta_a(x-y)\ud x&\le2\abs y^2+2da^2,\\
    \int U(x)\,\eta_a(x-y)\ud x&\le U(y)+\bigl(1+\abs{\nabla U(y)}\bigr)\int\abs z\e^{C_U\abs z}\eta_a(z)\ud z,\\
    \int\abs{\nabla U(x)}^2\eta_a(x-y)\ud x&\le\bigl(1+\abs{\nabla U(y)}\bigr)^2\int\e^{2C_U\abs z}\eta_a(z)\ud z,
\end{align*}
and all Gaussian factors on the right are finite for fixed $a\le1$.
Finally, the corresponding equilibrium moments from the equilibrium floor are finite by \cref{gm:lem:exp-integrability}(ii)--(iv). The proof is complete.
\end{proof}

\subsubsection{Weighted derivative bounds}

\begin{lemma}[Weighted spatial Fisher bounds]\label{gm:lem:weighted-fisher}
For every compact $K\subset I_{\varepsilon/2}$, there exists
$C'_K<\infty$ such that, for every $t\in K$,
\begin{align}
    \int\frac{\abs{\nabla_x\check\varrho^n_t}^2}{\check\varrho^n_t}\ud x
    &\le(1-\delta)\frac d{a^2}
    +\delta\int\abs{\nabla U}^2\ud\mux<\infty,
    \label{gm:eq:marg-fisher}\\
    \int(1+\abs v^4)
    \frac{\abs{\nabla_x\varrho^n_t(x,v)}^2}{\varrho^n_t(x,v)}\ud x\ud v
    &\le(1-\delta)\frac d{a^2}\,C'_K
    +\delta\,c_{4,d}\int\abs{\nabla U}^2\ud\mux<\infty,
    \label{gm:eq:weighted-fisher}
\end{align}
where $c_{4,d}:=\int(1+\abs v^4)\kv(\ud v)$.
\end{lemma}

\begin{proof}
The joint convexity of the map $(f,F)\mapsto\abs F^2/f$ and the decomposition
$\varrho^n=(1-\delta)\hat\varrho+\delta\varrho_\infty$ give, pointwise,
\begin{equation*}
    \frac{\abs{\nabla_x\varrho^n}^2}{\varrho^n}
    \le(1-\delta)\frac{\abs{\nabla_x\hat\varrho}^2}{\hat\varrho}
    +\delta\frac{\abs{\nabla_x\varrho_\infty}^2}{\varrho_\infty}
    =(1-\delta)\frac{\abs{\nabla_x\hat\varrho}^2}{\hat\varrho}
    +\delta\,\abs{\nabla U}^2\varrho_\infty.
\end{equation*}
Set $\tilde\varrho:=\eta_b*_v\varrho^{(\tau)}$, so that
$\hat\varrho=\eta_a*_x\tilde\varrho$. By \cref{gm:lem:gaussian-derivative},
\begin{equation*}
    \frac{\abs{\nabla_x\hat\varrho_t(x,v)}^2}{\hat\varrho_t(x,v)}
    \le\frac1{a^4}\int\abs{x-y}^2\eta_a(x-y)
    \tilde\varrho_t(y,v)\ud y.
\end{equation*}
Multiplying by $1+\abs v^4$ and integrating in $(x,v)$ gives
\begin{equation*}
    \int(1+\abs v^4)\frac{\abs{\nabla_x\hat\varrho_t}^2}{\hat\varrho_t}  \ud x\ud v \le\frac d{a^2}\int(1+\abs v^4)\tilde\varrho_t\ud x\ud v.
\end{equation*}
The $v$-convolution estimate from the proof of \cref{gm:lem:fourth-transfer} gives
\begin{equation*}
    \sup_{t\in K}\int(1+\abs v^4)\tilde\varrho_t\ud x\ud v
    \le C'_K<\infty.
\end{equation*}
The equilibrium floor contribution is
\begin{equation*}
    \int(1+\abs v^4)\abs{\nabla U(x)}^2\varrho_\infty(x,v)
    \ud x\ud v
    =c_{4,d}\int\abs{\nabla U}^2\ud\mux<\infty
\end{equation*}
by \cref{gm:lem:exp-integrability}(iii), proving \eqref{gm:eq:weighted-fisher}. Applying the same argument to
\begin{equation*}
    \check\varrho^n=(1-\delta)\eta_a*_x\bar \varrho^{(\tau)}
    +\delta\varrho_{\infty,x},
    \qquad
    \bar \varrho^{(\tau)}(x):=\int \varrho^{(\tau)}(x,v)\ud v,
\end{equation*}
and using $\abs{\nabla\varrho_{\infty,x}}^2/\varrho_{\infty,x} =\abs{\nabla U}^2\varrho_{\infty,x}$
proves \eqref{gm:eq:marg-fisher}.
\end{proof}

\subsubsection{Marginal floor, entropy, and current bounds}

\begin{proposition}[Marginal bounds]\label{gm:prop:A1}
For every $t\in I_{\varepsilon/2}$, $q^n_t$ is a probability density with
respect to $\mux$ and satisfies
\begin{equation*}
    q^n_t\ge\delta,
    \qquad
    \int\frac{\abs{j^n_t}^2}{q^n_t}\ud\mux\le M_{\widetilde I}^2.
\end{equation*}
Moreover, $\Ent_{\mux}(q^n_t)<\infty$, locally uniformly in $t$.
\end{proposition}

\begin{proof}
We suppress the time index $t$ throughout
proof for simplicity. First, by  \cref{gm:lem:fiber-transfer}, we have
\begin{equation*}
    \int\frac{\abs{j^n}^2}{q^n}\ud\mux
    =\int\abs{m^n}^2q^n\ud\mux
    \le M_{\widetilde I}^2.
\end{equation*}
For the marginal entropy boundedness, by  $q^n=\check\varrho^n/\varrho_{\infty,x}$ and $\ud\mux=\varrho_{\infty,x}\ud x$, we compute
\begin{equation} \label{auxeq_ent}
    \Ent_{\mux}(q^n) = \int\check\varrho^n\log\check\varrho^n\ud x + \int U\check\varrho^n\ud x+\log Z_x.
\end{equation}
The potential term $\int U\check\varrho^n\ud x$ is locally bounded in time by \cref{gm:lem:slice-moments}. It remains to show that
$\check\varrho^n\log\check\varrho^n\in L^1(\ud x)$, locally uniformly in
time. We first control its positive part. Setting $\bar\varrho^{(\tau)}(x)
 :=\int\varrho^{(\tau)}(x,v)\ud v$, the  $v$ convolution preserves the spatial marginal, so the construction gives
\begin{equation*}
    \check\varrho^n
    =(1-\delta)\eta_a*_x\bar\varrho^{(\tau)}
    +\delta\varrho_{\infty,x}.
\end{equation*}
Young's convolution inequality yields
\begin{equation*}
\begin{aligned}
    \norm{\check\varrho^n}_{L^\infty}
    &\le(1-\delta)
      \norm{\eta_a}_{L^\infty}
      \norm{\bar\varrho^{(\tau)}}_{L^1}      +\delta\norm{\varrho_{\infty,x}}_{L^\infty}\\
    &=(1-\delta)(2\pi a^2)^{-d/2}    +\delta\norm{\varrho_{\infty,x}}_{L^\infty}<\infty,
\end{aligned}
\end{equation*}
where $\norm{\varrho_{\infty,x}}_{L^\infty}\le Z_x^{-1}$. It follows that
\begin{equation*}
\int\check\varrho^n(\log\check\varrho^n)_+\ud x \le\log_+\norm{\check\varrho^n}_{L^\infty}<\infty.
\end{equation*}
For the negative part, decompose $\{\check\varrho^n<1\}$ as
\begin{equation*}
    \{\check\varrho^n<\e^{-1-\abs x^2}\}
    \quad\text{and}\quad
    \{\e^{-1-\abs x^2}\le\check\varrho^n<1\}.
\end{equation*}
On the first set, monotonicity of $s\mapsto s\abs{\log s}$ on
$(0,\e^{-1})$ gives
\begin{equation*}
    \check\varrho^n\abs{\log\check\varrho^n}
    \le\e^{-1-\abs x^2}(1+\abs x^2),
\end{equation*}
which is integrable. On the second set,
$\abs{\log\check\varrho^n}\le1+\abs x^2$, so
\begin{equation*}
\int_{\{\e^{-1-\abs x^2}\le\check\varrho^n<1\}}\check\varrho^n(\log\check\varrho^n)_-\ud x \le \int_{\{\e^{-1-\abs x^2}\le\check\varrho^n<1\}}(1+\abs x^2)\check\varrho^n\ud x,
\end{equation*}
which is locally bounded in time by \cref{gm:lem:slice-moments}. Thus
$\int\check\varrho^n\log\check\varrho^n\ud x$ is finite and locally
bounded in time. The identity \eqref{auxeq_ent} now proves the claim.
\end{proof}

\subsubsection{Moment equations}

\begin{lemma}[Zeroth and first moment equations]\label{gm:lem:moment-equations}
On $I_{\varepsilon/2}\times\RR^d$, the following identities hold pointwise:
\begin{equation}\label{gm:eq:zeroth-moment}  \partial_t\check\varrho^n+\divg\check\jmath^n=0,
    \qquad\text{equivalently}\qquad
    \partial_tq^n_t-\nabla_x^{*}j^n_t=0,
\end{equation}
and
\begin{equation}\label{gm:eq:first-moment}
    \partial_t\check\jmath^n
    =-\divg S_L^n-\nabla U\,\check\varrho^n+\check F^n,
    \qquad
    \check F^n_t(x):=\int F^n_t(x,v)\ud v.
\end{equation}
\end{lemma}

\begin{proof}
Let $\chi\in C_c^\infty(\RR^d)$ satisfy $0\le\chi\le1$ and $\chi=1$ on the unit ball, and set $\chi_R(v):=\chi(v/R)$. Define the truncated moments:
\begin{equation*}
    \check\varrho_R^n:=\int\chi_R\varrho^n\ud v, \qquad
    \check\jmath_R^n:=\int  v\chi_R\varrho^n\ud v,
    \qquad  S_{L,R}^n:=\int v\otimes v\chi_R\varrho^n\ud v.
\end{equation*}
By the fourth moment bound in \cref{gm:lem:fourth-transfer}, for every compact $K\subset I_{\varepsilon/2}$, we have
\begin{equation*}
    \sup_{t\in K}\int(1+\abs v^2)\varrho_t^n\ud x\ud v<\infty,
\end{equation*}
and hence dominated convergence gives
\begin{equation}\label{gm:eq:truncated-moment-convergence}
    \check\varrho_R^n\to\check\varrho^n,
    \qquad
    \check\jmath_R^n\to\check\jmath^n,
    \qquad
    S_{L,R}^n\to S_L^n
\end{equation}
in $L^1(K\times\RR^d)$ as $R\to\infty$.

Set
\begin{equation*}
    F^{n,\mathrm{tot}}:=-\nabla U\,\varrho^n+F^n,
\end{equation*}
for the total velocity flux. We first establish its integrability needed for the velocity cutoff limit below. By Cauchy--Schwarz inequality, \cref{gm:lem:fourth-transfer,gm:lem:slice-moments} give
\begin{equation*}
\begin{aligned}
    \int_K\!\int(1+\abs v)\abs{\nabla U(x)}\varrho^n
    \ud x\ud v\ud t \le
    \left(\int_K\!\int(1+\abs v)^2\varrho^n
    \ud x\ud v\ud t\right)^{1/2}
    \left(\int_K\!\int\abs{\nabla U(x)}^2\varrho^n
    \ud x\ud v\ud t\right)^{1/2}<\infty.
\end{aligned}
\end{equation*}
Moreover, \eqref{gm:eq:slice-action-bound} implies
\begin{equation*}
    \sup_{t\in K}\int\frac{\abs{F^n_t}^2}{\varrho^n_t}\ud x\ud v<\infty.
\end{equation*}
Together with \cref{gm:lem:fourth-transfer}, it follows that
\begin{equation*}
\begin{aligned}
    \int_K\!\!\int(1+\abs v)\abs{F^n}\ud x\ud v\ud t
    &\le
    \left(\int_K\!\!\int(1+\abs v)^2\varrho^n\ud x\ud v\ud t\right)^{1/2}
    \left(\int_K\!\!\int\frac{\abs{F^n}^2}{\varrho^n}
    \ud x\ud v\ud t\right)^{1/2}<\infty.
\end{aligned}
\end{equation*}
Therefore
\begin{equation} \label{auxeq_intef}
    (1+\abs v)F^{n,\mathrm{tot}}\in
    L^1\bigl(K\times\RR^{2d}\bigr)
    \qquad\text{for every }K\subset I_{\varepsilon/2}.
\end{equation}

\medskip
\noindent \emph{Zeroth moment.}
Integrating \eqref{gm:eq:final-equation} against $\chi_R$ in the velocity
variable gives, in $\mathcal D'(I_{\varepsilon/2}\times\RR^d)$,
\begin{equation*}
    \partial_t\check\varrho_R^n+\divg\check\jmath_R^n
    =\frac1R\int (\nabla\chi)(v/R)\cdot F^{n,\mathrm{tot}}\ud v.
\end{equation*}
The $L^1(K\times\RR^d)$-norm of its right-hand side is bounded by
$R^{-1}\norm{\nabla\chi}_{L^\infty}\int_K\!\int\abs{F^{n,\mathrm{tot}}} \ud v \ud t$ and hence tends to zero as $R \to \infty$. Passing to the limit with \eqref{gm:eq:truncated-moment-convergence} yields
\begin{equation} \label{auxeq:zero}
    \partial_t\check\varrho^n+\divg\check\jmath^n=0.
\end{equation}

\medskip
\noindent \emph{First moment.}
Multiplying \eqref{gm:eq:final-equation} by $v_i\chi_R$ and integrating
in $v$, we use
\begin{equation*}
    \int v_i\chi_R\divg_vF^{n,\mathrm{tot}}\ud v
    =-\int\chi_R(F^{n,\mathrm{tot}})_i\ud v
      -\frac1R\int v_i(\nabla\chi)(v/R)\cdot F^{n,\mathrm{tot}}\ud v.
\end{equation*}
Thus, in vector form,
\begin{equation*}
    \partial_t\check\jmath_R^n+\divg S_{L,R}^n
    =\int\chi_RF^{n,\mathrm{tot}}\ud v
      +\frac1R\int v(\nabla\chi)(v/R)\cdot F^{n,\mathrm{tot}}\ud v.
\end{equation*}
The main term $\int\chi_RF^{n,\mathrm{tot}}\ud v$ converges to
$\int F^{n,\mathrm{tot}}\ud v$ in $L^1(K\times\RR^d)$. The
$L^1(K\times\RR^d)$-norm of the second cutoff error term is bounded by
\begin{equation*}
    \frac{\norm{\nabla\chi}_{L^\infty}}{R}
    \int_K\!\int\abs v\,\abs{F^{n,\mathrm{tot}}}\ud x\ud v\ud t,
\end{equation*}
which tends to zero by the weighted integrability \eqref{auxeq_intef} of
$F^{n,\mathrm{tot}}$. Therefore, we have \eqref{gm:eq:first-moment}:
\begin{equation*}
    \partial_t\check\jmath^n+\divg S_L^n
    =\int F^{n,\mathrm{tot}}\ud v
    =-\nabla U\,\check\varrho^n+\check F^n.
\end{equation*}
Finally, by \eqref{gm:eq:relative-marginals} and \eqref{gm:eq:adjoint-divergence}, as well as \eqref{auxeq:zero}, we have
\begin{equation*}
    \partial_t\check\varrho^n+\divg\check\jmath^n
    =\varrho_{\infty,x}
    \bigl(\partial_tq_t^n-\nabla_x^*j_t^n\bigr),
\end{equation*}
which proves the equivalent relative form in \eqref{gm:eq:zeroth-moment}.
\end{proof}

\subsubsection{Velocity and acceleration bounds}

We now verify the assumptions for \cref{lem:wasserstein-acceleration}.

\begin{proposition}[Velocity and acceleration bounds]\label{gm:prop:A2} The pair $(q^n_t,j^n_t)_{t\in I_{\varepsilon/2}}$ is a regular density-current pair in the sense of \cref{def:regular-pair}. In particular, $m^n_t=j^n_t/q^n_t$ is smooth and satisfies
\begin{equation*}
    \sup_{t\in K}\norm{m^n_t}_{L^\infty(\RR^d)}\le M_{\widetilde I},
    \qquad
    \sup_{t\in K}\int\abs{a^n_t}^2q^n_t\ud\mux<\infty
\end{equation*}
for every compact $K\subset I_{\varepsilon/2}$, where
\begin{equation*}
    a^n_t:=\partial_tm^n_t+(m^n_t\cdot\nabla_x)m^n_t.
\end{equation*}
\end{proposition}

\begin{proof}
By \cref{gm:lem:smoothness,gm:prop:A1}, the pair $(q^n,j^n)$ has the
regularity, positivity, entropy, and current-energy properties required in
\textup{(A1)} of \cref{def:regular-pair}. The pointwise continuity equation
in \textup{(A2)} follows from \cref{gm:lem:moment-equations}, while
\cref{gm:lem:fiber-transfer} gives
\begin{equation*}
    \sup_{t\in K}\norm{m^n_t}_{L^\infty(\RR^d)}
    \le M_{\widetilde I}<\infty
    \qquad\text{for every }K\subset I_{\varepsilon/2}.
\end{equation*}
It remains only to verify the acceleration bound in \textup{(A2)}.

Since $\check\varrho^n\ge\delta\varrho_{\infty,x}>0$, the field
$m^n=\check\jmath^n/\check\varrho^n$ is smooth. Using
$m^n\check\varrho^n=\check\jmath^n$ and \eqref{gm:eq:zeroth-moment}, we obtain
\begin{equation*}
\begin{aligned}
    \check\varrho^n\,a^n
    &=\partial_t\check\jmath^n-m^n\,\partial_t\check\varrho^n
      +(\check\jmath^n\cdot\nabla_x)m^n\\
    &=\partial_t\check\jmath^n+m^n\,\divg\check\jmath^n +(\check\jmath^n\cdot\nabla_x)m^n   =\partial_t\check\jmath^n+\divg(m^n\otimes\check\jmath^n).
\end{aligned}
\end{equation*}
Substituting \eqref{gm:eq:first-moment} and $m^n\otimes\check\jmath^n
 = \check\jmath^n\otimes\check\jmath^n/\check\varrho^n$ gives
\begin{equation}\label{gm:eq:acceleration-identity}
\begin{aligned}
    \check\varrho^n\,a^n
    &=-\divg\left(S_L^n-
      \frac{\check\jmath^n\otimes\check\jmath^n}{\check\varrho^n}\right)
      -\nabla U\,\check\varrho^n+\check F^n\\
    &=-\divg\Theta_L^n-\nabla\check\varrho^n
      -\nabla U\,\check\varrho^n+\check F^n,
\end{aligned}
\end{equation}
where the last equality follows from \eqref{gm:eq:stress-lebesgue}.
Consequently,
\begin{equation}\label{gm:eq:acceleration-bound}
\begin{aligned}
    \int\abs{a^n}^2q^n\ud\mux
    &=\int\abs{a^n}^2\check\varrho^n\ud x\\
    &\le4\int\frac{\abs{\divg\Theta_L^n}^2
      +\abs{\nabla\check\varrho^n}^2
      +\abs{\nabla U}^2(\check\varrho^n)^2
      +\abs{\check F^n}^2}{\check\varrho^n}\ud x.
\end{aligned}
\end{equation}
The second and third terms on the right-hand side of \eqref{gm:eq:acceleration-bound} are locally bounded by
\eqref{gm:eq:marg-fisher} and \cref{gm:lem:slice-moments}. For the fourth term, by the Cauchy--Schwarz inequality, we have
\begin{equation*}
    \abs{\check F^n}^2
    \le\check\varrho^n\int\frac{\abs{F^n}^2}{\varrho^n}\ud v,
\end{equation*}
and therefore
\begin{equation*}
    \int\frac{\abs{\check F^n}^2}{\check\varrho^n}\ud x
    \le\int\frac{\abs{F^n}^2}{\varrho^n}\ud x\ud v,
\end{equation*}
which is locally bounded in time by \eqref{gm:eq:slice-action-bound}.

It remains to control the stress-divergence term in \eqref{gm:eq:acceleration-bound}. Differentiating the
centered representation \eqref{gm:eq:stress-lebesgue} and using
$\int(v-m^n)\varrho^n\ud v=0$,  we find
\begin{equation}\label{gm:eq:centred-divergence}
    \divg\Theta_L^n
    =\int\bigl[(v-m^n)\otimes(v-m^n)-I_d\bigr]\nabla_x\varrho^n\ud v .
\end{equation}
Since $\abs{m^n}\le M_{\widetilde I}$ by \cref{gm:lem:fiber-transfer},
\begin{equation*}
    \abs{(v-m^n)\otimes(v-m^n)-I_d}^2
    \le C_d(1+M_{\widetilde I}^4)(1+\abs v^4).
\end{equation*}
It follows that
\begin{equation*}
    \abs{\divg\Theta_L^n}^2
    \le C_d(1+M_{\widetilde I}^4)\check\varrho^n
    \int(1+\abs v^4)\frac{\abs{\nabla_x\varrho^n}^2}{\varrho^n}\ud v.
\end{equation*}
After division by $\check\varrho^n$ and integration in $x$, \eqref{gm:eq:weighted-fisher} yields
\begin{equation}\label{gm:eq:divTheta-bound}
    \int\frac{\abs{\divg\Theta_L^n}^2}{\check\varrho^n}\ud x
    \le C_d(1+M_{\widetilde I}^4)  \int(1+\abs v^4)\frac{\abs{\nabla_x\varrho^n}^2}{\varrho^n} \ud x\ud v<\infty,
\end{equation}
locally uniformly in $t$. Thus, for every compact
$K\subset I_{\varepsilon/2}$, the four terms on the right-hand side of
\eqref{gm:eq:acceleration-bound} are uniformly bounded for $t\in K$, and
hence
\begin{equation*}
    \sup_{t\in K}\int\abs{a^n_t}^2q^n_t\ud\mux<\infty.
\end{equation*}
Together with the pointwise continuity equation and the locally uniform
$L^\infty$ bound on $m^n_t$ established above, this is precisely
\textup{(A2)}. Since $q^n_t\mux$ is a probability measure, the latter bound
also implies
\begin{equation*}
    \sup_{t\in K}\int\abs{m^n_t}^2q^n_t\ud\mux
    \le M_{\widetilde I}^2,
\end{equation*}
as an immediate consequence. This completes the proof.
\end{proof}

\subsubsection{The entropy dissipation identity}

\begin{lemma}[Entropy dissipation identity]\label{gm:lem:entropy-identity}
On $I_{\varepsilon/2}$, $\Ent_\mu(g^n_t)$ and
$\Iv(g^n_t)$ are locally bounded, the map
$t\mapsto\Ent_\mu(g^n_t)$ is locally absolutely continuous, and
\begin{equation}\label{gm:eq:entropy-identity}
    \frac{\ud}{\ud t}\Ent_\mu(g^n_t)
    =\int W^n_t\cdot\nabla_v\log g^n_t\ud\mu
    \qquad\text{for \textup{a.e.}\ }t\in I_{\varepsilon/2}.
\end{equation}
\end{lemma}

\begin{proof}
For a compact interval $K\subset I_{\varepsilon/2}$, by \eqref{gm:eq:functionals-lebesgue}, the slice bounds
\eqref{gm:eq:slice-action-bound} and \eqref{gm:eq:slice-fisher-bound} give
\begin{equation}\label{gm:eq:entropy-proof-slice-bounds}
    \sup_{t\in K}\left[
        \Iv(g^n_t)
        +\int\frac{\abs{W^n_t}^2}{g^n_t}\ud\mu
    \right]<\infty.
\end{equation}
Moreover, the entropy decomposition, the marginal entropy bound in
\cref{gm:prop:A1}, and the fiberwise Gaussian LSI imply
\begin{equation*}
    \Ent_\mu(g^n_t)
    =\Ent_{\mux}(q^n_t)+\Entv(g^n_t)
    \le\Ent_{\mux}(q^n_t)+\frac12\Iv(g^n_t).
\end{equation*}
Thus both $\Ent_\mu(g^n_t)$ and $\Iv(g^n_t)$ are uniformly bounded for
$t\in K$.

By \cref{gm:lem:smoothness}, the pair $(g^n,W^n)$ satisfies the regularity for \cref{lem:entropy-chain-rule}, and the equation \eqref{gm:eq:final-graph-equation}. The local integrability assumptions of that lemma follow from \eqref{gm:eq:entropy-proof-slice-bounds} and the entropy bound above. Therefore, \cref{lem:entropy-chain-rule} gives the local absolute continuity of $t\mapsto\Ent_\mu(g^n_t)$ and the identity \eqref{gm:eq:entropy-identity}. Since $K\subset I_{\varepsilon/2}$ was arbitrary, the proof is complete.
\end{proof}

\subsubsection{The corrector inequality}

We now restrict to $I_\varepsilon$. For each fixed time slice, let $\xi_{q^n_t}:=x-T_{q^n_t}$ be the Brenier displacement from $q^n_t\mux$ to $\mux$. It is well defined because $q^n_t\mux$ has finite second moment by \cref{gm:lem:slice-moments}. Moreover, Talagrand's inequality gives
\begin{equation}\label{gm:eq:talagrand-slice}
    \int q^n_t\abs{\xi_{q^n_t}}^2\ud\mux  =W_2^2(q^n_t\mux,\mux) \le\frac2\rho\Ent_{\mux}(q^n_t)<\infty,
\end{equation}
locally uniformly in $t$ by \cref{gm:prop:A1}.

We also verify the remaining regularity needed for \cref{lem:localized-stress}. The zeroth, first, and second velocity moments of the $v$-convolution are, respectively, the corresponding moments before convolution, with the additional term $b^2\check\varrho^n I_d$ in the second moment. They are therefore Gaussian $x$-convolutions of $L^1$ moment fields and are smooth in $x$. The floor adds only smooth equilibrium moments. Thus $\check\varrho^n$, $\check\jmath^n$, and $S_L^n$ are smooth; since $\check\varrho^n\ge\delta\varrho_{\infty,x}>0$, the quotient formula \eqref{gm:eq:stress-lebesgue} gives
\begin{equation*}
    \Theta^n\in C^1(\RR^d;\RR^{d\times d})
\end{equation*}
for every time slice. Together with \cref{gm:prop:A1}, the Gaussian LSI, and \cref{gm:lem:fourth-transfer}, all assumptions of
\cref{lem:localized-stress} are satisfied. Importantly, that lemma is applied below at each fixed cutoff; no global cutoff-limit hypothesis is being assumed.

\begin{proposition}[Corrector inequality for the approximants]\label{gm:prop:corrector}
Let $\Gamma>0$ and $\gamma=\Gamma\sqrt\rho$. Then the approximant
$(g^n,W^n)$ satisfies \eqref{gm:eq:weak-corrector} in
$\mathcal D'(I_\varepsilon)$.
\end{proposition}

\begin{proof}
Setting
\begin{equation*}
    \widetilde B^n_t:=\Piv(W^n_t)    =\frac{\check F^n_t}{\varrho_{\infty,x}},
\end{equation*}
we have, for every $t\in I_\varepsilon$,
\begin{equation}\label{gm:eq:B-slice-bound}
    \int\frac{\abs{\widetilde B^n_t}^2}{q^n_t}\ud\mux \le\int\frac{\abs{W^n_t}^2}{g^n_t}\ud\mu.
\end{equation}
It follows from \eqref{gm:eq:talagrand-slice} that
\begin{equation}\label{gm:eq:B-xi-slice-bound}
    \Abs{\int \widetilde B^n_t\cdot\xi_{q^n_t}\ud\mux}
    \le
    \left(\int\frac{\abs{W^n_t}^2}{g^n_t}\ud\mu\right)^{1/2}
    \left(\frac{2}{\rho}\Ent_{\mux}(q^n_t)\right)^{1/2}.
\end{equation}

\medskip
\noindent \emph{Step 1: acceleration inequality.}
By \cref{gm:prop:A2}, \cref{lem:wasserstein-acceleration} applies to
$\nu^n_t=q^n_t\mux$ and gives, in $\mathcal D'(I_\varepsilon)$,
\begin{equation}\label{gm:eq:cor-step1}
    \frac{\ud}{\ud t}\,\mathcal C_{\rm OT}(g^n_t)
    \le J(g^n_t)+\int\xi_{q^n_t}\cdot a^n_tq^n_t\ud\mux,
    \qquad
    J(g^n_t):=\int\frac{\abs{j^n_t}^2}{q^n_t}\ud\mux .
\end{equation}
The two terms on the right-hand side are locally integrable in time:
$J(g^n_t)$ is locally bounded by \cref{gm:prop:A1}, while
Cauchy--Schwarz, \cref{gm:prop:A2}, and
\eqref{gm:eq:talagrand-slice} give
\begin{equation*}
    \Abs{\int\xi_{q^n_t}\cdot a^n_tq^n_t\ud\mux}
    \le    \left(\int\abs{a^n_t}^2q^n_t\ud\mux\right)^{1/2}    \left(\int\abs{\xi_{q^n_t}}^2q^n_t\ud\mux\right)^{1/2},
\end{equation*}
which is locally bounded. Moreover,
\eqref{est:corrector} and \cref{gm:lem:entropy-identity} show that
$t\mapsto\mathcal C_{\rm OT}(g^n_t)$ is locally bounded.

\medskip
\noindent
\emph{Step 2: localized estimate of the acceleration pairing.} Dividing
\eqref{gm:eq:acceleration-identity} by $\varrho_{\infty,x}$ and using
\eqref{gm:eq:adjoint-divergence} gives
\begin{equation}\label{gm:eq:cor-accel-relative}
    q^n_ta^n_t
    =-\nabla_xq^n_t+\nabla_x^{*}\Theta^n_t+\widetilde B^n_t .
\end{equation}
Since $a^n_t,\xi_{q^n_t}\in L^2(q^n_t\mux)$, we have
$\xi_{q^n_t}\cdot a^n_tq^n_t\in L^1(\mux)$. Choose
$\chi\in C_c^\infty(\RR^d)$ with
$0\le\chi\le1$, $\chi=1$ on the unit ball, and
$\abs{\nabla\chi}\le C$, and set $\chi_R(x):=\chi(x/R)$. Dominated
convergence yields
\begin{equation}\label{gm:eq:cor-dominated}
    \int\xi_{q^n_t}\cdot a^n_tq^n_t\ud\mux
    =\lim_{R\to\infty}\int\chi_R\xi_{q^n_t}\cdot a^n_tq^n_t\ud\mux .
\end{equation}
Using $\chi_R\nabla_x^*\Theta^n_t
    =\nabla_x^*(\chi_R\Theta^n_t)+\Theta^n_t\nabla\chi_R$,
\eqref{gm:eq:cor-accel-relative} gives
\begin{equation}\label{gm:eq:cor-expansion}
\begin{aligned}
    \int\chi_R\,\xi_{q^n_t}\cdot q^n_ta^n_t\ud\mux
    &=\left[
      \int\xi_{q^n_t}\cdot\nabla_x^*(\chi_R\Theta^n_t)\ud\mux
      -\int\chi_R\nabla_xq^n_t\cdot\xi_{q^n_t}\ud\mux
      \right]\\
    &\quad+\int\xi_{q^n_t}\cdot\Theta^n_t\nabla\chi_R\ud\mux
      +\int\chi_R\xi_{q^n_t}\cdot\widetilde B^n_t\ud\mux.
\end{aligned}
\end{equation}

Applying \cref{lem:localized-stress} to the first part of the right-hand side of
\eqref{gm:eq:cor-expansion}, with $\eta=1/4$, yields
\begin{equation}\label{gm:eq:cor-localized-acceleration}
\begin{aligned}
    \int\chi_R\,\xi_{q^n_t}\cdot q^n_ta^n_t\ud\mux
    &\le 4\Entv(g^n_t)
      -\int\chi_Rq^n_t\log q^n_t\ud\mux
      +\int q^n_t\nabla\chi_R\cdot\xi_{q^n_t}\ud\mux\\
    &\quad+\int\xi_{q^n_t}\cdot\Theta^n_t\nabla\chi_R\ud\mux
      +\int\chi_R\xi_{q^n_t}\cdot\widetilde B^n_t\ud\mux.
\end{aligned}
\end{equation}

\medskip
\noindent
\emph{Step 3: removal of the cutoff.}
The source term in \eqref{gm:eq:cor-localized-acceleration} converges by \eqref{gm:eq:B-xi-slice-bound} and dominated
convergence:
\begin{equation*}
    \int\chi_R\xi_{q^n_t}\cdot\widetilde B^n_t\ud\mux
    \to \int\xi_{q^n_t}\cdot\widetilde B^n_t\ud\mux.
\end{equation*}
Moreover, by \eqref{gm:eq:talagrand-slice} and \cref{gm:lem:stress-L2},
\begin{equation*}
\Abs{\int\xi_{q^n_t}\cdot\Theta^n_t\,\nabla\chi_R\ud\mux}
    \le\frac CR
    \left(\int q^n_t\abs{\xi_{q^n_t}}^2\ud\mux\right)^{1/2}
    \left(\int\frac{\abs{\Theta^n_t}^2}{q^n_t}\ud\mux\right)^{1/2}
    \to 0.
\end{equation*}
Similarly, using $\int q^n_t\ud\mux=1$,
\begin{equation*}
    \Abs{\int q^n_t\nabla\chi_R\cdot\xi_{q^n_t}\ud\mux}
    \le\frac CR\left(\int q^n_t\abs{\xi_{q^n_t}}^2\ud\mux\right)^{1/2}
    \to 0.
\end{equation*}
Moreover, the proof of \cref{gm:prop:A1} shows that
$q^n_t\abs{\log q^n_t}\in L^1(\mux)$. Hence
\begin{equation*}
    \int\chi_Rq^n_t\log q^n_t\ud\mux
    \to \Ent_{\mux}(q^n_t).
\end{equation*}
Letting $R\to\infty$ in \eqref{gm:eq:cor-localized-acceleration} and using
\eqref{gm:eq:cor-dominated} therefore gives
\begin{equation} \label{auxeq:slicebound}
    \int\xi_{q^n_t}\cdot a^n_tq^n_t\ud\mux
    \le4\,\Entv(g^n_t)-\Ent_{\mux}(q^n_t)
      +\int\xi_{q^n_t}\cdot\widetilde B^n_t\ud\mux .
\end{equation}

\medskip
\noindent
\emph{Step 4: distributional assembly.}
The entropy terms in \eqref{auxeq:slicebound} are locally bounded by \cref{gm:prop:A1,gm:lem:entropy-identity}. Its source pairing is locally
bounded by \eqref{gm:eq:B-xi-slice-bound}, \eqref{gm:eq:slice-action-bound},
and \cref{gm:prop:A1}. We may therefore substitute this inequality into
\eqref{gm:eq:cor-step1}. By
\eqref{eq:imported-size} and the Gaussian LSI \eqref{eq:gausslsi},
\begin{equation*}
    J(g^n_t)+4\Entv(g^n_t)
    \le6\Entv(g^n_t)
    \le3\Iv(g^n_t).
\end{equation*}
It follows from \eqref{gm:eq:cor-step1} that, in $\mathcal D'(I_\varepsilon)$,
\begin{equation} \label{auxeq:corrector}
    \frac{\ud}{\ud t}\,\mathcal C_{\rm OT}(g^n_t)
    \le-\Ent_{\mux}(q^n_t)+3\Iv(g^n_t)
      +\int\widetilde B^n_t\cdot\xi_{q^n_t}\ud\mux.
\end{equation}
Finally, recall from \ref{gm:R3} that
\begin{equation*}
    B^n_t:=\Piv(W^n_t)+\gamma j^n_t
    =\widetilde B^n_t+\gamma j^n_t
\end{equation*}
and that $\mathcal C_{\rm OT}(g^n_t)=\int j^n_t\cdot\xi_{q^n_t}\ud\mux$.
Therefore,
\begin{equation*}
    \int\widetilde B^n_t\cdot\xi_{q^n_t}\ud\mux
    =\int B^n_t\cdot\xi_{q^n_t}\ud\mux
      -\gamma\mathcal C_{\rm OT}(g^n_t),
\end{equation*}
and the above inequality \eqref{auxeq:corrector} is precisely \eqref{gm:eq:weak-corrector}.
\end{proof}

We are now ready to finish the proof of \cref{gm:thm:recovery}.

\begin{proof}[Proof of \cref{gm:thm:recovery}]
We choose sequences $\tau_n,a_n,\delta_n\downarrow0$ such that
$0<\tau_n<\varepsilon/2$ and $0<a_n,\delta_n<1$, and set $b_n:=a_n$ such that $b_n^2/a_n = a_n \to 0$. Let $(g^n,W^n)$ be the density-current pairs constructed in \cref{gm:ssec:floor} with parameters $(\tau_n,a_n,b_n,\delta_n)$. By \cref{gm:lem:smoothness}, these pairs have the regularity stated in \ref{gm:R1}.  The divergence equation \eqref{gm:eq:final-graph-equation} holds pointwise on $I_{\varepsilon/2}$ by \cref{gm:lem:smoothness}, and hence in particular on $I_\varepsilon$. The local absolute continuity and entropy dissipation identity, as well as the local boundedness of the entropy and velocity Fisher information, follow from \cref{gm:lem:entropy-identity}. This proves \ref{gm:R2}. The controlled corrector inequality \ref{gm:R3} is precisely \cref{gm:prop:corrector}.  Finally, \cref{gm:prop:L1} gives the convergence in \ref{gm:R4}. Since $a_n,b_n\to0$ and $b_n^2/a_n\to0$, \cref{gm:prop:limsup} gives both limsup bounds in \ref{gm:R5}. This completes the proof.
\end{proof}

\begin{remark}\label{gm:rem:hypotheses}
We explain here the roles of the additional assumptions \ref{gm:Q1}--\ref{gm:Q3} in the recovery theorem \cref{gm:thm:recovery}. Condition \ref{gm:Q1} provides, through \cref{gm:lem:fiber-transfer}, a pointwise bound on the conditional mean $m$ that is locally uniform in time. This verifies the bounded-velocity hypothesis of \cref{lem:wasserstein-acceleration} and controls the $\abs{m}^4$ term in \eqref{gm:eq:divTheta-bound}. Condition \ref{gm:Q2} is used to establish the $(1+\abs{v}^4)$-weighted spatial Fisher estimate \eqref{gm:eq:weighted-fisher}. Through \eqref{gm:eq:divTheta-bound}, this
controls the stress-divergence term in the acceleration estimate \eqref{gm:eq:acceleration-bound} and yields the required $L^2(q\mux)$ bound on the acceleration. The same fourth-moment assumption also gives the stress bound \eqref{gm:eq:stress-L2}, which is used to remove the spatial cutoff in the proof of \cref{gm:prop:corrector}. Therefore, \ref{gm:Q1} and \ref{gm:Q2} enter only the verification of the controlled corrector inequality \ref{gm:R3}. Condition \ref{gm:Q3} ensures the integrability needed for the current regularity in \ref{gm:R1}. It also controls the force-commutator term in \eqref{gm:eq:force-commutator-action}, and hence both the local action bound \eqref{gm:eq:slice-action-bound} used with \cref{lem:entropy-chain-rule} to prove \ref{gm:R2} and the action limsup bound in \ref{gm:R5}. Finally, it controls the $\abs{\nabla U}^2$ term in the acceleration estimate \eqref{gm:eq:acceleration-bound} used for \ref{gm:R3}.

Moreover, condition \ref{gm:Q3} is implied by \ref{gm:P1} under a stronger gradient bound on the potential:
\begin{equation} \label{eq:gradidentpotential}
    \abs{\nabla U}^2\le C_0(1+U).
\end{equation}
Indeed, choose $\lambda>0$ such that $\lambda C_0<1$. By \cref{gm:lem:exp-integrability}(ii),
\begin{equation*}
    \int \e^{\lambda\abs{\nabla U}^2}\ud\mux
    \le
    \e^{\lambda C_0}
    \int \e^{\lambda C_0U}\ud\mux
    <\infty.
\end{equation*}
Applying the entropy variational inequality \eqref{eq:variational} with $\phi=\lambda\abs{\nabla U}^2$ and the density $g_t$ gives, for \textup{a.e.}\ $t$,
\begin{equation*}
    \int\abs{\nabla U}^2g_t\ud\mu
    \le
    \frac{1}{\lambda}
    \left(
        \Ent_\mu(g_t)
        +
        \log\int\e^{\lambda\abs{\nabla U}^2}\ud\mux
    \right).
\end{equation*}
Integrating in time and using \ref{gm:P1}, together with
\begin{equation*}
    (1+\abs{\nabla U})^2
    \le 2(1+\abs{\nabla U}^2),
\end{equation*}
yields \ref{gm:Q3}. For higher-order polynomial potentials in the H\'erau--Nier class \cites{HerauNier2004,Lu2026}, $\abs{\nabla U}^2$ generally grows faster than $U$, so the preceding gradient condition \eqref{eq:gradidentpotential} fails and \ref{gm:P1} does not yield \ref{gm:Q3}.

Thus, conditions \ref{gm:Q1}--\ref{gm:Q3} are technical requirements of the present recovery argument rather than intrinsic components of the finite-action cost. Whether they can be weakened or removed from the finite-action space-time LSI remains open. One possible route to weakening \ref{gm:Q2} is to replace the Brenier corrector by an entropically regularized transport corrector. The additional smoothness of the latter may allow the acceleration pairing to be controlled under a weighted $L^1$ assumption. Making this approach rigorous would require a localized stress estimate uniform in the regularization parameter, together with a justified passage to the
zero-regularization limit. Such an argument could potentially remove \ref{gm:Q2}, and possibly also
\ref{gm:Q1}, from \cref{gm:thm:recovery}.
\end{remark}

\end{document}